\documentclass[12 pt,reqno]{amsart}
\usepackage{pdfpages}
\usepackage{amsmath,amsthm,amssymb,amscd}
\usepackage{setspace}
\usepackage{upgreek}
\usepackage[utf8]{inputenc}
\usepackage{cite}
\usepackage{url}
\usepackage{mathtools}
\usepackage{fullpage}
\usepackage{float}
\usepackage{comment}
\usepackage{enumitem}
\usepackage{tikz-cd}
\usepackage{rotating}
\usepackage{color, soul}
\usepackage{nicefrac,xfrac}
\usepackage{diagbox}
\usepackage{quiver}

\usepackage[bookmarksopen,bookmarksdepth=3]{hyperref}

\setlist[enumerate,1]{label={(\alph*)}}

\input xy
\xyoption{all}

\newtheorem{tm}{Theorem}
\newtheorem{cj}[tm]{Conjecture}
\newtheorem{lem}[tm]{Lemma}
\newtheorem{pr}[tm]{Proposition}
\newtheorem{crl}[tm]{Corollary}

\newtheorem{thm}{Theorem}[section]
\newtheorem{prop}[thm]{Proposition}
\newtheorem{lm}[thm]{Lemma}
\newtheorem{cor}[thm]{Corollary}

\theoremstyle{definition}
\newtheorem{df}[tm]{Definition}

\newtheorem{dfn}[thm]{Definition}

\theoremstyle{remark}
\newtheorem{rmk}[tm]{Remark}

\newtheorem{rem}[thm]{Remark}

\newcommand{\interior}[1]{%
  {\kern0pt#1}^{\mathrm{o}}%
}
\newcommand{\R}{\mathbb{R}}
\newcommand{\Z}{\mathbb{Z}}
\newcommand{\F}{\mathbb{F}}

\newcommand{\bcc}{\widetilde{c}}
\newcommand{\e}{e}
\newcommand{\bcb}{\widetilde{b}}
\newcommand{\bcm}{\widetilde{m}}

\newcommand{\mD}{\mathcal{D}}

\newcommand{\mA}{\mathcal{A}}
\newcommand{\mQ}{\mathcal{Q}}
\newcommand{\mG}{\mathcal{G}}

\newcommand{\mR}{\mathcal{R}}

\newcommand{\mM}{\mathcal{M}}

\newcommand{\fRs}{\mathfrak{R}^s}
\newcommand{\fU}{\mathfrak{A}}
\newcommand{\fS}{\mathfrak{S}}
\newcommand{\fUs}{\mathfrak{A}^s}

\newcommand{\fB}{\mathfrak{B}}

\newcommand{\fR}{\mathfrak{R}}

\newcommand{\fpar}{\mathfrak{d}}

\newcommand{\eval}
{\mathfrak{eval}}
\newcommand{\C}{\mathbb{C}}
\newcommand{\zzs}{\mathbb{Z}_{m,n+1}^\sigma}
\newcommand{\zz}{\mathbb{Z}_{m,l}}
\newcommand{\K}{\mathbb{K}}

\newcommand{\fCC}{\widecheck{\CC}}
\newcommand{\fmR}{\widecheck{\mathcal{R}}}
\newcommand{\fmQ}{\widecheck{\mathcal{Q}}}

\newcommand{\s}{\mathfrak{s}}
\newcommand{\deltr}{\widetilde{\triangle}}

\newcommand{\oogw}{\text{OGW}}

\newcommand{\dcc}{d_\text{cc}}

\newcommand{\cp}{\mathbb{C}P}
\newcommand{\Q}{\mathbb{Q}}
\newcommand{\rp}{\mathbb{R}P}

\newcommand{\Mms}{M_n^s}
\newcommand{\Ms}{M^s}
\newcommand{\TT}{\overline{T}}
\newcommand{\Rs}{R^s}

\newcommand{\CC}{C}
\newcommand{\HH}{H}
\newcommand{\Rtr}{R_\triangle}
\newcommand{\Str}{S_\triangle}
\newcommand{\Mtr}{M_\triangle}
\newcommand{\Wtr}{W_\triangle}
\newcommand{\Msf}{M_n}
\newcommand{\Ssf}{S_n}
\newcommand{\Dsf}{D_n}
\newcommand{\Rsf}{R_n}

\newcommand{\ems}{\End(M)^s}
\newcommand{\Ems}{\End(M_n)^s}
\newcommand{\EEms}{\End(M_n)^s}
\newcommand{\eems}{\End(M)^s}
\newcommand{\Rms}{R_n^s}
\newcommand{\Sms}{S_n^s}
\newcommand{\ccl}{\mathrm{Cl}(\triangle)}
\newcommand{\Mmm}{\widehat{\mM}}
\newcommand{\Wsf}{W_n}
\newcommand{\gto}{\tau}
\newcommand{\simp}{\deltr^n}
\newcommand{\Cl}{\mathrm{Cl}(\triangle^n)}

\newcommand{\ccle}{\mathrm{Cl}^+_{\mathrm{even}}(\triangle^n)}
\newcommand{\thet}{\widetilde{\tau}}
\newcommand{\trian}{\triangle^n}

\newcommand{\expp}{\widetilde{\exp}}
\newcommand{\bgc}{\sigma}
\newcommand{\Zt}{{\Z/2}}
\newcommand{\Wtrb}{{W_{\triangle,\bgc}}}
\newcommand{\Spt}{\mathrm{Sp}_{\triangle,\bgc}}
\newcommand{\Strb}{{S_{\triangle,\bgc}}}

\newcommand{\wtr}{{w_{\triangle}}}
\newcommand{\Dtr}{{D_\triangle}}
\newcommand{\Xtr}{{X_\triangle}}
\newcommand{\Xtrr}{X_\triangle^\R}
\newcommand{\wtrb}{{w_{\triangle,\bgc}}}
\newcommand{\Mtrb}{{M_{\triangle,\bgc}}}
\newcommand{\Dtrb}{{D_{\triangle,\bgc}}}
\newcommand{\Ltrb}{{L_{\triangle,\bgc}}}
\newcommand{\cclb}{{\mathrm{Cl}(\triangle,\bgc)}}

\DeclareMathOperator{\ima}{Im}
\DeclareMathOperator{\id}{Id}

\DeclareMathOperator{\Fuk}{Fuk}
\DeclareMathOperator{\Coh}{Coh}
\DeclareMathOperator{\rank}{rank}
\DeclareMathOperator{\MF}{MF}

\DeclareMathOperator{\End}{End}

\DeclareMathOperator{\Hom}{Hom}

\DeclareFontFamily{U}{mathx}{}
\DeclareFontShape{U}{mathx}{m}{n}{<-> mathx10}{}
\DeclareSymbolFont{mathx}{U}{mathx}{m}{n}
\DeclareMathAccent{\widehat}{0}{mathx}{"70}
\DeclareMathAccent{\widecheck}{0}{mathx}{"71}

\title{Numerical invariants of normed matrix factorizations}
\author[M. Sela]{May Sela}
\address{Institute of Mathematics, Hebrew University}
\email{mayysela@gmail.com}
\author[J. Solomon]{Jake P. Solomon}
\address{Institute of Mathematics, Hebrew University}
\email{jake@math.huji.ac.il}
\begin{document}

\subjclass[2020]{53D37,14J33}
\keywords{mirror symmetry, open Gromov-Witten invariants, Welschinger invariants, matrix factorization, non-Archimedean norm, Calabi-Yau structure}
\date{December 2024}
\begin{abstract}
We define a normed matrix factorization category and a notion of bounding cochains for objects of this category.
We classify bounding cochains up to gauge equivalence for spherical objects and use this classification to define numerical invariants. 
These invariants are expected to correspond under mirror symmetry to
the open Gromov-Witten invariants with only boundary constraints of Lagrangian rational cohomology spheres defined by the second author and Tukachinsky.

For each Delzant polytope, we construct a normed matrix factorization category. For Delzant polytopes satisfying a combinatorial relative spin condition, we construct an object of this category called the Dirac factorization. The Dirac factorization is expected to correspond under mirror symmetry to the Lagrangian submanifold given by the real locus of the toric symplectic manifold associated to the Delzant polytope. In the case of the $n$-simplex for $n$ odd, we show that the Dirac factorization is spherical, mirroring the fact that $\R P^n$ is a rational cohomology sphere. For $n = 1,$ we show the numerical invariants of the Dirac factorization coincide with the open Gromov-Witten invariants of $\R P^1 \subset \C P^1.$ For $n = 3$ in low degrees, computer calculations verify that the numerical invariants of the Dirac factorization coincide with the open Gromov-Witten-Welschinger invariants of $\R P^3 \subset \C P^3.$  
Although $\R P^n$ is trivial in the Fukaya category of $\C P^n$ over any field of characteristic zero, the above results can be seen as a manifestation of mirror symmetry over a Novikov ring.

\end{abstract}
\maketitle
\tableofcontents
\section{Introduction}
Our starting point is the question of how to recover from mirror symmetry the open Gromov-Witten invariants~\cite{Cho,Georgieva,Solomon,Point_like_bc} of a Lagrangian submanifold $L \subset X,$ or the Welschinger invariants~\cite{Welschinger2,Welschinger} of a real symplectic manifold $X$, in the case $X$ is a Fano manifold. This question has remained open for over fifteen years since the Calabi-Yau case of $X$ the quintic threefold and $L$ its real locus was treated in~\cite{Morrison_Walcher,PSW,Walcher}. For example, consider the case $(X,L) = (\C P^3, \R P^3).$ A major obstacle arises from the fact that $\R P^3$ is the trivial object in the Fukaya category of $\C P^3$ over any field of characteristic zero. So, one does not expect to find a meaningful mirror for $\R P^3$ in the usual matrix factorization category that is equivalent to the Fukaya category of $\C P^3$ according to homological mirror symmetry. On the other hand, $\R P^3$ has non-trivial Floer cohomology over an appropriate Novikov ring. This leads us to consider an enhanced version of the matrix factorization category where objects and morphism complexes come equipped with non-Archimedean norms. We call this the normed matrix factorization category. Taking unit balls with respect to the norms, one obtains a category over a Novikov ring. Thus one can hope to find normed matrix factorizations mirror to the relevant Lagrangian submanifolds. We define numerical invariants of normed matrix factorizations and present initial evidence for a mirror correspondence with open Gromov-Witten and Welschinger invariants. The chain level Calabi-Yau structure on the category of matrix factorizations plays a key role in the definition and computation of invariants.

\subsection{Background and overview}

Matrix factorizations were first introduced by Eisenbud~\cite{eisenbud} in order to study maximal Cohen-Macaulay modules. Later, the category of matrix factorizations played a role in Kontsevich's
homological mirror symmetry conjecture~\cite{kontsevich1994homological}.
Roughly speaking, mirror symmetry is a duality between the symplectic geometry of a Calabi-Yau manifold $X$ and the complex geometry of a mirror Calabi-Yau manifold $X^\vee.$ 
One way to understand this duality is through the homological mirror symmetry conjecture~\cite{kontsevich1994homological}, which asserts an equivalence of categories
\[
D^\pi\Fuk(X)\cong D^b\Coh(X^\vee), \qquad D^\pi\Fuk(X^\vee)\cong D^b\Coh(X).
\] 
Mirror symmetry has been extended to some families of non-Calabi-Yau manifolds as well. In particular, mirror symmetry for Fano manifolds has been studied extensively. A K\"ahler manifold  is called Fano if its anti-canonical bundle is ample. 
In this case, the mirror of $X$ is given by a Landau-Ginzburg model $(X^\vee, W),$ consisting a Calabi-Yau manifold $X^\vee$ and a holomorphic function $W$ called the superpotential. 
For more details see~\cite{auroux2007mirror}. 

The homological mirror symmetry conjecture extends to the Fano case. It asserts~\cite{Kapustin_2003} that the derived Fukaya category $D^\pi\Fuk(X)$ is equivalent to the derived category of singularities $D^b_{\mathrm{sing}}(X^\vee)$. 
In~\cite{orlov2004triangulated}, Orlov showed that $D^b_{\mathrm{sing}}(X^\vee)\cong H^0\MF(W),$ where $\MF(W)$ is the matrix factorization category.
Hence, the homological mirror symmetry conjecture can be reformulated as $D^\pi\Fuk(X)\cong H^0\MF(W).$
In this paper, we will focus on the algebraic side of the mirror, that is the matrix factorization category associated to a Landau-Ginzburg model. 

We define the notion of a normed matrix factorization category $\MF(W).$ 
This definition refines the one given by Orlov~\cite{orlov2004triangulated}. In our definition, the objects of the category are $\Z$-graded modules equipped with a family of valuations, the sets of morphisms are valued DG modules, and we require the category to be equipped with a normed $\infty$-trace.
The latter is an operator defined on the cyclic complex of $\MF(W).$ 
It refines the notion of a Calabi-Yau structure~\cite{KS} by requiring 
it to satisfy a certain compatibility condition with the valuations on the sets of morphisms. 

Building on the general framework of deformation theory, Fukaya-Oh-Ohta-Ono introduced the notion of bounding cochains to address the problem of obstructions in Lagrangian Floer theory~\cite{FOOO1}. 
A bounding cochain is a solution to Maurer-Cartan equation in the Fukaya $A_\infty$-algebra of a Lagrangian submanifold satisfying conditions involving the relevant grading and valuation. In~\cite{Point_like_bc}, the second author and Tukachinsky used bounding cochains in order to define open Gromov-Witten invariants of $(X,L)$ where the Lagrangian submanifold $L\subset X$ is a rational cohomology sphere.
These are numerical invariants of the symplectic manifold $X$ and the Lagrangian submanifold $L$ that are preserved under deformations of the symplectic form.
The bounding cochains have two geometric roles in the definition of open Gromov-Witten invariants. 
The first role is to cancel 
the bubbling of $J$-holomorphic disks,
and the second is to constrain the boundary of $J$-holomorphic disks. 
There is a natural equivalence relation on bounding cochains known as gauge equivalence.
The main result of~\cite{Point_like_bc} is a classification of bounding cochains up to gauge equivalence for Lagrangian rational cohomology spheres $L$ in terms of the integral over $L$. This classification determines a canonical bounding cochain, called point-like, that mimics the constraint of passing through a given point on $L.$ Point-like bounding cochains are universal in the sense that open Gromov-Witten invariants extracted from a generic bounding cochain can also obtained from the point-like one~\cite{relative_quantum_co}.  

The main goal of this paper is to define a bounding cochain for an object of the category of normed matrix factorizations, and to classify the bounding cochains of objects that correspond under mirror symmetry to Lagrangian rational cohomology spheres.
These objects are the normed Calabi-Yau spherical objects. The normed $\infty$-trace on the normed matrix factorization category plays the role of the integral over the Lagrangian submanifold in the classification of bounding cochains.

Given a bounding cochain $b,$ the Maurer-Cartan equation associates to $b$ an element of the form 
\[
c=\sum_{k\ge0} \sum_{ d\in G} N_{d,k+1}\frac{s^k}{k!}T^d,\qquad N_{d,k+1}\in \R,
\]
where $s,$ $ T$ are formal variables and $G$ is an abelian group. We call $c$ the potential of $b.$  
Using the normed $\infty$-trace on $\MF(W,w),$ we define a notion for a matrix factorization analogous
to a point-like bounding cochain in a Fukaya $A_\infty$-algebra.
Let $b_{pt}$ be a point-like bounding cochain of a normed Calabi-Yau spherical object, and let $c_{pt}$ be its potential.
We define numerical invariants of normed Calabi-Yau spherical object to be the numbers $N_{d,k}$ induced by $c_{pt}.$
The classification of bounding cochains shows that the invariants $N_{d,k}$ are independent of all choices.
The ``mirror" of the definition of these numerical invariants yields the open Gromov-Witten invariants with only boundary constraints of a spherical object in a Fukaya $A_\infty$-algebra.

An important family of examples of Fano manifolds arises from toric geometry. For each Delzant polytope $\triangle$, there is an associated toric manifold by the work of Delzant~\cite{Delzant}.
For each $\triangle,$ we construct a normed matrix factorization category by refining the construction of Hori and Vafa~\cite{Hori_Vafa}. When the Delzant polytope satisfies a combinatorial relative spin condition, we construct an object of the matrix factorization category called the Dirac factorization. The Dirac factorization is expected to correspond under mirror symmetry to the real locus of the toric manifold associated to $\triangle.$
We illustrate this construction with the example of $\cp^n$ where $n$ is odd, which has Delzant polytope the $n$-simplex $\triangle^n$. 
In~\cite{Shklyarov}, Shklyarov defines a Calabi-Yau structure $\Theta$ which supplements the Kapustin-Li formula~\cite{kapustin2003topological} with higher order correction terms. 
We prove that $\Theta$ is a normed $\infty$-trace on the normed matrix factorization category $\MF(\Wsf)$ associated to $\triangle^n.$ 
We show that the Dirac factorization object $\Msf$ of $\MF(\Wsf)$ is spherical, which corresponds by mirror symmetry to the fact that $\R P^n$ is a rational cohomology sphere. 
In the case of $n=1,$ we compute the numerical invariants $N_{d,k}$ of $\Msf$. These invariants coincide with the open Gromov-Witten invariants of the Lagrangian submanifold $\rp^1 \subset \cp^1.$ In the case of $n=3,$ based on computer calculations we show the numerical invariants of $\Msf$ in low degrees coincide with the open Gromov-Witten invariants of $\rp^3 \subset \cp^3,$ which in turn coincide with the real enumerative invariants of Welschinger~\cite{Welschinger} up to a factor of $2$.

For $n$ odd, the Lagrangian submanifold $\rp^n \subset \cp^n$ is trivial in the Fukaya category of $\cp^n$ over any field of characteristic zero. Thus, one doesn't expect to find an interesting mirror in the category of matrix factorizations. However, in the normed matrix factorization category, one can take unit balls to obtain a category over a Novikov ring. In the Fukaya category over this Novikov ring, $\rp^n$ is no longer trivial.

\subsection{Statement of results}
For a graded ring $\Upsilon,$ we denote by $(\Upsilon)^i$ the set of homogeneous elements of degree $i$ in $\Upsilon.$ Let $\K$ be a field of characteristic zero.

\subsubsection{Normed DG category}\label{section:1.2.1}

We will mainly focus on algebraic objects equipped with valuations and norms. 
In Section~\ref{Section:valuation} we recall the notions of valued and normed  rings, modules and algebras.

Below all the tensor products and direct sums will be implicitly completed with respect to the relevant valuation or norm. In addition, all the tensor products will be the graded tensor products.
\begin{df}
Let $(R,d_R, \nu_R)$ be a valued DGA over a field $\K$. We say that $\mA$ is
a \textbf{normed DG category} \textbf{over} $R$ 
if:
\begin{enumerate}[label=(\arabic*)]
    \item for all objects $X^1,X^2,$ the morphism set $\Hom(X^1,X^2)$ is equipped with a valuation $\nu_1$ such that $(\Hom(X^1,X^2),d,\nu_1)$ is a valued DG module over $R$ satisfying $\nu_1\big(d(f)\big)\ge\nu_1(f),$ for all $f\in\Hom(X^1,X^2),$
    \item for all objects $X^1, X^2, X^3$ and $f\in (\Hom(X^1,X^2), \nu_1),$   $g\in (\Hom(X^2,X^3),\nu_2),$ the composition $g\circ f\in (\Hom(X^1,X^3),\nu_3)$ satisfies
    \[
    \nu_3(g\circ f)\ge \nu_2(g) +\nu_1(f),
    \]
    \begin{equation}\label{eq:leib}
        d(g\circ f)= d(g)\circ f+ (-1)^{|g|}g\circ d(f),
    \end{equation}
where we assume in~\eqref{eq:leib} that $g$ is homogeneous.   
\end{enumerate} 
\end{df}

\subsubsection{The category of normed matrix factorizations}\label{Section:normed_mf_category_def}
We introduce a normed DG category called the normed matrix factorization category.
Most of the work in this paper will be related to this category.
In Section~\ref{Section:toric_construction_sub} we will describe a construction of a special case of this category motivated by toric geometry. This special case will be used later in this paper in order to construct examples of normed matrix factorization categories. 

We first give the notion of an orthogonal basis of a valued module. 
Let $I$ be an index set of at most countable cardinality, and
let $(N,\nu_N)$ be a free valued module over $(Q,\nu_Q).$ 
    A Schauder basis $\{n_i\}_{i\in I}$ of $N$ is called $\nu_N$-orthogonal if
\[
\nu_N(\sum_{i\in I} q_in_i)=\inf_{i\in I}\{\nu_Q(q_i)+\nu_N(n_i)\},\quad q_i\in Q.
\]   
Let $G$ be an abelian group equipped with homomorphisms $g_G : G \to 2\Z$ and $h_G:G\to \R.$ 
We require the set $\ima h_G$ to be discrete. Let 
\begin{equation*}
R=\bigg\{\sum_{i=0}^\infty a_iT^{\beta_i}| \ a_i\in\R, \ \beta_i\in G, \ \lim_ih_G(\beta_i)=\infty \bigg\}
\end{equation*}
be the Novikov field associated to $G$ equipped with the trivial differential $d_R = 0.$ Thus, $R$ is a completion of the group ring of $G.$ A grading on $R$ is defined by declaring $|T^\beta| = g_G(\beta).$
Equip $R$ with the valuation $\nu_R$ given by
\[
\nu_R(\sum_{i=0}^\infty a_iT^{\beta_i})=\inf_{a_i\ne0}\{h_G(\beta_i)\}. 
\]
Let $n\in\Z_{>0}$ be odd, and let $\Lambda$ be a bounded (possibly not closed) $n$-dimensional polytope with vertices $p_1,\ldots, p_c.$ 
Let $S$ be a finitely generated graded unital $R$-algebra
equipped with a family of valuations $\{\nu^S_\lambda\}_{\lambda\in \Lambda}$ satisfying the following conditions.
\begin{enumerate}[label=(\arabic*)]\label{dfn:S}
    \item $(S, \nu^S_\lambda)$ is a valued algebra over $(R,\nu_R)$ for all $\lambda\in \Lambda,$
    \item\label{1_condition} there exist $s_1,\ldots, s_b \in S$
     such that $\{\prod_{i=1}^bs_i^{d_i}\}_{(d_1,\ldots,d_b)\in\Z^b_{\ge0}}$ is a $\nu_\lambda^S$-orthogonal Schauder basis of $S$ as an $R$-module for all $\lambda \in \Lambda$. In addition, 
     $\nu^S_\lambda(s_i)$ are affine maps and 
    $\lim_{\lambda\to p_l}\nu^S_\lambda(s_i)\in\ima\nu_R$ for all $l\in\{1,\ldots, c\}$ and $i\in\{1,\ldots, b\}.$
\end{enumerate} 
Let $W\in S$ such that $\inf_{\lambda\in\Lambda}\nu^S_\lambda(W)\ge0.$ The element $W$ is called the superpotential.
\begin{df}
    The data $\big((R,\nu_R),(S,\{\nu^S_\lambda\}_{\lambda\in \Lambda}) , W\big)$ is called a \textbf{valued Landau-Ginzburg model.}
\end{df}
We define the \textbf{normed matrix factorization category} $\MF(W,w)$ of a valued Landau-Ginzburg model $\big((R,\nu_R),(S,\{\nu^S_\lambda\}_{\lambda\in \Lambda}) , W \big)$ as follows.
Let $w\in R$ satisfying $\nu_R(w)\ge\sup_{\lambda\in\Lambda}\nu_\lambda^S(W).$ 
The objects of the category are ordered triples $(M,D_M,\{\nu^M_\lambda\}_{\lambda\in \Lambda})$ where
\begin{enumerate}[label=(\arabic*)]\label{dfn:M}
    \item $M$ is a finitely generated free $\Z$-graded $S$-module,
    \item $D_M:M\to M[1]$ is a homomorphism of $S$-modules such that $D_M^2=(W-w)\cdot\id_M,$ and $\nu_\lambda^M(D(m))\ge\nu_\lambda^M(m)$ for all $\lambda\in\Lambda$ and $m\in M,$
    \item\label{03_condition} $\{\nu^M_\lambda\}_{\lambda\in  \Lambda}$ is a family of valuations on $M$ such that $(M,\nu^M_\lambda )$ is a valued module over $(S,\nu^S_\lambda)$ for all $\lambda\in \Lambda.$ 
    In addition, there exists a $\nu_\lambda^M$-orthogonal basis $m_1,\ldots,m_a$ of $M$ for all $\lambda\in \Lambda$ satisfying that $\nu^M_\lambda(m_i)$ are affine maps and
     $\lim_{\lambda\to p_l}\nu^M_\lambda(m_i)\in\ima\nu_R$ for all $i\in\{1,\ldots,a\},$ $l\in\{1,\ldots, c\}.$ 
\end{enumerate}
Abbreviate $\Hom(M^1, M^2)=\Hom\big((M^1, D_{M^1}, \{\nu^{M^1}_\lambda\}_{\lambda\in \Lambda}),(M^2, D_{M^2}, \{\nu^{M^2}_\lambda\}_{\lambda\in \Lambda}) \big).$
The morphism set $\Hom(M^1, M^2)$ is the set of $S$-module homomorphisms between $M^1$ and $M^2.$
This set forms a $\Z$-graded $S$-module equipped with the differential $\delta$ acting on homogeneous elements as
\begin{equation}\label{differential_mf}
    \delta (\Phi):=D_{M^2}\circ \Phi-(-1)^{|\Phi|}\Phi\circ D_{M^1}, \quad \Phi\in \Hom(M^1, M^2).
\end{equation}
The composition of morphisms is given by the composition of homomorphisms of $S$-modules.
In Section~\ref{Section_valuations} we show that the families of valuations $\{\nu_\lambda^{M^i}\}_{\lambda\in\Lambda}$ for $i=1,2$ give rise to a norm $\nu_{M^1,M^2}$ on $\Hom(M^1,M^2).$ 
Lemma~\ref{lm:valuation_algebra_end} shows that $(\Hom(M^1, M^2),\delta,\nu_{M^1, M^2})$ is a valued DG module over $R,$
and Lemma~\ref{lm:valuation_delta} shows that   
the differential $\delta$ satisfies $\nu_{M^1,M^2}(\delta(\Phi))\ge\nu_{M^1,M^2}(\Phi)$ for all $\Phi\in\Hom(M^1,M^2).$ 
By Lemma~\ref{lm_2.14}, the composition between the Hom-complexes respects the valuations.
Hence, $\MF(W,w)$ is a normed DG category over $R.$
We call the objects of $\MF(W,w)$ matrix factorizations. We will also denote the 
objects of $\MF(W,w)$
by $M,$ leaving the homomorphism $D_M$ and the valuations $\{\nu^M_\lambda\}_{\lambda\in \Lambda}$ implicit. 
Define
\[
\MF(W):=\prod_{\substack{w\in R\\ \nu_R(w)\ge\sup_{\lambda\in\Lambda}\nu^S_\lambda(W)}}\MF(W,w).
\]
\subsubsection{The \texorpdfstring{$\infty$}{infinity}-trace}\label{section:normed_cy_str}
Let $(R,d_R,\nu_R)$ be a valued DGA over $\K,$ and let
$\mA$ be a small normed DG category over $R.$
The \textbf{Hochschild complex} $\CC_*(\mA)=\bigoplus_{l\ge1} \CC_l(\mA)$ is given by 
\[
\CC_l(\mA)=\bigoplus \Hom(X^l, X^1)\otimes  \Hom(X^{l-1}, X^l)[1]\otimes\ldots \otimes \Hom(X^1, X^2)[1]
\]
where the sum is over all the collections $X^1,\ldots,X^l$ of objects of $\mA$ of length $l$.
For $a_l\in\Hom(X^l, X^1),$ and $a_i\in\Hom(X^i, X^{i+1})$ where $1\le i\le l-1,$
the corresponding element in $C_l(\mA)$ is denoted by $a_l\otimes \ldots\otimes a_1.$ For $a_i$ homogeneous elements, we write $\epsilon_1=\sum_{i=1}^l|sa_i|,$ where $|sa_i|=|a_i|-1.$ Note that $|a_l\otimes\ldots\otimes a_1|=|a_l|+\sum_{i=1}^{l-1}|sa_i|.$
The complex is equipped with a differential $\dcc$
defined in Section~\ref{subsection_cy} 
and with a valuation $\nu$ induced by the valuations on the Hom-complexes of $\mA.$ 
Let
\begin{equation}\label{eq:cyclic_pernutation}
t: (\CC_*(\mA),\dcc)\to (\CC_*(\mA),\dcc),
\end{equation}
\[t(a_l\otimes a_{l-1}\otimes \ldots\otimes a_1)=(-1)^{|sa_l|(\epsilon_1-|sa_l|)}a_{l-1}\otimes a_{l-2}\otimes\ldots\otimes a_1\otimes a_l\]
be the cyclic permutation.
\begin{df} 
The \textbf{cyclic complex} $(\CC^\lambda_*(\mA),\dcc)$ of $\mA$ is defined by \[
\CC^\lambda_*(\mA):=\CC_*(\mA)/(\ima(\id-t)).\]
The cohomology of this complex, the cyclic homology of $\mA$, is denoted by $\HH_*^\lambda(\mA).$
\end{df}
For a DGA $A$, we denote by $\CC_*(A)$ and $\CC_*^\lambda(A)$ the Hochschild and cyclic complexes of the category consisting of one object with endomorphism algebra $A.$ 
See Section~\ref{subsection_cy} for further details. 
By an operator $\gto$ on $\CC_*^\lambda(\mA),$ we mean that $\gto$ is a sum of operators $\gto=\sum_{l\ge1}\gto_l$ such that for all $l\ge1,$ 
$\gto_l$ is an operator on $\CC_l^\lambda(\mA).$ 
Let $d$ stands for the differential on the $\Hom$-complexes of $\mA.$
Let $  X^1, X^2$ be two objects of $\mA,$ and let $\nu_1$ denote the valuation on the $\Hom(X^1,X^2).$
Define
\[
\Hom(X^1,X^2)_0:=\{\Phi\in \Hom(X^1,X^2)| \ \nu_1(\Phi)\ge0\}/ \{\Phi\in \Hom(X^1,X^2)| \ \nu_1(\Phi)>0\},
\]
\[
R_0:=\{r\in R| \ \nu_R(r)\ge0\}/ \{r\in R| \ \nu_R(r)>0\}.
\]
The requirement $\nu_1\big(d(f)\big)\ge\nu_1(f)$ implies that $d$ descends to $\Hom(X^1,X^2)_0.$
We say $\mA$ is a proper category if for all $X^1,X^2$ objects of $\mA$ we have
$\rank H^*(\Hom(X^1,X^2)_0,d)<\infty.$ 
\begin{df}\label{def:cy_str}
Let $\mA$ be a proper normed DG category over a valued DGA $R$. An \textbf{$\infty$-trace} of dimension $m$ on $\mA$ is an operator 
$\gto: \CC_*^\lambda(\mA)\to R[-m]$ such that
$\gto$ descends to $\HH^\lambda_*(\mA).$ That is,
    \[
    \gto\circ (1-t)=0, \quad \gto\circ\dcc=0.
    \]
We say $\tau$ is \textbf{cohomologically unital} if
\begin{enumerate}[label=(\arabic*)]
\item for every object $A$ of $\mA,$
$\gto$ vanishes on 
the image of the natural embedding
$\CC_*^\lambda(R) \hookrightarrow \CC_*^\lambda(A).$
\end{enumerate}
We say that $\gto$ is \textbf{gapped} if
\begin{enumerate}[label=(\arabic*),resume]
   \item\label{eq:preserves_val_D}
$\mD:=\{\nu_R(\gto(a))- \nu(a)| \  a\not\in \ker\gto \}$
is a discrete subset of $\R.$
\end{enumerate}
We say $\tau$ is \textbf{normed} if it is gapped and
\begin{enumerate}[label=(\arabic*),resume]
   \item
$\nu_R(\gto(a))\ge\nu(a)$
for every $a\in\CC_*^\lambda(\mA).$
\end{enumerate}
We say that $\gto$ is a \textbf{normed chain level Calabi-Yau} structure if it is normed and
\begin{enumerate}[label=(\arabic*),resume]
\item $\gto$ is non-degenerate. That is, for any objects $X^1$ and $X^2$ the induced pairing
\[
H^*(\Hom(X^2,X^1)_0,d)\otimes H^{*-m}(\Hom(X^1,X^2)_0,d)\to R_0, \quad [a_2]\otimes [a_1]\mapsto [\gto_1(a_2a_1)]
\]
is non-degenerate. 
\end{enumerate}
\end{df}
\begin{rmk}
The definition of a chain level Calabi-Yau structure used here is an adaptation of the definition given by Shklyarov~\cite{Shklyarov} to the normed $\Z$-graded setting.
\end{rmk}
\begin{rmk}
In the following, a normed $\infty$-trace is implicitly cohomologically unital. A normed Calabi-Yau structure means a cohomologically unital normed chain level Calabi-Yau structure.
\end{rmk}

\begin{df}\label{dfn:normed_cy_object}
Let $\mA$ be a proper normed DG category over a valued DGA $R$ equipped with a normed $\infty$-trace $\gto$. An object $X$ of $\mA$ is called \textbf{normed Calabi-Yau} if the restriction of $\gto$ to the full subcategory of $\mA$ consisting of the single object $X$ is a normed Calabi-Yau structure.
\end{df}

\subsubsection{Bounding cochains}\label{section:1.2.2}

\begin{df}\label{dfn:bc}
Let $(R,d_R, \nu_R)$ be a commutative valued DGA over $\K.$ Let $(A,d_A,1_A,\nu_A)$ be a normed unital DGA over $R,$ and let $c\in (R)^2$ with $d_R(c)=0$ and $\nu_R(c)>0.$ An element $b\in (A)^1$ is called a \textbf{bounding cochain} if $\nu_A(b)>0,$ and $b$ satisfies the Maurer-Cartan equation
\begin{equation}\label{eq:Marer-Cartan}
d_A(b)-b^2=c\cdot 1_A.
\end{equation}
We denote by $\Mmm(A,c)$ the set of all bounding cochains in $A$ with a fixed $c\in (R)^2$ satisfying $d_R(c)=0$ and $\nu_R(c)>0.$ In addition, we denote by $\Mmm(A)$ the set of all bounding cochains in $A.$
\end{df}
\begin{rmk}
    Bounding cochains were first introduced in~\cite{FOOO1}.
    A bounding cochain in our terminology translates in the terminology of~\cite{FOOO1} to a weak bounding cochain.
\end{rmk}
Let $n\in \Z_{>0}$ be odd, and let $\big((R,\nu_R),(S,\{\nu^S_\lambda\}_{\lambda\in \Lambda}) , W \big)$ be a valued Landau-Ginzburg model. Assume the associated normed matrix factorization category $\MF(W,w)$ is equipped with a dimension $n$ normed $\infty$-trace  
\[
\gto: \CC_*^\lambda(\MF(W,w))\rightarrow R.
\]
\begin{df}\label{dfn:spherical_object}
    An object $(M,D_M,\{\nu^M_\lambda\}_{\lambda\in \Lambda})$ of $\MF(W,w)$ is called a \textbf{spherical object} if
    $H^*(\End(M)_0)\cong H^*(S^n;R_0).$
\end{df}

Let $s$ be a formal variable of degree $1-n.$ Consider the ring $\R[[s]]$ which we equip with the trivial differential and the valuation $\nu_s$ given by
\begin{equation}\label{eq:valuation_s}
\nu_s(\sum_{i=0}^\infty x_is^i)=\min\{i| \ x_i\ne0\}. 
\end{equation}
Define
\begin{equation}\label{eq:tensor_product}
\Rs:=\R[[s]]\otimes R, \quad \ems:=\R[[s]]\otimes\End(M),
\end{equation}
and abbreviate $\nu_M=\nu_{M,M}.$
Equip $\Rs$ with the valuation $\nu_{\Rs}=\nu_s\otimes \nu_R,$ and $\ems$ with the norm 
$\nu_{\Ms}=\nu_s\otimes \nu_M.$ 
The differential $\delta$ and the normed $\infty$-trace $\gto$ are extended to $\ems$ by extension of scalars.
It follows by Corollary~\ref{lm:end_val_dga_2} that $(\ems, \delta, \id_{\Ms},\nu_{\Ms})$ is a normed unital DGA over $\Rs.$ 

Let $\mR_*^\lambda\subset \CC_*^\lambda(\ems)$ be the subcomplex consisting of infinite sums of  elements of the form $a_l\otimes\ldots\otimes a_1$ such that there exists $1\le i\le l$ such that $a_i\in \Rs.$ 
Let 
$b\in \Mmm(\ems,c),$ 
 define
\[
\exp(b):=\id_M \oplus b \oplus \frac{b\otimes b}{2}\oplus\frac{b\otimes b\otimes b}{3}\oplus\ldots\in \CC_*^\lambda(\ems).
\] 
In Corollary~\ref{lm_b(exp)_is_closed} we show that $\dcc(\exp(b))\in\mR_*^\lambda.$
Let $y_b\in \mR_*^\lambda$ be an element satisfying
\[
    \dcc(y_b)=\dcc(\exp(b))+z_b,
\] 
for some $z_b\in\CC_*^\lambda(\Rs),$ and define $\expp(b)=\exp(b)-y_b.$
The element $\expp(b)$ depends on the choice of $y_b,$ but by 
Lemma~\ref{lm:y_b_doesnt_dep}, we obtain that $\gto(\expp(b))$ does not depend on $y_b.$ 
Throughout this paper we will mainly be interested in the value $\gto(\expp(b)),$ and not in the element $\expp(b)$ itself. Hence, we can work with a specific choice of $y_b.$ Section~\ref{section:constructing_y_b} is devoted for the construction of a canonical choice of $y_b,$ and in Corollary~\ref{cor:y_b} we give an explicit formula for this canonical choice. 

Definition~\ref{dfn_gauge_equivalence} gives a notion of gauge equivalence between bounding cochains. Let $\sim$ denote the resulting equivalence relation.
Define a map
\[
\varrho:\Mmm(\ems)/ \sim\longrightarrow (\Rs)^{1-n}
\]
by
\[
\varrho([b]):=\gto(\expp(b)).
\]
Lemma~\ref{building_equiv_2} shows that $\varrho$ is well defined. 
The following theorem is the main result of this paper.
\begin{tm}
\label{thm:bc_equivalence}
    Assume $M$ is a spherical object which is normed Calabi-Yau.  Then $\varrho$ is bijective.
\end{tm}
We think of Theorem~\ref{thm:bc_equivalence} as the ``mirror" of Theorem 1 in~\cite{Point_like_bc}. The latter theorem gives a classification of bounding cochains in the Fukaya $A_\infty$-algebra associated to a Lagrangian submanifold which is a rational cohomology sphere. This type of Lagrangian submanifolds can be considered as the spherical objects of the Fukaya category. Section~\ref{section_bc_mf} is dedicated for the proof of Theorem~\ref{thm:bc_equivalence}. 
 
\subsubsection{Numerical invariants}\label{section:numerical_invariants}
Our goal is to extract numerical invariants from bounding cochains in matrix factorizations that are normed Calabi-Yau spherical objects. 

Let $n\in\Z_{>0}$ be odd. Let $\big((R,\nu_R),(S,\{\nu^S_\lambda\}_{\lambda\in \Lambda}) , W \big)$ be a valued Landau-Ginzburg model, and
let $\MF(W,w)$ be the associated normed matrix factorization category.
Assume $\MF(W,w)$ is equipped with a dimension $n$ normed $\infty$-trace $\gto.$
Let $M$ be a spherical object which is normed Calabi-Yau.
Proposition~\ref{theorem:existence_bc} shows that there exists a bounding cochain $b\in\Mmm(\ems)$ satisfying $\gto(\expp(b))=s.$ We call such $b$ a \textbf{point-like} bounding cochain.

Let $b\in\Mmm(\ems,c).$ We call to $c$ the potential associated to $b.$
The following lemma asserts that the potential of a bounding cochain depends only on the equivalence class of the bounding cochain. 
\begin{lem}\label{building_equiv_2}
Let $b_i\in \Mmm(\ems,c_i)$ for $i=0,1$ such that
$b_0\sim b_1.$ Then $c_0=c_1$ and $\gto(\expp(b_0))=\gto(\expp(b_1)).$  
\end{lem}
The proof of Lemma~\ref{building_equiv_2}
is given in Section~\ref{section:gauge}. 
Hence, we can define invariants of $M$ as follows. 

\begin{df}
\label{def:numerical_invariants}
    Let $M$ be a spherical object which is normed Calabi-Yau. Let $b\in\Mmm(\ems)$
    be a point-like bounding cochain, and let $c$ be the associated potential.
    Write
    \[
    c=\sum_{k\ge0} \sum_{ d\in G} N^M_{d,k+1}\frac{s^k}{k!}T^d,\qquad N^M_{d,k+1}\in \R.
    \]
The constants $N^M_{d,k+1}$ are called the \textbf{numerical invariants} associated to $M.$    
\end{df}

\subsubsection{Construction from toric geometry }\label{section:lg_model}

Delzant~\cite{Delzant} proved a bijective correspondence between symplectic toric manifolds and what are now known as Delzant polytopes. In Section~\ref{Section_toric_cons} we construct from each Delzant polytope $\triangle$ a valued Landau-Ginzburg model $LG(\triangle) = ((R_\triangle,\nu_{R_\triangle}),(S_\triangle,\{\nu^{S_\triangle}_\lambda\}_{\lambda \in \interior{\triangle}}),W_\triangle)$. We equip $\MF(W_\triangle)$ 
with the $\infty$-trace $\Theta$ defined by Shklyarov in~\cite{Shklyarov}, which refines the Kapustin-Li trace given in~\cite{kapustin2003topological}. The definition of $\Theta$ is recalled in Theorem~\ref{Theta} below. We prove that $\Theta$ is of dimension $\dim \triangle$ in Lemma~\ref{lm:degree_theta}. When $\dim \triangle$ is odd, we prove that $\Theta$ is cohomologically unital in Lemma~\ref{lm:vanish_theta}.
\begin{cj}\label{conj:cy}
For each Delzant polytope $\triangle,$ the $\infty$-trace $\Theta$ on $MF(W_\triangle)$ is a normed Calabi Yau structure.    
\end{cj}
Let $\triangle^n = \{(\zeta_1,\ldots,\zeta_n) \in \R^n\,|\, \sum_{i=1}^n \zeta_i \leq 1, \, \zeta_i \geq 0 \,\forall i\}$ denote the standard $n$-simplex.
\begin{tm}\label{thm:norm}
The $\infty$-trace $\Theta$ on $MF(W_{\triangle^n})$ is normed for $n$ odd.
\end{tm}
Let $(\Xtr,\omega_\triangle)$ be the symplectic toric manifold associated to the polytope $\triangle$ by Delzant's construction. The manifold $\Xtr$ is endowed with a canonical real structure, that is, an involution $\phi_\triangle : X \to X$ such that $\phi_\triangle^* \omega_\triangle = -\omega_\triangle.$ Let $\Xtrr = Fix(\phi_\triangle)$ denote the real locus of $\Xtr.$ In particular, $\Xtrr\subset \Xtr$ is a Lagrangian submanifold. For example, $X_{\triangle^n} = \C P^n$ and $X_{\triangle^n}^\R = \R P^n.$ A submanifold $M \subset N$ is called \textbf{relatively spin} if it is orientable and the second Stiefel-Whitney class $w_2(M)$ belongs to the image of the restriction map $H^2(N) \to H^2(M).$ The relative spin condition plays an important role in the definition of the Floer cohomology of a Lagrangian submanifold~\cite{FOOO1}. 

In Section~\ref{Section_toric_cons}, we define a \textbf{combinatorial relative spin} condition for a Delzant polytope, which we expect to be connected to $\Xtrr \subset \Xtr$ being relatively spin when $\Xtr$ is Fano. For example, when $n$ is odd, $\triangle^n$ is combinatorially relatively spin and $\rp^n \subset \cp^n$ is relatively spin. For each combinatorially relatively spin Delzant polytope, we construct an object $(M_\triangle,D_\triangle, \{\nu_\zeta^{M_{\triangle}}\}_{\zeta\in\interior{\triangle}})$ of the normed matrix factorization category $MF(W_\triangle),$ which we call the Dirac matrix factorization. 
\begin{tm}\label{thm:cy}
The object $M_{\triangle^n}$ of $MF(W_{\triangle^n})$ is normed Calabi-Yau for $n$ odd.
\end{tm}
The proof of Theorems~\ref{thm:norm} and~\ref{thm:cy} are given in Section~\ref{section:tight_mf_cpn}. They depend on extensive manipulations of the residues and supertraces that are used to define $\Theta.$

\begin{cj}\label{conj:cohom}
For each combinatorially spin Delzant polytope $\triangle$ such that $\Xtrr$ is Fano, we have 
\[
H^*(\End(M_{\triangle})_0)\cong H^*(\Xtrr; (\Rtr)_0).
\]
\end{cj}

\begin{tm}\label{lm:cohomology}
Conjecture~\ref{conj:cohom} holds for $\triangle = \triangle^n$ for $n$ odd. In particular,  
 \[
H^*(\End(M_{\triangle^n})_0)\cong H^*(\R P^n; (R_{\triangle^n})_0),
\] 
so $M_{\triangle^n}$ is a spherical object.
\end{tm}

\subsubsection{Mirror symmetry}
Conjecture~\ref{conj:cohom} and Theorem~\ref{lm:cohomology} reflect mirror symmetry between the Landau Ginzburg model $LG(\triangle)$ and the symplectic manifold $X_\triangle$ at the level of cohomology. We discuss next mirror symmetry at the level of enumerative geometry.

Let $(X,\omega)$ be a symplectic manifold and $L \subset X$ be a Lagrangian submanifold such that $H^*(L;\R) = H^*(S^n;\R).$ Let $J$ be an $\omega$-tame almost complex structure on $X$, and let $\beta \in H_2(X,L;\Z).$ 
In~\cite{Point_like_bc}, a definition is given for open Gromov-Witten invariants of $L$ that count $J$-holomorphic maps $u:(D^2,\partial D^2) \to (X,L)$ representing $\beta$ satisfying various constraints, along with corrections terms that are needed to make the count independent of the choice of $J$. Following~\cite{Point_like_bc}, let $\oogw^L_{\beta,k}$ denote the invariant where the constraints are taken to be $k$ points on the Lagrangian $L.$ Let $G_\triangle$ denote the group underlying the completed group ring $R_\triangle.$
Our main motivation for the preceding constructions is the following conjecture.
\begin{cj}\label{conj:ms}
Suppose $\Xtr$ is Fano toric and $H^*(\Xtrr;\R) = H^*(S^n;\R).$ Then there is a homomorphism $\varphi_\triangle : H_2(\Xtr,\Xtrr;\Z) \to G_\triangle$ such that for all $k \geq 1$ and $\beta \in G_\triangle,$ we have
\[
N_{\beta,k}^{M_\triangle} = \sum_{\tilde \beta \in \varphi_\triangle^{-1}(\beta)}\oogw^{\Xtrr}_{\tilde\beta,k}.
\]
\end{cj}

\begin{tm}\label{thm:gw_inv_numerical_invariants}
Conjecture~\ref{conj:ms} holds for $\triangle = \triangle^1.$
\end{tm}

Computer calculations, explained in greater detail in Section~\ref{sssec:ccalc}, provide evidence for Conjecture~\ref{conj:ms} for $\triangle = \triangle^3.$ When $\Xtr$ is not Fano, we expect a similar conjecture to hold but with corrections coming from the mirror map. We expect homological  mirror symmetry to manifest itself as an equivalence of categories between $MF(W_\triangle)$ and an appropriate version of the Fukaya category of $X_\triangle.$ We plan to address homological mirror symmetry in this context in future work.

\subsection{The example of projective space}
\subsubsection{Matrix factorizations}\label{sectin:lg_model_introduction}
We describe equivalent valued Landau-Ginzburg model and matrix factorization category to those associated to $\triangle^n,$ the Delzant polytope of $\cp^n$, when $n$ is odd.
This equivalence is shown in Section~\ref{section:changing_coordinates}.

Define
\[
\simp= \{(\zeta_0,\ldots, \zeta_n)\in \R^{n+1}| \ \zeta_0+\ldots+\zeta_n=1, \  0\le \zeta_i \ \forall i \}.
\]
We identify $\simp$ with $\triangle^n$ by the map $(\zeta_0,\ldots,\zeta_n) \mapsto (\zeta_1,\ldots,\zeta_n).$
Let $\interior{\simp}$ denote the interior of $\simp.$ Define
\[
\Rsf:=\bigg\{\sum_{i=0}^\infty a_iT^{\beta_i} | \ a_i\in\R, \ \beta_i\in \frac{1}{n+1}\Z, \ \lim_i\beta_i=\infty \bigg\}.
\]
A grading on $\Rsf$ is defined by $|T^\beta|=2(n+1)\beta.$ 
Let $\nu_{\Rsf}$ be a valuation on $\Rsf$ given by
\begin{equation}\label{eq:val_r_cpn}
\nu_{\Rsf}(\sum_{i=0}^\infty a_iT^{\beta_i})=\inf_{a_i\ne0}\{\beta_i\}.
\end{equation}
Let $\Rsf\langle z_0,\ldots,z_n\rangle$ denote the Tate algebra generated by $z_0,\ldots, z_n.$
Define 
\[
 \Ssf:=\Rsf\langle z_0,\ldots,z_n\rangle/\langle z_0\cdots z_n- T\rangle. 
\]
A grading on $\Ssf$ is defined by declaring $|z_i|=2.$ Equip $\Ssf$ with a family of valuations parameterized by $\interior{\simp}$ as follows.
For $\zeta=(\zeta_0,\ldots , \zeta_n)\in \interior{\simp},$ define
\begin{equation}\label{eq:val_s_cpn}
    \nu_\zeta^{\Ssf}(z_i)=\zeta_i, \quad 0\le i\le n
\end{equation}
and require the Schauder basis $\{\prod_{i=0}^n z_i^{d_i}\}_{(d_0,\ldots,d_n)\in\Z_{\ge0}^{n+1}}$ to be $\nu_\zeta^{\Ssf}$-orthogonal.
Note that $(\Ssf,\nu_\zeta^{\Ssf})$ is a valued algebra over $\Rsf$ for all $\zeta\in\interior{\simp}.$
Let $\Wsf=(-1)^{\frac{n(n+1)}{2}}z_0+\ldots+ z_n\in \Ssf$ be the superpotential and let $w\in\Rsf$ such that $\nu_{\Rsf}(w)\ge \sup_{\zeta\in\interior{\simp}}\nu^{\Ssf}_\zeta(\Wsf).$ Consider the normed matrix factorization category $\MF(\Wsf,w)$ associated to the valued Landau-Ginzburg model $\big((\Rsf,\nu_{\Rsf}),(\Ssf,\{\nu^{\Ssf}_\zeta\}_{\zeta\in \interior{\simp}}) , \Wsf \big).$
We equip $\MF(\Wsf,w)$ 
with the $\infty$-trace $\Theta$ defined in Theorem~\ref{Theta}.
Theorem~\ref{thm:norm} asserts that $\Theta$ is normed.

Consider the case where $w=0.$
Let $e_0,\ldots, e_n$ denote a basis of $\Ssf^{\oplus n+1},$
and let $T(\Ssf^{\oplus n+1})$ denote the tensor algebra.
Let $I$ be the ideal generated by the sets $\{e_0^2-(-1)^{\frac{n(n+1)}{2}}z_0\},$ $\{e_i^2-z_i| \ 1\le i\le n\},$ $\{e_ie_j+e_je_i| \ 0\le i<j\le n\}.$
Let $\Cl$ be the $\Ssf$-algebra defined by  
$\Cl:=T(\Ssf^{\oplus n+1})/I.$ Since $z_i$ are elements of degree two, it follows that $e_i$ are elements of degree one.
Equip $\Cl$ with the family of valuations $\{\nu_\zeta^{\Cl}\}_{\zeta\in \interior{\simp}}$ defined by
\[
\nu_\zeta^{\Cl}(e_i)=\frac{1}{2}\zeta_i
\]
and by requiring the basis $\{e_{i_0}\cdots e_{i_k}\}_{i_0<\ldots<i_k}$ to be $\nu_\zeta^{\Cl}$-orthogonal.
Consider the left ideal $\langle e_0\cdots e_n- T^{1/2}\rangle$ in $\Cl.$ 
Define the left $\Cl$-module $\Msf:=\Cl/\langle e_0\cdots e_n- T^{1/2}\rangle.$ In particular, $\Msf$ is a free $\Ssf$-module.
The family of valuations $\{\nu_\zeta^{\Cl}\}_{\zeta\in \interior{\simp}}$ descends to a family of valuations $\{\nu^{\Msf}_\zeta\}_{\zeta\in \interior{\simp}}$ on $\Msf.$
    \begin{rmk}\label{rem:cl_endomorphism}
Every element $m\in \Cl$ induces an endomorphism $m: \Cl\to \Cl$ given by $x\mapsto m\cdot x.$ So, we consider $\Cl$ as a subalgebra of $\End(\Msf).$
\end{rmk}
Let $\Dsf\in \End(\Msf)$ be given by left multiplication by $e_0+\ldots +e_n$ and note that $\Dsf^2=\Wsf\cdot \id_{\Msf}.$
Equip $\End(\Msf)$ with the norm $\nu_{\Msf}$ defined by
\[
\nu_{\Msf}(\Phi)=\inf_{\zeta\in\interior{\simp}}\inf_{m\in \Msf\backslash\{0\}}\{\nu_\zeta^{\Msf}(\Phi(m))-\nu_\zeta^{\Msf}(m)\}, \quad \Phi\in\End(\Msf).
\]
Write 
\[ \Rms=\R[[s]]\otimes\Rsf, \quad\Ems=\R[[s]]\otimes\End(\Msf).
\]
By Theorems~\ref{thm:cy},~\ref{lm:cohomology} the objects $\Msf$ are normed Calabi-Yau spherical objects.
Hence, Theorem~\ref{thm:bc_equivalence} gives a full classification of bounding cochains in $\Ems$ for all $n>1$ odd. 
We compute the numerical invariants of $\Msf$ defined in Definition~\ref{def:numerical_invariants} where $n=1$.

\begin{pr}\label{prop:bounding cochain}
   Let $n=1.$ Let
    \[
    b=(e^{-s}-e^s)\frac{z_1}{2T^{1/2}}\cdot e_0- (e^{-s}+e^s-2)\cdot\frac{e_1}{2},\qquad c=(e^s-e^{-s})T^{1/2}.
    \]
Then, $b\in \Mmm(\Ems,c)$ such that $\Theta(\expp(b))=s.$
\end{pr}
\begin{crl}\label{cor:n_d_k}
Let $n=1.$ Then,
    \[N_{d,k}=
         \left\{\begin{array}{ll}
        2, &  d=\frac{1}{2},\  k\in 2\Z_{\ge0} ,\\
        0, & \mathrm{otherwise}.\\
        \end{array} \right.\]
\end{crl}

\subsubsection{Construction of bounding cochains}\label{section:construction_of_bc}
The norm $\nu_{\Mms}=\nu_s\otimes\nu_{\Msf}$ on $\EEms$ induces a filtration $F^E$ on $\EEms$ defined by
\[
F^E\EEms=\{\Phi\in \EEms|\nu_{\Mms}(\Phi)>E \}.
\]
In order to compute the invariants $N_{d,k}$ of $\Msf,$ where $n$ is odd, it suffices to compute a bounding cochain $b$ modulo $F^{d+k-1}\EEms$ that satisfies $\Theta(\expp(b))\equiv s \pmod{F^{d+k-1}\EEms}.$ 

We describe briefly the construction of a bounding cochain modulo $F^7\EEms$ following the inductive process described in Proposition~\ref{theorem:existence_bc}.
We begin with a bounding cochain $b_{(1)}:=-\frac{sz_1\cdots z_n}{T^{1/2}}\cdot e_0$ modulo $F^1\EEms.$ 
It satisfies 
\[
\delta(b_{(1)})-b_{(1)}^2\equiv 0 \pmod{F^1\EEms}, \]
and by Lemma~\ref{lm:theta_of_h_0} we get
\[
\Theta(\expp(b_{(1)}))\equiv\Theta_1(b_{(1)})= s  \pmod{F^1\EEms}.
\]
The idea of the inductive step of this construction goes as follows. Assume we build $b_{(l)}\in\End^1(\Msf)$  satisfying 
\[
\delta(b_{(l)})-b_{(l)}^2\equiv c_{(l)}\cdot\id_{\Msf} \pmod{F^{E_l}\EEms}, \]
\[
\Theta(\expp(b_{(l)}))\equiv s \pmod{F^{E_l}\EEms},
\] 
where $c_{(l)}\in\Rms.$
Consider the Maurer-Cartan equation modulo $
F^{E_l+1}\EEms:$ 
\[
\delta(b_{(l)})-b_{(l)}^2\equiv c_{(l)}\cdot\id_{\Msf}+ o_l \pmod{F^{E_l+1}\EEms},
\]
where $o_l\in\End^1(\Msf)^s.$ If $o_l\in \Rms\cdot \id_{\Msf},$ then $b_{(l)}$ is also a bounding cochain modulo $F^{E_l+1}\EEms.$ Otherwise, since $\Msf$ is a spherical object, we will obtain a correction term $b_l\in\EEms$ such that $b_{(l)}+b_l$ is a bounding cochain modulo $F^{E_l+1}\EEms.$ In addition, the properties of $\Theta$ given in Definition~\ref{def:cy_str} together with $\Msf$ being a normed Calabi-Yau object will guarantee there exists a closed element $\alpha_l\in\EEms$ such that 
\[
\Theta_1(\alpha_l)\equiv\Theta(\exp(b_{(6)})+\Theta_1(b_l)-\Theta(y_{b_{(l)}})- s\pmod{F^{E_l}\EEms}, 
\]
and $b_{(l+1)}:=b_{(l)}+b_{l+1}-\alpha_l$ satisfies  
\[
\delta(b_{(l+1)})-b_{(l+1)}^2\equiv \delta(b_{(l)}+b_l)-(b_{(l)}+b_l)^2 \pmod{F^{E_l+1}\EEms},
\]
\[
\Theta(\expp(b_{(l+1)}))\equiv s  \pmod{F^{E_l+1}\EEms}.
\]
So, $b_{(l+1)}$ is a bounding cochain modulo $F^{E_l+1}\EEms.$
\subsubsection{Computer calculation of \texorpdfstring{$N_{d,k}$}{the numerical invariants}} \label{sssec:ccalc}
Consider the object $\Msf$ where $n=3.$ Computing the first steps in the inductive process described in Section~\ref{section:construction_of_bc}, we obtain
\[
b_1=\frac{s^2z_2z_3}{2}\cdot e_1, \quad b_2=0,\quad  b_3=\frac{s^4z_1z_2z_3^2}{8}\cdot e_2,\quad  b_l=0,\quad l=4,5,6
\]
and 
\[
\alpha_l=0,\quad l=1,2,3,4, \quad \alpha_5=\frac{s^5T^{1/2}z_1z_2z_3}{120}\cdot e_0, \quad \alpha_6=0. 
\]
Hence,
\[
b_{(7)}=-\frac{sz_1z_2z_3}{T^{1/2}}\cdot e_0+\frac{s^2z_2z_3}{2}\cdot e_1+\frac{s^4z_1z_2z_3^2}{8}\cdot e_2+\frac{s^5T^{1/2}z_1z_2z_3}{120}\cdot e_0,
\]
\[
c_{(7)}=-2sT^{1/2}+\frac{s^5T^{3/2}}{60}.
\]
Therefore, 
    \begin{equation}\label{eq:n_d_k_3}
    N_{\frac{1}{2},2}=-2,\quad  N_{\frac{3}{2},6}=2, \quad N_{d,k}=0,\quad d+k\le7.
      \end{equation}
      
There are two main difficulties in the computation of these numerical invariants $N_{d,k}.$ 
The first difficulty is to construct $y_{b_{(l)}}.$
Section~\ref{section:defining_the_homotopy} and Section~\ref{section:constructing_y_b} are devoted to the construction of $y_{b_{(l)}},$ and the explicit formula of $y_{b_{(l)}}$ is given in equation~\eqref{eq:y_b_l}.
Roughly speaking, the element $y_{b_{(l)}}$ modulo $F^{E_l}\EEms$ can be written as a linear combination of elements of the form \[
-\bigotimes_{k=1}^m(c_{(l)}\cdot\id_{\Msf}\otimes\id_{\Msf})^{\otimes j_k}\otimes b_{(l)}^{\otimes i_k}.
\]
For example,
\begin{align*}
y_{b_{(5)}}&\equiv -(c_{(5)}\cdot\id_{\Msf}\otimes\id_{\Msf})\otimes b_{(5)}- (c_{(5)}\cdot\id_{\Msf}\otimes\id_{\Msf})\otimes b_{(5)}^{\otimes 2}-(c_{(5)}\cdot\id_{\Msf}\otimes\id_{\Msf})\otimes b_{(5)}^{\otimes 3}\\
&-2(c_{(5)}\cdot\id_{\Msf}\otimes\id_{\Msf})^{\otimes2} \otimes b_{(5)}^{\otimes2}\\
&-\frac{1}{2}(c_{(5)}\cdot\id_{\Msf}\otimes\id_{\Msf})\otimes b_{(5)}\otimes (c_{(5)}\cdot\id_{\Msf}\otimes\id_{\Msf})\otimes b_{(5)} \pmod{F^5\EEms}.
\end{align*}

The second difficult part  is to find the elements $\alpha_l.$ The elements $\alpha_l$ are determined by the values of $\Theta(\exp(b_{(l)})),$ $\Theta(y_{b_{(l)}})$ and $\Theta_1(b_l)$ modulo $F^{E_l+1}\EEms.$ 
In order to compute these values, we need to use the formula of $\Theta$ given in 
Theorem~\ref{Theta}.
In order to use this formula,
we represent the elements of $\EEms$ and the operator $\Dsf$ by $8\times 8$ matrices with coefficients in $\Sms.$ 
For $\Phi_1,\ldots,\Phi_m\in\EEms,$ the value of $\Theta_l(\Phi_m\otimes\ldots\otimes\Phi_1)$ is based on calculation of residues of complex functions. These functions arise as supertraces of products of the matrices that represent $\Phi_1,\ldots, \Phi_m$ and the derivatives with respect to $z_1,\ldots,z_n$ of the matrix that represents $\Dsf$.
The number of the residues we need to calculate in the formula for $\Theta_m$ grows exponentially with $m.$
For $m=1$ there are six residues we need to calculate, for $m=4$ there are $3,900$ residues, and for $m=7$ there are $345,912$ residues we need to calculate.

Until the fifth step of the inductive process, there are no contributions coming from $\Theta,$ that is $\alpha_l=0$ for $l=1,\ldots,4.$ Consider the bounding cochain modulo $F^6\EEms$
\[
b_{(5)}=-\frac{sz_1z_2z_3}{T^{1/2}}\cdot e_0+\frac{s^2z_2z_3}{2}\cdot e_1+\frac{s^4z_1z_2z_3^2}{8}\cdot e_2.
\]
There no correction terms coming from the Maurer-Cartan equation, that is $b_5=0.$ However, in the sixth step we get non-trivial contributions from $\Theta(\exp(b_{(5)}))$ and $\Theta(y_{b_{(5)}}).$ Based on computer computations, the non-trivial contributions from $\Theta(\exp(b_{(5)}))$ are given modulo $F^6\EEms$ as follows:
\begin{gather*}
\Theta_1({b_{(5)}}) \equiv s, \qquad \qquad \Theta_2(\frac{b_{(5)}^{\otimes 2}}{2}) \equiv \frac{s^5T}{8}, \qquad\qquad
    \Theta_3(\frac{b_{(5)}^{\otimes 3}}{3}) \equiv \frac{s^5T}{24},  \\
    \Theta_4(\frac{b_{(5)}^{\otimes 4}}{4}) \equiv -\frac{3s^5T}{10}, \qquad\qquad
    \Theta_5(\frac{b_{(5)}^{\otimes 5}}{5}) \equiv \frac{s^5T}{6}.
\end{gather*}
and the non-trivial contributions of $\Theta(y_{b_{(5)}})$ modulo $F^6\EEms$ are:
\begin{gather*}
\Theta_3\big(-(c_{(5)}\cdot\id_{\Msf}\otimes \id_{\Msf}\otimes b_{(5)})\big) \equiv \frac{s^5T}{8},  \quad \Theta_4\big(-(c_{(5)}\cdot\id_{\Msf}\otimes \id_{\Msf}\otimes b_{(5)}^{\otimes3})\big) \equiv -\frac{11s^5T}{30}, \\
\quad \Theta_6\big(-(c_{(5)}\cdot\id_{\Msf}\otimes \id_{\Msf}\otimes b_{(5)}^{\otimes4})\big)\equiv\frac{4s^5T}{15}.
\end{gather*}
Hence, $\alpha_5$ is a closed element satisfies
\[
\Theta_1(\alpha_l)\equiv\Theta(\exp(b_{(5)})-\Theta(y_{b_{(l)}})-s\equiv\frac{s^5T}{120}. 
\]

\subsection{Outline}
In Section~\ref{section:background} we recall basic definitions such as valued rings and algebras, pseudoisotopy and gauge equivalence.

In Section~\ref{Section_MF} we construct valuations on the Hom-complexes and norms on the End-complexes in normed matrix factorization categories. Then, we construct the normed matrix factorization category associated to a toric manifold, and define the Dirac matrix factorization. 

Section~\ref{Section:hochschild} aims to show the existence of the canonical choice
of the element $y_b.$  

The goal of Section~\ref{section_bc_mf} is to prove Theorem~\ref{thm:bc_equivalence}. 
The proof of this theorem is divided into three parts. First, we prove Proposition~\ref{theorem:existence_bc} which shows that $\varrho$ is surjective. Then, we construct a pseudoisotopy between a normed Calabi-Yau spherical object and itself. Finally, we use this pseudoisotopy in order to prove Proposition~\ref{building_equiv_1} which asserts that $\varrho$ is injective. 

In Section~\ref{Section:example} we discuss the example of the normed matrix factorization category $\MF(\Wsf,w).$  In Sections~\ref{section:4},~\ref{section:spherical_object} we describe the construction of $\MF(\Wsf,w)$ and prove that $\Msf$ is a spherical objects for all $n$ odd.
Then, in Sections~\ref{section:tight_mf_cpn},~\ref{section:normed_cy_object} we show that $\Theta$ is a normed $\infty$-trace on $\MF(\Wsf),$ and that $\Msf$ is normed Calabi-Yau.
In Section~\ref{section:example} we prove Proposition~\ref{theorem:existence_bc} which yields the numerical invariants of $\Msf$ for $n=1.$

\subsection{Acknowledgments} 
The authors would like to thank P. Giterman, who played an important role in helping them understand the significance of Calabi-Yau structures in open Gromov-Witten theory. The authors would like to thank M. Abouzaid, D. Kazhdan, M.~Temkin and D.~Tonkonog for helpful conversations.
The authors were partially supported by ISF grant 1127/22.

\subsection{Notations}
For a graded ring, module or an algebra $\Upsilon$ and a homogeneous element $\alpha\in\Upsilon,$ we denote by $|\alpha|$ the degree of $\alpha$ in $\Upsilon$ and by $(\Upsilon)^i$ the degree $i$ part of $\Upsilon.$ 
let $\Upsilon'$ be a $\R$-vector space and let $\Upsilon$ be a ring or a module. For $x\in \Upsilon'\otimes \Upsilon$ and a monomial $\lambda\in \Upsilon,$ denote by $[\lambda](x)\in \Upsilon'$ the coefficient of $\lambda$ in $x.$ 

For a graded vector space $V,$ let $V[p]$ denote the graded vector space $V$ where its grading is shifted by $p,$ that is $(V[p])^i= (V)^{p+i}.$ 
For an homogeneous element $v\in V,$ let $sv$ denote the element $v\in V$ with degree shifted by $-1.$ So, $|sv|=|v|-1.$ 

Throughout this paper, let $\K$ denote a field of characteristic zero. All modules are taken to be left modules.

\section{Background}\label{section:background}
\subsection{Valuations}\label{Section:valuation}
We give a few definitions concerning valuations and norms on algebraic objects. Some of the definitions are slightly different from those given in the literature. 
\begin{dfn}\label{dfn:valuation_def}
 Let $R$ be a ring. A map $\nu: R\to \R\cup\{\infty\}$ is called a \textbf{valuation} on $R$ if
  \begin{enumerate}[label=(\arabic*)]
    \item $\nu(r)=\infty$ if and only if $r=0,$
      \item $\nu(r_1+r_2)\ge \min\{\nu(r_1),\nu(r_2)\},$
      \item  $\nu(r_1r_2)= \nu(r_1)+\nu(r_2).$
  \end{enumerate}
We call the pair $(R,\nu)$ a \textbf{valued ring}.  
\end{dfn}
If we replace the equality in the third condition with the inequality $\nu(r_1r_2)\ge \nu(r_1)+\nu(r_2),$ then the map $\nu$ is called a \textbf{norm} and the pair $(R, \nu)$ is called a \textbf{normed ring}.

We will also speak about valued fields. These are fields which as rings are valued rings. 

\begin{dfn}\label{3}
    Let $(R,\nu_R)$ be a valued ring, and let $M$ be an $R$-module. We say that $M$ is a \textbf{valued module over} $R,$ if it admits a map $\nu_M:M\to \R \cup \{\infty\}$ satisfying
    \begin{enumerate}[label=(\arabic*)]
        \item $\nu_M(m)=\infty$ if and only if $m=0,$
        \item $\nu_M(m_1+m_2)\ge \min\{\nu_M(m_1),\nu_M(m_2)\},$
        \item $\nu_M(r\cdot m)= \nu_R(r)+\nu_M(m).$
    \end{enumerate}
\end{dfn}
\begin{dfn}
    Let $(R, \nu_R)$ be a valued ring, and let $A$ be an $R$-algebra. We say that $A$ is a \textbf{valued (resp. normed) algebra over} $R,$ if it admits a map $\nu_A: A\to \R\cup \{\infty\}$ such that
\begin{enumerate}[label=(\arabic*)]
        \item $(A,\nu_A)$ is a valued module over $R,$  \item $(A,\nu_A)$ is a valued (resp. normed) ring.
    \end{enumerate}

\end{dfn}
\begin{rem}
Let $(A,1_A, \nu_A)$ be a valued unital algebra over a valued field $(\K,\nu_\K).$ Thus, it follows from the third condition in Definition~\ref{3} that  $\nu_A|_\K=\nu_\K.$        
\end{rem}

\subsection{Gauge equivalence}\label{Subsection_algebraic_frame}
We
recall definitions and properties concerning gauge equivalence. We follow~\cite{pavel}.  
\begin{dfn} 
Let $(R,d_R)$ be a commutative DGA over $\K.$
Let $(A_0,d_{A_0})$ and $(A_1,d_{A_1})$ be DGAs over 
$R.$ A morphism $\varphi:A_0\to A_1$ between DGAs is a graded algebra morphism which satisfies $d_{A_1}\circ\varphi=\varphi\circ d_{A_0}.$
\end{dfn}
\begin{rem}
Let $A_0$ and $A_1$ be unital DGAs, and let $\varphi:A_0\to A_1$ be a morphism of DGAs. If $\varphi(1_{A_0})=1_{A_1}$, then $\varphi$ is called unital.     
\end{rem}

\begin{dfn}
    Let $R_0$ and $R_1$ be two graded commutative $\K$-algebras and let $\varphi:R_0\to R_1$ be a $\K$-algebra morphism. Given an $R_0$-module $C$ and an $R_1$-module $D,$ a $\K$-linear map $f:C\to D$ which satisfies 
    \[
    f(r\cdot c)= \varphi(r)\cdot f(c)
    \]
    for all $r\in R_0$ and $c\in C$ will be called a map over $\varphi.$
\end{dfn}

\begin{dfn}\label{dfn:pseudoisotopy}
Let $(R,d_R,\nu_R)$ be a commutative valued DGA over $\K.$ Let $(A_0,d_{A_0}, \nu_{A_0})$ and $(A_1,d_{A_1},\nu_{A_1})$ be two normed DGAs over $R$ equipped with cohomologically unital $\infty$-traces $\tau^{A_0}$ and $\tau^{A_1}$ respectively. A \textbf{pseudoisotopy} between $A_0$ and $A_1$ is given by the following data:
\begin{enumerate}[label=(\arabic*)]
    \item a valued DGA $(\fR, \fpar,\nu_{\fR})$ over $R$ together with two homotopic DGA morphisms $\mathrm{eval}^i:\fR\to R$ where $i=0,1,$
    \item a normed DGA $(\fU,\bcm,\nu_{\fU})$ over $\fR,$
    \item DGA morphisms $\eval^i:\fU\to A_i$ over $\mathrm{eval}^i$ where $i=0,1,$ 
    \item a cohomologically unital operator $\thet:\CC_*^\lambda(\fU)\to \fR$ satisfying $\fpar\circ\thet=\thet\circ \dcc,$  such that the following diagrams commute.
\[\begin{tikzcd}[ampersand replacement=\&]
	\CC_*^\lambda(\fU)  \&  \fR[m]  \\
	\CC_*^\lambda(A) \& R[m]
	\arrow["\thet", from=1-1, to=1-2]
	\arrow["\mathrm{eval}^i",  , from=1-2, to=2-2]
	\arrow["\mathfrak{eval}^i", from=1-1, to=2-1]
	\arrow["\gto^{A_i}", from=2-1, to=2-2] 
\end{tikzcd}\]
\end{enumerate}
\end{dfn}
We will denote the data of a pseudoisotopy between $A_0$ and $A_1$ by $\fU,$ leaving the base DGA $\fR,$ the operator $\thet$ and the morphisms implicit. 
    If the DGAs $A_0, A_1$ are unital, we require $\fU$ and the morphisms $\eval^0, \eval^1$ to be unital. In this case we say the pseudoisotopy is unital.

\begin{dfn}\label{dfn_gauge_equivalence}
Let $(R,d_R,\nu_R)$ be a commutative valued DGA over $\K.$
Let $(A_0,d_{A_0},1_{A_0},\nu_{A_0})$ and $(A_1,d_{A_1},1_{A_1},\nu_{A_1})$ be two normed unital DGAs over $R$, and let $\fU$ be a unital pseudoisotopy between $A_0$ and $A_1.$
Let $b_0\in\Mmm(A_0,c_0), $ and  $b_1\in\Mmm(A_1,c_1).$  
We say that $b_0$ and $b_1$ are $\fU$\textbf{-gauge equivalent} if there exists an element $\bcc\in (\fR)^2$ with
\[
\fpar\bcc=0,\quad \nu_{\fRs}(\bcc)>0,\quad \mathrm{eval}^0(\bcc)=c_0, \quad \mathrm{eval}^1(\bcc)=c_1,
\] and $\bcb\in \Mmm(\fU,\bcc)$ such that 
\[
\eval^0(\bcb)=b_0, \quad \eval^1(\bcb)=b_1.
\]
\end{dfn}
We call $\widetilde{b}$ a pseudoisotopy between $b_0$ and $b_1.$ If $b_0$ and $ b_1$ are $\fU$-gauge equivalent, then we denote $b_0\sim_\fU b_1.$ 
The following lemma is proved in~\cite{pavel}.
\begin{lm}\label{remain_c}
Let $(R,d_R,\nu_R)$ be a commutative valued DGA over $\K.$
Let $(A_0,d_{A_0},1_{A_0},\nu_{A_0})$ and $(A_1,d_{A_1},1_{A_1},\nu_{A_1})$ be two normed unital DGAs over $R$.
Let $b_0\in \Mmm(A_0, c_0)$ and $b_1\in \Mmm(A_1,c_1)$ be two $\fU$-gauge equivalent bounding cochains. Then $[c_0]=[c_1]\in H^*(R,d_R).$   
\end{lm}
\subsection{Central simple graded algebras}\label{section:azumaya_algebra}
Let $A$ be a graded $\K$-algebra. Let $h(A)$ denote the set of homogeneous elements of $A.$ Recall the center of a graded algebra $A$ is defined by 
\[
\widehat{Z}(A):=\{a\in A| \ aa'=(-1)^{|a|\cdot|a'|}a'a, \ \forall a'\in h(A)\}.
\]
\begin{dfn}
    Let $A$ be a finite dimensional graded $\K$-algebra. We say that $A$ is a \textbf{central simple graded algebra} (CSGA) over $\K$ if $A$  has no proper graded ideal and $\widehat{Z}(A)=\K.$
\end{dfn}
The following is Theorem 2.1 in~\cite{quadratic_forms}.
\begin{lm}\label{lm:clifford_is_csga}
    Let $V$ be a vector space over $\K$ equipped with a quadratic form $q.$ The Clifford algebra $\mathrm{Cl}(V,q)$ is a CSGA over $\K.$
\end{lm}
The following is part of Theorem 2.3 in~\cite{quadratic_forms}.
\begin{lm}\label{lm:tensor_producy_csga}
    If $A$ and $B$ are both CSGAs over $\K,$ so is $A\otimes_\K B.$ 
\end{lm}

\begin{lm}\label{lm:a_otimes_f_csga}
    Let $A$ be a $\K$-algebra and let $\F$ be a field extension of $\K.$ Then, if $A\otimes \F$ is a CSGA, so is $A.$
\end{lm}
\begin{proof}
First, we show that $\widehat{Z}(A\otimes_\K \F)=\F.$
Since $A$ is a $\K$-algebra, we have
$\K\subset \widehat{Z}(A).$ Thus, we need to show $\widehat{Z}(A)\subset \K. $
Let $a\in \widehat{Z}(A),$ so $aa'=(-1)^{|a|\cdot |a'|}a'a$ for all $a'\in h(A).$ Thus,
\[
(a\otimes 1_A)\cdot (a'\otimes f)= aa'\otimes f=(-1)^{|a|\cdot|a'|}a'a\otimes f= (-1)^{|a|\cdot |a'|}(a'\otimes f)\cdot (a\otimes 1_A),
\] 
for all $a'\otimes f\in A\otimes_\K \F$ where $a'\in h(A).$ 
So, $a\otimes 1_A\in\widehat{Z}(A\otimes_\K \F).$
Since $\widehat{Z}(A\otimes_\K \F)=\F,$ it follows that  
$a\in \K.$ Thus, $\widehat{Z}(A\otimes_\K \F)\subset \K. $

Second, we show that $A$ is simple. Let $I$ be a graded ideal of $A,$ so $I\otimes_\K \F$ is a graded ideal of $A\otimes_\K\F.$ Hence, $I\otimes_\K \F$ is either the zero ideal or $A\otimes_\K \F,$ which implies that $I$ is either the zero ideal or $A.$ 
\end{proof}
Recall the definition of $\Cl$ in Section~\ref{sectin:lg_model_introduction}. 
Let $\mathfrak{m}$ be a maximal ideal of $\Ssf,$ and write 
$k=\Ssf/\mathfrak{m}.$ Let $\overline{k}$ denote the algebraic closure of $k.$ 
Define $ \mathrm{Cl}(\overline{k}^{n+1}):=\mathrm{Cl}_{\mathrm{even}}(\triangle^n)\otimes \overline{k}.$
Since $\overline{k}$ is algebraically closed, all the quadratic forms with coefficients in $\overline{k}$ are equivalent. Hence, we may assume that $\mathrm{Cl}(\overline{k}^{n+1})$ is defined by the standard quadratic form. That is, $\mathrm{Cl}(\overline{k}^{n+1})$ denotes the Clifford algebra of $\overline{k}^{n+1}$ such that $e_ie_j+e_je_i=2\delta_{ij},$ where $\delta_{ij}$ is the Kronecker delta.
Let $\mathrm{Cl}_{\mathrm{even}}(\overline{k}^{n+1})$ denote the subalgebra of $\mathrm{Cl}(\overline{k}^{n+1})$ that consists the elements of even degree,
 and define $\mathrm{Cl}_{\mathrm{even}}^+(\overline{k}^{n+1}):=\mathrm{Cl}_{\mathrm{even}}(\overline{k}^{n+1})/(e_0\cdots e_n- T^{1/2}).$ 
The proof of the following lemma is by an argument similar to that of Lemma 2.3.3 in~\cite{Morgan}.
\begin{lm}\label{lm:clifford_isomorphism}
Let $\mathfrak{m}$ be a maximal ideal of $\Ssf$ and write 
$k=\Ssf/\mathfrak{m}.$ Let $\overline{k}$ denote the algebraic closure of $k.$ Then,
    $\mathrm{Cl}^-_\mathrm{even}(\overline{k}^{n+1})\cong\mathrm{Cl}(\overline{k}^{n-1}),$ where $n\ge1$ odd. 
\end{lm}
\begin{lm}\label{lm:ccle_csga}
    Let $\mathfrak{m}$ be a maximal ideal of $\Ssf$ and write 
$k=\Ssf/\mathfrak{m}.$ Then,
    $\ccle\otimes k$ is a CSGA where $n\ge1$ odd. 
\end{lm}
\begin{proof}
Let $\overline{k}$ denote the algebraic closure of $k.$ By Lemma~\ref{lm:clifford_is_csga} $\mathrm{Cl}(\overline{k}^{n-1})$ is a CSGA. Thus, by Lemma~\ref{lm:clifford_isomorphism} it follows that $\mathrm{Cl}_{\mathrm{even}}^+(\overline{k}^{n+1})$ is a CSGA. Hence,
since $(\ccle\otimes k)\otimes\overline{k}\cong \mathrm{Cl}_{\mathrm{even}}^+(\overline{k}^{n+1}),$ by Lemma~\ref{lm:a_otimes_f_csga} we get that $\ccle\otimes k$ is a CSGA.

\end{proof}
\section{Normed matrix factorization category}\label{Section_MF}
This section is divided into three parts. We begin by proving a few properties concerning orthogonal bases of valued modules. Then, we
construct valuations on Hom-complexes in the normed matrix factorization category. 
Finally,
we  
describe a construction of a valued Landau-Ginzburg model motivated by toric geometry. An example of this construction will be discussed in more details in Section~\ref{Section:example}.

\subsection{Orthogonal bases in valued modules}

Let $(Q,\nu_Q)$ be a valued ring, and
let $(N,\nu_N),$ $(N',\nu_{N'})$ be valued modules over $(Q, \nu_Q).$ 
Define
\[
\nu_{N,N'}(\Phi):=\inf_{n\in N\backslash\{0\}}\{\nu_{N'}(\Phi(n))-\nu_N(n)\}, \quad\Phi\in \Hom(N,N').
\]

\begin{lm}\label{lm:N_valuation}
   $(\Hom(N,N'), \nu_{N,N'})$ is a valued module over $Q.$
\end{lm}

\begin{proof}
First, since $(N,\nu_N),$ $(N',\nu_{N'})$ are valued modules over $(Q, \nu_Q)$, it follows that $\Hom(N,N')$ is a $Q$-module and
from the definition of $\nu_{N,N'}$ it follows that $\nu_{N,N'}(\Phi)=\infty$ if and only if $\Phi$ is the zero homomorphism. Let $\Phi,\Psi\in \Hom(N,N').$
Next, we claim  $\nu_{N,N'}(\Phi+\Psi)\ge \min\{ \nu_{N,N'}(\Phi), \nu_{N,N'}(\Psi)\}.$
Indeed,
\begin{align*}
    \nu_{N,N'}(\Phi+\Psi)&= \inf_{n\in N\backslash\{0\}}\{\nu_{N'}\big(\Phi(n)+\Psi(n)\big)-\nu_N(n)\}\\
    &\ge \inf_{n\in N\backslash\{0\}}\big\{\min\{\nu_{N'}(\Phi(n)), \nu_{N'}(\Psi(n))\}-\nu_N(n)\big\}\\
    &=\min\big\{ \inf_{n\in N\backslash\{0\}}\{\nu_{N'}(\Phi(n))-\nu_N(n)\},\inf_{n\in N\backslash\{0\}}\{ \nu_{N'}(\Psi(n))-\nu_N(n)\}\big\}\\
    &=\min\{ \nu_{N,N'}(\Phi), \nu_{N,N'}(\Psi)\}.
\end{align*}
For any $q\in Q,$ we obtain 
\begin{align*}
    \nu_{N,N'}(q\Phi)&= \inf_{n\in N\backslash\{0\}} \{\nu_{N'}(q\Phi(n))- \nu_N(n)\}\\
    &=\inf_{n\in N\backslash\{0\}} \{\nu_Q(q)+\nu_{N'}(\Phi(n))- \nu_N(n)\}\\
    &=\nu_Q(q)+\inf_{n\in N\backslash\{0\}} \{\nu_{N'}(\Phi(n))- \nu_N(n)\}\\
    &=\nu_Q(q)+\nu_{N,N'}(\Phi).
\end{align*}
\end{proof} 
\begin{lm}\label{lm:weakly_normed_N}
    Let $(N,\nu_N), (N',\nu_{N'}), (N'',\nu_{N''})$ be valued modules over $Q.$
    Then, for all $\Psi\in \Hom(N,N')$ and $\Phi\in \Hom(N',N''),$ we have
    \[
    \nu_{N,N''}(\Phi \Psi)\ge \nu_{N',N''}(\Phi)+\nu_{N,N'}(\Psi).
    \]
\end{lm}
\begin{proof}
 \begin{align*}
   \nu_{N,N''}(\Phi\Psi)&= \inf_{n\in N\backslash\{0\}}\{\nu_{N''}(\Phi(\Psi(n)))-\nu_N(n)\}\\
    &=\inf_{n\in N\backslash\{0\}}\{\nu_{N''}(\Phi(\Psi(n)))- \nu_{N'}(\Psi(n)) +  \nu_{N'}(\Psi(n))- \nu_N(n)\}\\
    &\ge \inf_{n\in N\backslash\{0\}}\{\nu_{N''}(\Phi(\Psi(n)))- \nu_{N'}(\Psi(n))\}  +\inf_{n\in N\backslash\{0\}}\{  \nu_{N'}(\Psi(n))- \nu_N(n)\}\\
    &\ge \inf_{n\in N'\backslash\{0\}}\{\nu_{N''}(\Phi(n))- \nu_{N'}(n)\}  +\inf_{n\in N\backslash\{0\}}\{  \nu_{N'}(\Psi(n))- \nu_N(n)\}\\
    &= \nu_{N',N''}(\Phi)+\nu_{N,N'}(\Psi).
\end{align*}    
\end{proof}
Let $I$ be an index set of at most countable cardinality.
Recall that a Schauder basis $\{n_i\}_{i\in I}$
of $(N,\nu_N)$ is called $\nu_N$-orthogonal if
\[
\nu_N(\sum_{i\in I}q_in_i)=\inf_{i\in I}\{\nu_Q(q_i)+\nu_N(n_i)\},\quad q_i\in Q.
\]  
Assume both $(N,\nu_N)$ and $(N',\nu_{N'})$ admit $\nu_N$-orthogonal and $\nu_{N'}$-orthogonal Schauder bases denoted by
$\{n_i\}_{i\in I}$ and $\{n'_j\}_{j\in J}$ respectively.
Let $\{e_{ij}\}_{(i,j)\in I\times J}$ be a Schauder basis of $\Hom(N, N')$ acting on $\{n_i\}_{i\in I}$ as $
e_{ij}(n_k)= \delta_{ik}n'_j,$ where $\delta_{ik}$ is the Kronecker delta.
\begin{lm}\label{lm:N_e_ij}
    $\nu_{N,N'}(e_{ij})=\nu_{N'}(n'_j)-\nu_N(n_i).$
\end{lm}
\begin{proof}
 Let $n\in N\backslash\{0\},$ and write $n=\sum_{k\in I}q^n_kn_k$ where $q^n_k\in Q.$ 
We obtain
\begin{align*}
    \nu_{N'}(e_{ij}(n))-\nu_N(n)&=  \nu_{N'}(e_{ij}(\sum_{k\in I}q^n_kn_k))-\nu_N(\sum_{k\in I}q^n_kn_k)\\
    &= \nu_{N'}(q^n_in'_j)-\inf_{k\in I}\{\nu_Q(q^n_k)+\nu_N(n_k)\}.
\end{align*}
If $q^n_i=0,$ then $ \nu_{N'}(e_{ij}(n))-\nu_N(n)=\infty.$ If $q^n_i\ne0,$ then 
\begin{align*}
 \nu_{N'}(e_{ij}(n))-\nu_N(n)&=\nu_{N'}(q^n_in'_j)-\inf_{k\in I}\{\nu_Q(q^n_k)+\nu_N(n_k)\}\\
&\ge\nu_Q(q^n_i)+ \nu_{N'}(n'_j)-\nu_Q(q^n_i)-\nu_N(n_i)\\
&=\nu_{N'}(n'_j)-\nu_N(n_i).
\end{align*}
 Hence, 
 \[
\nu_{N,N'}(e_{ij})\ge \nu_{N'}(n'_j)-\nu_N(n_i).
\] 
Taking $n=n_i$ yields $\nu_{N'}(e_{ij}(n))-\nu_N(n)=\nu_{N'}(n'_j)-\nu_N(n_i).$
\end{proof}

\begin{lm}\label{lm:orthogonal_basis}
    $\{e_{ij}\}_{(i,j)\in I\times J}$ is a $\nu_{N,N'}$-orthogonal Schauder basis of $\Hom(N,N').$
\end{lm}

\begin{proof}
Note that  $\{e_{ij}\}_{(i,j)\in I\times J}$ is a Schauder basis of $\Hom(N,N'),$ so we need to show it is $\nu_{N,N'}$-orthogonal.
Let $n\in N\backslash\{0\},$ and write $n=\sum_{k\in I}q^n_kn_k$ where $q^n_k\in Q.$ 
We get 
\begin{align*}
\nu_{N'}\big( \sum_{i,j}q_{i,j}e_{i,j}(n) \big)&=\nu_{N'}\big( \sum_{i,j}q_{i,j}q_i^nn_j' \big)\\
&=\inf_{i,j}\{\nu_Q(q_{i,j})+\nu_{N'}(q_i^nn_j')\}\\
&=\inf_{i,j}\{\nu_Q(q_{i,j})+\nu_{N'}(e_{i,j}(n))\}.
\end{align*}
Thus,
\begin{align*}
    \nu_{N,N'}(\sum_{i,j}q_{i,j}e_{i,j})&=\inf_{n\in N\backslash\{0\}}\{ \nu_{N'}\big( \sum_{i,j}q_{i,j}e_{i,j}(n) \big)- \nu_N(n) \}\\
    &=\inf_{n\in N\backslash\{0\}}\{ \inf_{i,j}\{\nu_Q(q_{i,j})+\nu_{N'}(e_{i,j}(n))\}- \nu_N(n) \}\\
    &=\inf_{i,j}\big\{\inf_{n\in N\backslash\{0\}}\{\nu_Q(q_{i,j})+ \nu_{N'}(e_{i,j}(n))- \nu_N(n) \}\big\}\\
    &=\inf_{i,j}\big\{\nu_Q(q_{i,j})+\inf_{n\in N\backslash\{0\}}\{ \nu_{N'}(e_{i,j}(n))- \nu_N(n) \}\big\}\\
    &=\inf_{i,j}\{\nu_Q(q_{i,j}) +\nu_{N,N'}(e_{i,j})\}.
\end{align*}
    
\end{proof}

\subsection{Valuations on Hom-complexes}\label{Section_valuations}
Let $\big((R,\nu_R),(S,\{\nu^S_\lambda\}_{\lambda\in \Lambda}) , W \big)$ be a valued Landau-Ginzburg model, and let $\MF(W,w)$ be the associated normed matrix factorization category. 
 
Let $(M^i,D_{M^i},\{\nu^{M^i}_\lambda\}_{\lambda\in \Lambda})$ for $i=1,2$ be two objects of $\MF(W,w).$
Note that $\Hom(M^1,M^2)$ is naturally an $S$-module, but since $S$ is an $R$-algebra, we can consider $\Hom(M^1,M^2)$ as an $R$-module. 
Define
\begin{equation}\label{eq:valuation_lambda_hom}
{\nu_{M^1,M^2}}_\lambda(\Phi):=\inf_{m\in M^1\backslash\{0\}}\{\nu^{M^2}_\lambda(\Phi(m))-\nu^{M^1}_\lambda(m)\}, \quad \lambda\in \Lambda, \quad \Phi\in\Hom(M^1,M^2).
\end{equation}
\begin{lm}\label{lm:val_fam_end_m}
   $(\Hom(M^1,M^2), {\nu_{M^1,M^2}}_\lambda)$ is a valued DG module over $R$ for all $\lambda\in \Lambda.$
\end{lm}

\begin{proof}
    By Lemma~\ref{lm:N_valuation} we need to show that $(M^i,\nu_\lambda^{M^i})$ are valued modules over $R$ for all $\lambda\in\Lambda.$ 
    Since $(S, \nu^S_\lambda)$ is a valued unital algebra over $R$ for all $\lambda\in \Lambda,$ it follows that
$\nu_R(r)=\nu^S_\lambda(r)$ for all $r\in R$ and $\lambda\in \Lambda.$ Hence, for all $r\in R$ and $m_i\in M^i$ we get $\nu^{M^i}_\lambda(rm_i)=\nu_\lambda^S(r)+\nu^{M^i}_\lambda(m_i)=\nu_R(r)+\nu^{M^i}_\lambda(m_i),$ where $i=1,2.$ Thus, $(M^i,\nu_\lambda^{M^i})$ are valued modules over $R$ for all $\lambda\in\Lambda.$
\end{proof}

\begin{rem}\label{remark_S_M}
   It follows by Lemma~\ref{lm:N_valuation} that  $(\Hom(M^1,M^2), {\nu_{M^1,M^2}}_\lambda)$ is a valued module over $(S, \nu_\lambda^S)$ for all $\lambda\in \Lambda.$
\end{rem}
Define 
\begin{equation}\label{eq:val_hom_01}
\nu_{M^1,M^2}(\Phi)=\inf_{\lambda\in \Lambda}\{{\nu_{M^1,M^2}}_\lambda(\Phi)\},\quad \Phi\in\Hom(M^1,M^2).    
\end{equation}

\begin{lm}\label{lm:valuation_algebra_end}
 $(\Hom(M^1,M^2),\nu_{M^1,M^2})$ is a valued DG module over $R.$
\end{lm}
\begin{proof}
    First,  $\nu_{M^1,M^2}(\Phi)=\infty$ if and only if ${\nu_{M^1,M^2}}_\lambda(\Phi)=\infty$ for all $\lambda\in \Lambda.$ Hence, $\nu_{M^1,M^2}(\Phi)=\infty$ if and only if $\Phi$ is the zero homomorphism.
    Let $\Phi, \Psi\in \Hom(M^1,M^2).$ We get
    \begin{align*}
        \nu_{M^1,M^2}(\Phi+\Psi)&=\inf_{\lambda\in \Lambda}\{{\nu_{M^1,M^2}}_\lambda(\Phi+\Psi)\}\\
        &\ge\inf_{\lambda\in \Lambda}\big\{\min\{{\nu_{M^1,M^2}}_\lambda(\Phi),{\nu_{M^1,M^2}}_\lambda(\Psi)\}\big\}\\
        &=\min\big\{\inf_{\lambda\in \Lambda}\{{\nu_{M^1,M^2}}_\lambda(\Phi)\},\inf_{\lambda\in \Lambda}\{{\nu_{M^1,M^2}}_\lambda(\Psi)\}\big\}\\
        &=\min\{\nu_{M^1,M^2}(\Phi),\nu_{M^1,M^2}(\Psi)\},
    \end{align*}
and for all $r\in R$ we obtain
\begin{align*}
    \nu_{M^1,M^2}(r\Phi)&=\inf_{\lambda\in \Lambda}\{{\nu_{M^1,M^2}}_\lambda (r\Phi)\}\\
    &=\inf_{\lambda\in \Lambda}\{\nu_R(r)+{\nu_{M^1,M^2}}_\lambda  (\Phi) \}\\
    &=\nu_R(r)+\inf_{\lambda\in \Lambda}\{{\nu_{M^1,M^2}}_\lambda  (\Phi)\}\\
    &=\nu_R(r)+ \nu_{M^1,M^2}(\Phi).
\end{align*}
\end{proof}
Recall
\[
\Hom(M^1,M^2)_0:=\{\Phi\in \Hom(M^1,M^2)| \ \nu_{M^1,M^2}(\Phi)\ge0\}/ \{\Phi\in \Hom(M^1,M^2)| \  \nu_{M^1,M^2}(\Phi)>0\}.
\]
\begin{lm}\label{lm:valuation_delta}
     $\nu_{M^1,M^2}(\delta(\Phi))\ge\nu_{M^1,M^2}(\Phi)$ for all $\Phi\in\Hom(M^1,M^2).$
\end{lm}
\begin{proof}
Let $\Phi\in\Hom(M^1,M^2).$
It suffices to show that $\nu_{{M^1,M^2}_\lambda}(\delta(\Phi))\ge \nu_{{M^1,M^2}_\lambda}(\Phi)$ for all $\lambda\in\Lambda.$ We have
\begin{align*}
    \nu_{{M^1,M^2}_\lambda}(\delta(\Phi))&=\inf_{m\in M^1\backslash\{0\}}\{\nu^{M^2}_\lambda(\delta(\Phi)(m))-\nu^{M^1}_\lambda(m)\}\\
    &=\inf_{m\in M^1\backslash\{0\}}\{\nu^{M^2}_\lambda\big(D_{M^2}\circ\Phi(m)-(-1)^{|\Phi|}\Phi\circ D_{M^1}(m)\big)-\nu^{M^1}_\lambda(m)\}\\
    &=\inf_{m\in M^1\backslash\{0\}}\big\{\min\{ \nu^{M^2}_\lambda\big(D_{M^2}\circ\Phi(m)\big),\nu^{M^2}_\lambda\big(\Phi\circ D_{M^1}(m)\big)\}-\nu^{M^1}_\lambda(m)\big\}\\
    &=\inf_{m\in M^1\backslash\{0\}}\big\{\min\{ \nu^{M^2}_\lambda\big(D_{M^2}\circ\Phi(m)\big)-\nu^{M^1}_\lambda(m),\nu^{M^2}_\lambda\big(\Phi\circ D_{M^1}(m)\big)-\nu^{M^1}_\lambda(m)\}\big\}.
\end{align*}
Recall
$\nu_\lambda^{M^i}(D_{M^i}(m))\ge\nu_\lambda^{M^i}(m)$ for all $\lambda\in\Lambda$ and $m\in M^i.$ Thus, we get
\[
\nu^{M^2}_\lambda\big(D_{M^2}\circ\Phi(m)\big)-\nu^{M^1}_\lambda(m)\ge\nu^{M^2}_\lambda(\Phi(m))-\nu^{M^1}_\lambda(m),
\]
and 
\[
\nu^{M^2}_\lambda\big(\Phi\circ D_{M^1}(m)\big)-\nu^{M^1}_\lambda(m)
\ge\nu^{M^2}_\lambda\big(\Phi\circ D_{M^1}(m)\big)- \nu_\lambda^{M^1}(D_{M^1}(m)).
\]
Hence,
\begin{align*}
    \nu_{{M^1,M^2}_\lambda}(\delta(\Phi))&\ge\inf_{m\in M^1\backslash\{0\}}\big\{\min\{ \nu^{M^2}_\lambda(\Phi(m))-\nu^{M^1}_\lambda(m),\nu^{M^2}_\lambda\big(\Phi\circ D_{M^1}(m)\big)- \nu_\lambda^{M^1}(D_{M^1}(m))\}\big\}\\
    &=\inf_{m\in M^1\backslash\{0\}}\{\nu^{M^2}_\lambda(\Phi(m))-\nu^{M^1}_\lambda(m)\}\\
    &=\nu_{{M^1,M^2}_\lambda}(\Phi).
\end{align*}

\end{proof}

Recall $\Lambda$ is a bounded $n$-dimensional polytope with vertices $p_1,\ldots, p_c.$ 
Let $s_1,\ldots, s_b\in S$  satisfy~\ref{1_condition}. 
 Let $\{w_k\}_{k=1}^\infty$ be the basis of $S$ as an $R$-module that generated by products of $s_1,\ldots, s_b,$ that is 
    \[
    w_k=\prod_{i=1}^b s_i^{{d_k}_i}, \quad {d_k}_i\in \Z_{\ge0}.
    \]
Abbreviate $\nu^S_{p_l}(s):=\lim_{\lambda\to p_l}\nu^S_\lambda(s).$

\begin{lm}\label{lm:s_pj}
    $\nu^S_{p_l}(s)\in\ima\nu_R$ for all $s\in S$ and $l\in\{1,\ldots,c\}.$
\end{lm}
\begin{proof}
Since $\nu_\lambda^S(w_k)=\sum_{i=1}^b{d_k}_i\nu^S_\lambda(s_i)$ and $\nu^S_{p_l}(s_i)\in\ima\nu_R,$ it follows that $\nu_{p_l}^S(w_k)\in\ima\nu_R$ for all $l\in\{1,\ldots, c\}$ and $k\ge1.$      
Let $s\in S.$ So, we can write
\[
s=\sum_{k=1}^\infty r_k w_k, \quad r_k\in R. 
\]
We get, $\nu^S_{p_l}(s)=\inf_k\{\nu_R(r_k)+\nu^S_{p_l}(w_k)\}\in\ima\nu_R$ for all $l\in\{1,\ldots, c\}.$
\end{proof}
We say that a map $f:\Lambda\to \R$ is piecewise affine if $f(\lambda)=\inf_{i\ge1}\{g_i(\lambda)\},$ where $g_i(\lambda)$ are affine maps. Note that such $f$ is a concave map. Hence, the infimum of $f$ is achieved when $\lambda$ converges to one of the vertices of $\Lambda.$
\begin{lm}\label{lm:affine}
    $\nu_\lambda^S(s)$ is a piecewise affine map for all $s\in S.$
\end{lm}
\begin{proof}
We have $\nu_\lambda^S(w_k)=\sum_{i=1}^b{d_k}_i\nu^S_\lambda(s_i).$ Since $\nu^S_{p_l}(s_i)$ are  affine maps, it follows that $\nu_\lambda^S(w_k)$ are affine as well. 
Let $s\in S.$ So, we can write
\[
s=\sum_{k=1}^\infty r_k w_k, \quad r_k\in R. 
\]
Since $\nu^S_{p_l}(s)=\inf_k\{\nu_R(r_k)+\nu^S_{p_l}(w_k)\}\in\ima\nu_R,$ it follows that $\nu_\lambda^S(s)$ is a piecewise affine maps
\end{proof}

Let $m_1,\ldots m_a$ and $m_1',\ldots, m_{a'}'$ be $\nu_\lambda^{M^1}$-orthogonal and $\nu_\lambda^{M^2}$-orthogonal bases of $M^1$ and $M^2$ 
respectively.
Let $\{e_{ij}\}_{\substack{1\le i\le a\\ 1\le j\le a'}}$ be a basis of $\Hom(M^1, M^2)$ as an $S$-module acting on $m_1,\ldots m_a$ as $
e_{ij}(m_k)= \delta_{ik}m'_j.$ 
The following lemma follows immediately from Lemma~\ref{lm:N_e_ij}. 
\begin{lm}\label{lm:e_ij}
    ${\nu_{M^1, M^2}}_\lambda(e_{ij})=\nu^{M^2}_\lambda(m'_j)-\nu^{M^1}_\lambda(m_i).$
\end{lm}
The following lemma is showing a relation between the valuations $\nu_{M^1,M^2}$ and $\nu_R.$ This will be used in order to prove Theorem~\ref{thm:bc_equivalence} and Theorem~\ref{thm:cy}. 
\begin{lm}\label{lm:ima_m_eq_ima_r}
$\ima\nu_{M^1,M^2}=\ima\nu_R.$ 
\end{lm}
\begin{proof}
Let $\Phi\in \Hom(M^1, M^2),$ and write
$
\Phi=\sum_{i,j}s_{ij}e_{ij} $ where $s_{ij}\in S.
$
By Lemma~\ref{lm:orthogonal_basis} we have
\[
{\nu_{M^1,M^2}}_\lambda(\Phi)=\inf_{i,j}\{\nu^S_\lambda(s_{ij})+{\nu_{M^1,M^2}}_\lambda(e_{ij})\}.
\]
By assumption, the maps $\nu^{M^1}_\lambda(m_i), \nu^{M^2}_\lambda(m'_j)$ are affine. 
By Lemma~\ref{lm:affine} 
$\nu^S_\lambda(s_{ij})$ are piecewise affine maps. Hence, by Lemma~\ref{lm:e_ij},  $\nu^S_\lambda(s_{ij})+{\nu_{M^1,M^2}}_\lambda(e_{ij})$ are piecewise affine maps as well. So, the infimum is
achieved when $\lambda$ converges to one of the vertices $p_1,\ldots, p_c.$ By Lemma~\ref{lm:s_pj} and since $\nu^{M^1}_{p_l}(m_i), \nu^{M^2}_{p_l}(m_j')\in \ima\nu_R,$ it follows that
$\nu^S_{p_l}(s_{ij})+\nu_{p_l}^{M^1}(m_i)-\nu_{p_l}^{M^2}(m_j')\in \ima\nu_R$ for all $l\in\{1,\ldots,c\}.$ Hence, by
Lemma~\ref{lm:e_ij} we get
\[
\nu_{M^1,M^2}(\Phi)=\inf_{\lambda\in\Lambda}\{{\nu_{M^1,M^2}}_\lambda(\Phi)\}= \min_{1\le l\le c}\inf_{i,j}\{\nu^S_{p_l}(s_{ij})+\nu_{p_l}^{M^1}(m_i)-\nu_{p_l}^{M^2}(m_j') \}\in\ima\nu_R.
\]
Therefore, $\ima\nu_{M^1,M^2}\subset\ima\nu_R.$

    Fix $0\ne\Psi\in\Hom(M^1,M^2),$ and let $r\in \ima\nu_R.$  
    Write $r_0=\nu_{M^1,M^2}(\Psi)\in \ima\nu_R,$ and let $r'\in R$ with $\nu_R(r')=r-r_0,$ where such $r'$ exists since $R$ is a field.
    Thus,
     \[\nu_{M^1,M^2}(r'\Psi)=\nu_R(r')+\nu_{M^1,M^2}(\Psi)=r.
     \] 
     Hence, $\ima\nu_R\subset\ima\nu_{M^1,M^2}.$
\end{proof}

The following two lemmas are essential for showing $\MF(W,w)$ is a normed DG category over a valued DGA.
\begin{lm}
    Let $(M^i,D_{M^i},\{\nu^{M^i}_\lambda\}_{\lambda\in \Lambda})$ for $i=1,2,3$ be  objects of $\MF(W,w).$ Then, for all $\Psi\in \Hom(M^1,M^2)$ and $\Phi\in \Hom(M^2,M^3)$ we have
    \[
    {\nu_{M^1,M^3}}_\lambda(\Phi \Psi)\ge {\nu_{M^2,M^3}}_\lambda(\Phi)+{\nu_{M^1,M^2}}_\lambda(\Psi), \quad \lambda\in\Lambda.
    \]
\end{lm}
\begin{proof}
This follows from Lemma~\ref{lm:weakly_normed_N}.
\end{proof}

\begin{lm}\label{lm_2.14}
    Let $(M^i,D_{M^i},\{\nu^{M^i}_\lambda\}_{\lambda\in \Lambda})$ for $i=1,2,3$ be  objects of $\MF(W,w).$ Then, for all $\Psi\in \Hom(M^1,M^2)$ and $\Phi\in \Hom(M^2,M^3)$ we have
    \[
    \nu_{M^1,M^3}(\Phi \Psi)\ge \nu_{M^2,M^3}(\Phi)+\nu_{M^1,M^2}(\Psi).
    \]
\end{lm}
\begin{proof}
    \begin{align*}
    \nu_{M^1,M^3}(\Phi\Psi)&=\inf_{\lambda\in \Lambda}\{{\nu_{M^1,M^3}}_\lambda(\Phi\Psi)\}\\
    &\ge \inf_{\lambda\in \Lambda}\{{\nu_{M^2,M^3}}_\lambda(\Phi)+{\nu_{M^1,M^2}}_\lambda(\Psi) \}\\
    &\ge \inf_{\lambda\in \Lambda}\{{\nu_{M^2,M^3}}_\lambda(\Phi) \}+ \inf_{\lambda\in \Lambda}\{{\nu_{M^1,M^2}}_\lambda(\Psi)\}\\
    &=\nu_{M^2,M^3}(\Phi)+\nu_{M^1,M^2}(\Psi).
\end{align*}
\end{proof}

In the case of $(M^1,D_{M^1},\{\nu^{M^1}_\lambda\}_{\lambda\in \Lambda})=(M^2,D_{M^2},\{\nu^{M^2}_\lambda\}_{\lambda\in \Lambda}),$ we will abbreviate ${\nu_{M^1}}_\lambda={\nu_{M^1,M^1}}_\lambda,$  $\nu_{M^1}=\nu_{M^1,M^1}$ and $\End(M^1)=\Hom(M^1,M^1).$ Note that 
$(\End(M), \delta, \id_M)$ is a unital DGA over $R$ 
for any object $(M,D_M,\{\nu^M_\lambda\}_{\lambda\in \Lambda})$ of $\MF(W,w).$

\begin{cor}\label{lm:end_val_dga}
    $(\End(M), \delta, \id_M, {\nu_M}_\lambda)$ is a normed unital DGA over $R$ for all $\lambda\in \Lambda.$
\end{cor}

\begin{cor}\label{lm:end_val_dga_2}
    $(\End(M), \delta, \id_M, \nu_M)$ is a normed unital DGA over $R.$ 
\end{cor}

\subsection{The toric construction}\label{Section_toric_cons}
\subsubsection{The valued Landau-Ginzburg model}\label{Section:toric_construction_sub}

The following definition is given in~\cite{da_Silva_toric} based on~\cite{Delzant}.
\begin{dfn}\label{df:Delzant}
    A \textbf{Delzant polytope} $\triangle$ in $\R^n$ is a polytope satisfying:
    \begin{enumerate}[label=(\arabic*)]
        \item there are $n$ edges meeting at each vertex,
        \item for any vertex $p$ of $\triangle,$ each edge meeting at $p$ is of the form $p+tu_i,$ where $t\ge0$ and $u_i\in\Z^n,$
        \item for each vertex, the corresponding $u_1,\ldots, u_n$ can be chosen to be a $\Z$ basis of $\Z^n.$
    \end{enumerate}
\end{dfn}
Let $\triangle\subset \R^n$ be a Delzant polytope with $d$ facets. Write
\[
K=\Z^n, \qquad L=\Hom(K, \Z), \qquad L_\R= L\otimes_\R \R,
\]  
and
\begin{gather*}
\widetilde{K}=\Z^d, \qquad \widetilde K_\Zt = K \otimes \Zt,  \\
 \widetilde{L}=\Hom(\widetilde{K}, \Z), \qquad \widetilde{L}_\Zt = \Hom(\widetilde K ,\Zt), \qquad \widetilde{L}_\R= \widetilde{L}\otimes_\R \R.
\end{gather*}
Let $f_1,\ldots, f_d \in \widetilde K$ denote the standard basis and let $e_1,\ldots, e_d \in \widetilde L$ denote the dual basis. Let the homomorphism $\tilde g : \widetilde K \to \Z$ be given by $\tilde g(f_i) = 2$ for $i = 1,\ldots,d.$
Let $v_1,\ldots, v_d\in  K$ be the primitive outward-pointing normal vectors to the facets of $\triangle$, and let
$(\lambda_1,\ldots, \lambda_d)\in \R^d$ such that 
\[
\triangle:= \{x\in L_\R| \ \langle x, v_i\rangle\le \lambda_i  \}.
\]
In this paper, we assume for simplicity that $\lambda_i\in\Q$ for all $i\in\{1,\ldots, d\}.$ Let 
\[
\lambda = \sum_i \lambda_i e_i \in \widetilde{L}_\R.
\]
Define 
\[
\pi: \widetilde{K}\to K,
\]
\[
f_i\mapsto v_i.
\] 
It is shown in~\cite{da_Silva_toric} that the Delzant condition implies $\pi$ is surjective.
We use this combinatorial data associated with $\triangle$ to construct a Landau-Ginzburg model as follows.

Let $J = \ker \pi \cong \Z^{d-n},$ and let $i:J\to \widetilde{K}$ denote the inclusion. So, 
\[
0\longrightarrow J \overset{i}{\longrightarrow} \widetilde{K}\overset{\pi}{\longrightarrow}K\longrightarrow0
\]
is a short exact sequence. Let $i^* : \widetilde L \to \Hom(J,\Z)$ be the dual homomorphism and let $i_\R^*:=i^*\otimes\id_\R:\widetilde{L}_\R\to \Hom(J,\R).$ Let 
\[
g_J : J \to 2\Z, \qquad h_J : J \to \R
\]
be the homomorphisms given by $g_J = i^* \tilde g$ and $h_J = i_\R^*\lambda$ respectively.  
Let
\[
J_\Zt : = J \otimes \Zt.
\]
Let $G \subset J \otimes \Q$ be the maximal subgroup such that $g_J$ extends to a homomorphism 
\[
g_G : G \to 2\Z.
\]
Let $h_G: G \to \R$ denote the unique extension of $h_J$ to $G.$

Define a quadratic form $q$ on  as follows. For $i,j \in \{1,\ldots, d\},$ let $\epsilon_{ij} \in \Zt$ be given by
\begin{equation}\label{eq:eps}
\epsilon_{ij} = 
\begin{cases}
0, & f_i + f_j \in J, \\
1, & \text{otherwise.}
\end{cases}
\end{equation}
Define a quadratic form $\tilde q : \tilde K \to \Zt$ by
\[
\tilde q(k) \equiv \sum_{i \leq j} \epsilon_{ij} e_i(k) e_j(k) \quad \pmod{2}, \qquad\qquad k \in \widetilde K.
\]
Define the quadratic form
\begin{equation*}
q : J \to \Zt
\end{equation*}
by $q = i^* \tilde q.$ More explicitly, for $\beta \in J$ we write $i(\beta) = \sum_i \beta_i f_i \in \tilde K.$ Then,
\begin{equation}\label{eq:qform}
q(\beta) \equiv \sum_{i \leq j} \epsilon_{ij} \beta_i \beta_j \pmod 2.
\end{equation}
By abuse of notation, we may also denote by $q$ the induced quadratic form on $J_\Zt.$
We say $\triangle$ is \textbf{orientable} if $g_J$ factors through $4\Z$ and \textbf{combinatorially relatively pin} if $q$ is a homomorphism. If both conditions hold, we say that $\triangle$ is \textbf{combinatorially relatively spin}. 

Define a graded valued field $\Rtr$ as follows. As a vector space,
\[
\Rtr:=\bigg\{\sum_{i=0}^\infty a_iT^{\beta_i} | \ a_i\in\R, \ \beta_i\in G, \ \lim_i h_G(\beta_i)=\infty \bigg\}.
\]
The multiplication is given by 
\[
T^{\beta_1}T^{\beta_2} =  T^{\beta_1 + \beta_2}.
\]
Equip $\Rtr$ with the valuation
\[
\nu_{\Rtr}(\sum_{i=0}^\infty a_iT^{\beta_i})=\inf_{a_i\ne0}\{h_G(\beta_i)\}.
\]
A grading on $\Rtr$ is defined by declaring $|T^{\beta}| = g_G(\beta).$ 

Let
\[
H_{\Zt} = \Hom(J,\Zt)
\]
and choose
$
\bgc \in H_\Zt.
$
We call $\bgc$ a \textbf{background class} because it is analogous from the perspective of mirror symmetry to the choice of a background class $\bgc^\vee$ in the definition of the Fukaya category of $\Xtr$. Namely, $\bgc^\vee \in H^2(\Xtr;\Zt)$ is a class that restricts to the second Stiefel-Whitney class of Lagrangian submanifolds that give objects of the Fukaya category. This analogy with the Fukaya category will not be used in the construction below.

We define the valued $\Rtr$-algebra $\Strb$ and the superpotential $\Wtrb \in \Strb$ as follows. 
Consider the group ring $\Rtr[\widetilde{K}].$ Let $z^A \in \Rtr[\widetilde{K}]$ denote the element corresponding to $A\in \widetilde{K}.$ Let $z_i \in \Rtr[\widetilde{K}]$ be the element corresponding to $f_i\in \widetilde{K}.$  Let $\Rtr\langle z_1,\ldots,z_d\rangle$ denote the Tate algebra over $\Rtr$ generated by $z_1,\ldots,z_d.$ So, we have an inclusion $\Rtr[\widetilde K] \subset \Rtr\langle z_1,\ldots,z_d\rangle.$
Define
\[
 \Strb:=\Rtr\langle z_1,\ldots,z_d\rangle/\langle (-1)^{\bgc(\beta)}z^{i(\beta)}T^{-\beta} - 1\,| \ \beta\in  J\rangle. 
\]
A grading on $\Strb$ is defined by declaring $|z_i|=2.$
Define a family of valuations on $\Strb$ by setting 
\[
\nu_\zeta^{\Strb}(z_i) =\lambda(f_i) - \langle\pi^*_\R(\zeta),  f_i\rangle  , \qquad \zeta \in \interior{\triangle},
\]
and requiring that the Schauder basis $\{\prod z_i^{d_i}\}$ be $\nu_\zeta^{\Strb}$-orthogonal. 
Define
\[
\Wtr = \Wtrb = \sum_{i = 1}^d z_i \in \Strb.
\]

For $\bgc \in H_{\Zt},$ we define the valued Landau-Ginzburg model $LG(\triangle,\bgc)$ by 
\[
LG(\triangle,\bgc) = \big((\Rtr,\nu_{\Rtr}),(\Strb,\{\nu^{\Strb}_\zeta\}_{\zeta\in \interior{\triangle}}) , \Wtr \big).
\]
If $\triangle$ is combinatorially relatively spin and not the $1$-simplex $\triangle^1,$ 
we define 
\[
LG(\triangle) = LG(\triangle,q).
\]
Let $\bgc^1 \in H_\Zt(\triangle^1) \cong \Zt$ be the non-trivial element. We define
\[
LG(\triangle^1) = LG(\triangle^1,\bgc^1).
\]
This exceptional definition is made because $\triangle^1$ is combinatorially spin on account of its low dimension unlike the higher dimensional simplicies $\triangle^n$ for $n \equiv 1 \pmod 4,$ which are only combinatorially relatively spin. We choose the non-trivial background class of $\triangle^1$ so that it behaves analogously to higher dimensions.
When $\triangle$ is not combinatorially relatively spin, we define $LG(\triangle) = LG(\triangle, 0).$ In all cases, we write
\[
LG(\triangle) = \big((\Rtr,\nu_{\Rtr}),(\Str,\{\nu^{\Str}_\zeta\}_{\zeta\in \interior{\triangle}}) , \Wtr \big).
\]

Let $w\in \Rtr$ satisfy $\nu_{\Rtr}(w)\ge\sup_{\zeta\in\interior{\triangle}}\nu_\zeta^{\Strb}(\Wtr).$ We call such $w$ \textbf{admissible}.
The valued Landau-Ginzburg model $LG(\triangle,\bgc)$ gives rise to a normed matrix factorization category $\MF(\Wtrb,w)$ as defined in Section~\ref{section:1.2.1}. The matrix factorization category associated to $LG(\triangle)$ is denoted by $\MF(\Wtr,w)$. 

We conclude with a lemma that will be needed later in the paper. Define a norm $\nu_\Strb$ on $\Strb$
by
\[
\nu_\Strb(x) = \inf_{\zeta \in \interior{\triangle}} \nu_\zeta^\Strb(x).
\]
\begin{lm}\label{lm:u_coordinates}
Let $n = \dim \triangle.$ After possibly permuting the labels of $z_1,\ldots,z_d,$ we may assume that the canonical homomorphism $\Rtr[z_1,z_1^{-1},\ldots,z_n,z_n^{-1}] \to \Strb$ is injective and the image is dense in the $\nu_{\Strb}$ topology. 
\end{lm}
\begin{proof}
We show that after permuting the labels of $v_1,\ldots,v_d$, we may assume that $v_1,\ldots,v_d$ from a basis for $K.$ Indeed, it follows from the Delzant condition of Definition~\ref{df:Delzant} that there are $n$ facets incident to each vertex of $\triangle.$ Choose $p$ a vertex of $\triangle \subset L_\R$ and permute the vectors $v_1,\ldots,v_d$ so that $v_1,\ldots,v_n$ are the primitive outward-pointing normal vectors of the facets incident to $p.$ We show that $v_1,\ldots,v_n$ are a basis of $K.$ Indeed, let $u_1,\ldots,u_n \in L$ be a basis such that the edges incident to $p$ have the form $p + tu_i$ with $t \geq 0$ as in Definition~\ref{df:Delzant}. The facets incident to $p$ are contained in the hyperplanes spanned by subsets of $\{u_1,\ldots,u_n\}$ of size $n-1.$ So, after possibly reordering $u_1,\ldots,u_n,$ we may assume that $\langle u_i,v_j \rangle = 0$ for $i \neq j.$ Since $u_1,\ldots,u_n$ is a basis and $v_j \neq 0,$ it follows that $\langle u_i,v_i \rangle \neq 0$ for all $i.$ Thus, after possibly renormalizing $v_1,\ldots,v_n$, we obtain a basis of $K$ dual to $u_1,\ldots,u_n$. But since the vectors $v_i$ were chosen to be primitive, such renormalization is not needed.

Given the preceding, the relations defining $\Strb$ allow us to express all $z_i$ as Laurent monomials in $z_1,\ldots,z_n.$ The lemma follows.
\end{proof}

\subsubsection{The Dirac matrix factorization} \label{ssec:DiracMF}
Let $\bgc' \in H_\Zt.$ We say that $\bgc'$ is \textbf{free} if it is of the form
\begin{equation}\label{eq:free}
\bgc'(\beta) = \sum_{\substack{i < j \\f_i + f_j \in J}} \bgc'(f_i + f_j) \beta_i \beta_j.
\end{equation}
For example, any $\bgc' \in H_\Zt$ is free for $\triangle = (\triangle^1)^n.$
Suppose $\triangle$ is combinatorially relatively spin, and set $\bgc = q + \bgc',$ where $\bgc' \in H_\Zt$ is free. We construct an admissible $\wtrb \in \Rtr$ and an object $(\Mtrb,\Dtrb)$ in the category $\MF(\Wtrb,\wtrb),$ called the Dirac matrix factorization. This construction provides us with a source of examples for computing the numerical invariants $N_{d,k}$ of Definition~\ref{def:numerical_invariants}. 

First, we define a Clifford algebra over $\Strb$. Let $\widetilde L_S = \widetilde L \otimes \Strb.$ 
Let $T(\widetilde L_S)$ denote the tensor algebra. Define a grading on $T(\widetilde L_S)$ by setting $|e_i| = 1.$ 
For $\zeta\in\interior{\triangle}$ define a $\nu_\zeta^{\Strb}$ valuation $\nu_\zeta^{T}$ on $T(\widetilde L_S)$ by setting
\[
\nu_\zeta^{T}(e_i)=\frac{1}{2}\big( \lambda(f_i)-\langle\pi^*_\R(\zeta),  f_i\rangle  \big)
\]
and requiring that the Schauder basis of $T(\widetilde L_S)$ given by pure tensors be $\nu_\zeta^{T}$-orthogonal.
Let
\begin{equation}\label{eq:vareps}
\varepsilon_{ij} = 
\begin{cases}
\sigma'(f_i + f_j), & f_i + f_j \in J, \\
1, & \text{otherwise.}
\end{cases}
\end{equation}
Define a homogeneous ideal $P \subset T(\widetilde L_S)$ by
\begin{equation}\label{eq:P}
P = \langle e_i \otimes e_j + (-1)^{1 + \varepsilon_{ij}}e_j \otimes e_i = z_i \delta_{ij} \,|\, i,j = 1,\ldots,d\rangle,
\end{equation}
and define
\[
\cclb:= T(\widetilde L_S)/ P.
\]
Since $\nu_\zeta^T(e_i \otimes e_j) = \nu_\zeta^T(e_j \otimes e_i),$ and if $i = j,$ these valuations equal $ \nu_\zeta^T(z_i),$ the valuation $\nu_\zeta^T$ descends to a valuation on $\cclb$ which we denote by $\nu_\zeta^\cclb.$

Next, we define $\Mtrb$ as a left $\cclb$ module obtained as a quotient of $\cclb$ by a left ideal. To define this ideal, we need the notion of a combinatorial relative spin structure.
Since $\triangle$ is orientable, for each $\beta \in J$ there is an element $\beta/2 \in G.$ 
For $\beta \in J,$ write $i(\beta)=(\beta_1,\ldots,\beta_d)\in \widetilde{K}$ and $e^{i(\beta)}:=e_1^{\beta_1}\cdots e_d^{\beta_d} \in \cclb.$ 
For $\s : J \to \Zt$ a map of sets, define a map of sets 
\[
\tilde \chi_\s : J \to \cclb^\times
\]
by $\tilde\chi_\s(\beta) = (-1)^{\s(\beta)}T^{-\frac{\beta}{2}}e^{i(\beta)}$. 

We say $\s$ is a \textbf{combinatorial relative spin structure} for $\triangle$ if $\tilde \chi_\s$ is a homomorphism, and we denote by $\Spt$ the set of combinatorial relative spin structures. For $\beta \in J$, let $[\beta] \in J_\Zt$ denote the reduction of $\beta$ modulo $2.$ There is an action $H_\Zt \times \Spt \to \Spt$ given by $(h \cdot \s)(\beta) = h([\beta]) + \s(\beta).$ The following lemma gives an analogy between $\Spt$ and the set of relative spin structures on $\Xtrr,$ which is an $H^1(\Xtrr;\Zt)$ torsor.
\begin{lm}\label{lm:mrs}
Under the assumption that $\triangle$ is combinatorially relatively spin, the set of combinatorial relative spin structures $\Spt$ is an $H_\Zt$ torsor. Moreover, for each $\s \in \Spt,$ the map $\tilde\chi_\s$ factors through a map $\chi_\s : J_{\Zt} \to \cclb^\times.$
\end{lm}
\begin{proof}
Define a bilinear form
\[
\langle \cdot, \cdot \rangle : J_\Zt \longrightarrow \Zt
\]
by
\[
\langle \beta,\beta' \rangle \equiv \sum_{i \neq j} \varepsilon_{ij} \beta_i \beta'_j \pmod 2, \qquad \beta,\beta' \in J_\Zt.  
\]
Since $\bgc'$ is free, it is given by~\eqref{eq:free}. Thus, recalling the formula~\eqref{eq:qform} for $q,$ and comparing the formula~\eqref{eq:eps} for $\epsilon_{ij}$ with the formula~\eqref{eq:vareps} for $\varepsilon_{ij}$, we obtain
\[
\bgc(\beta) = q(\beta) + \bgc'(\beta)  = 
\sum_{i \leq j} \varepsilon_{ij} \beta_i \beta'_j \pmod 2, \qquad \beta,\beta' \in J_\Zt.
\]
It follows that
\[
\bgc(\beta + \beta') + \bgc(\beta) + \bgc(\beta') = \langle \beta,\beta' \rangle.
\]
Since $\bgc$ is a homomorphism, $\langle \cdot,\cdot \rangle$ vanishes. On the other hand, 
\[
e^{i(\beta)} e^{i(\beta')} = (-1)^{\langle \beta,\beta' \rangle} e^{i(\beta')} e^{i(\beta)}.
\]
Thus, $e^{i(\beta)}$ and $e^{i(\beta')}$ commute for all $\beta,\beta' \in J_{\Zt}.$ Moreover, 
since $\triangle$ is orientable,
\[
\sum_i \beta_i^2 \equiv \sum_i \beta_i = \frac{1}{2}g_J(\beta) \equiv 0 \pmod 2,
\]
so
\[
\bgc(\beta) = \sum_{i < j} \varepsilon_{ij} \beta_i \beta_j + \sum_{i} \beta_i^2 \equiv \sum_{i<j} \varepsilon_{ij}\beta_i \beta_j \pmod 2.
\] 
Therefore,
\begin{equation}\label{eq:e2}
\left(e^{i(\beta)}\right)^2 = e^{i(\beta)}\cdot e^{i(\beta)} = (-1)^{\bgc(\beta)} z^{i(\beta)} = T^\beta.
\end{equation}

We proceed to show the existence of a combinatorial relative spin structure $\s$. Choose a basis $\gamma_1,\ldots,\gamma_{d-n}$ for $J.$ Choose $\s(\gamma_i) \in \Zt$ arbitrarily. For $\beta = \sum a_i \gamma_i \in J,$ define $\s(\beta)$ by
\[
(-1)^{\s(\beta)}T^{-\frac{\beta}{2}}e^{i(\beta)} = \prod_{i = 1}^{d-n} \left((-1)^{\s(\gamma_i)}T^{-\frac{\gamma_i}{2}}e^{i(\gamma_i)}\right)^{a_i}.
\]
We prove that for such $\s$ the map $\tilde\chi_\s$ is a homomorphism. Indeed, since $e^{i(\beta)}$ and $e^{i(\beta')}$ commute for $\beta,\beta' \in J,$ we have
\begin{align*}
\tilde\chi_\s(\beta)\cdot \tilde\chi_\s(\beta') &= \prod_{i = 1}^{d-n} \left((-1)^{\s(\gamma_i)}T^{-\frac{\gamma_i}{2}}e^{i(\gamma_i)}\right)^{a_i} \cdot \prod_{i = 1}^{d-n}\left( (-1)^{\s(\gamma_i)}T^{-\frac{\gamma_i}{2}}e^{i(\gamma_i)} \right)^{a_i'}  \\
& = \prod_{i = 1}^{d-n} \left((-1)^{\s(\gamma_i)}T^{-\frac{\gamma_i}{2}}e^{i(\gamma_i)}\right)^{a_i + a_i'} = \tilde\chi_\s(\beta + \beta').
\end{align*}

To prove that $\Spt$ is an $H_\Zt$ torsor, it remains to show that if $\s, \s' \in \Spt,$ then the map $\tilde h : J \to \Zt$ defined by
\[
\tilde h(\beta) = \s(\beta) + \s'(\beta)
\]
factors through a homomorphism $h : J_\Zt \to \Zt.$ Indeed, it suffices to show that $\tilde h$ is a homomorphism. This follows because by~\eqref{eq:e2} we have
\begin{equation*}
(-1)^{\tilde h(\beta + \beta')} = \tilde\chi_\s(\beta + \beta') \cdot \tilde\chi_{\s'}(\beta + \beta')
= \tilde\chi_\s(\beta) \cdot \tilde\chi_{\s'}(\beta) \cdot \tilde\chi_\s(\beta') \cdot \tilde\chi_{\s'}(\beta') = (-1)^{\tilde h(\beta) + \tilde h(\beta')}.  
\end{equation*}

Finally, for any combinatorial relative spin structure $\s \in \Spt,$ we show that $\tilde \chi_\s$ factors through a homomorphism $\chi_\s : J_\Zt \to \cclb^\times.$ Indeed, by~\eqref{eq:e2} we have
\[
\tilde \chi_\s(2\beta) = (\tilde \chi(\beta))^2 = \left((-1)^{\s(\beta)}T^{-\frac{\beta}{2}}e^{i(\beta)} \right)^2 = T^{-\beta}(e^{i(\beta)})^2 = 1.
\]
\end{proof}
Lemma~\ref{lm:mrs} allows us to choose a combinatorial relative spin structure~$\s.$ 
The choice will effect the construction of $(\Mtrb,\Dtrb)$, but we supress $\s$ from the notation to avoid it becoming too cumbersome. In the calculations in Section~\ref{Section:example} below, we specify $\s$ explicitly.
Define a left ideal $\Ltrb \subset \cclb$ by
\begin{equation}\label{eq:L}
\Ltrb = \langle (-1)^{\s(\beta)}T^{-\beta/2}e^{i(\beta)}- 1 \,| \,\  \beta\in J\rangle.
\end{equation}
Define a left $\cclb$ module $\Mtrb$ by
\[
\Mtrb:=\cclb/\Ltrb.
\]
It follows that $\Mtrb$ is a free $\Str$ module. Indeed, one uses an averaging procedure to construct invariant vectors for the representation of the finite group $J_\Zt$ on $\cclb$ given by right multiplication by $\chi_\s$. These invariant vectors descend to free generators of $\Mtrb.$

Define $\Dtrb\in\End(\Mtrb)$ to be left multiplication by $e_0+\ldots+e_n$ and define
\[
\wtrb = -2\sum_{\substack{\beta \in J \\ \exists i, j: \, \beta = f_i + f_j  \\ \bgc(\beta) = 0 }} (-1)^{\s(\beta)} T^{\frac{\beta}{2}}.
\]
From the perspective of mirror symmetry, $\wtr$ counts Maslov $2$ disks in $\Xtr$ with boundary on $\Xtrr,$ which come in pairs, one pair for each class in $H_2(\Xtr;\Z) \cong J.$ 
\begin{lm}
We have $(\Dtrb)^2 = \Wtrb - \wtrb.$
\end{lm}
\begin{proof}
By the definition~\eqref{eq:P} of the ideal $P$ that gives the relations for $\cclb$, we have $e_i^2 = z_i$ for all $i$ and if $i \neq j,$ then $e_i$ and $e_j$ anti-commute unless $f_i + f_j \in J,$ in which case they commute. So,
\begin{equation*}
(\Dtrb)^2 = \sum_i z_i  + 2\sum_{\substack{i,j \\ \bgc(f_i+f_j) = 0}} e_i e_j.
\end{equation*}
By the definition~\eqref{eq:L} of the ideal $\Ltrb$ from the definition of $\Mtrb,$ we have
\[
\sum_{\substack{i,j \\ \bgc(f_i + f_j) = 0}} e_i e_j = \sum_{\substack{\beta \in J \\ \exists i, j: \, \beta = f_i + f_j  \\ \bgc(\beta) = 0 }} (-1)^{\s(\beta)} T^{\frac{\beta}{2}}.
\]
The lemma follows by combining the preceding two equations.
\end{proof}
In light of the preceding, $(\Mtrb,\Dtrb)$ is an object of $\MF(\Wtrb,\wtrb).$ As in the case of $LG(\triangle)$ we set 
\[
\wtr = \wtrb, \qquad \Mtr = \Mtrb, \qquad \Dtr = \Dtrb, \qquad \ccl = \cclb,
\]
for $\bgc = q$ in the case $\triangle \neq \triangle^1$ and $\bgc$ non-trivial in the case $\triangle = \triangle^1.$

We conclude with a technical lemma that will be helpful in performing calculations in $\End(\Mtr).$ Since
$\nu_\zeta^{\ccl}(e^{i(\beta)})=\nu_R(T^{\beta/2})$,
it follows that $\nu_\zeta^{\ccl}$ descends to a valuation on $\Mtr$, which we denote by $\nu_\zeta^{\Mtr}.$
The family $\{\nu^{\Mtr}_\zeta\}_{\zeta\in\interior{\triangle}}$
induces 
a family of valuations $\{{\nu_{\Mtr}}_\zeta\}_{\zeta\in\interior{\triangle}}$ on the algebra $\End(\Mtr)$ as defined in~\eqref{eq:valuation_lambda_hom}.
\begin{lm}
${\nu_{\Mtr}}_\zeta(x)=\nu_\zeta^{\ccl}(x),$ for all $x\in \ccl$ and $\zeta\in\interior{\triangle}.$
\end{lm}
\begin{proof}
   \begin{multline*}
{\nu_{\Mtr}}_\zeta(x)=\inf_{m\in \Mtr\backslash\{0\}}\{\nu_\zeta^{\Mtr}(xm)-\nu_\zeta^{\Mtr}(m)\} =\\
=\inf_{m\in \Mtr\backslash\{0\}}\{\nu_\zeta^{\ccl}(x)+\nu_\zeta^{\Mtr}(m)-\nu_\zeta^{\Mtr}(m)\}=\nu_\zeta^{\ccl}(x).
\end{multline*}
\end{proof}

\section{The cyclic complex}\label{Section:hochschild}
  
\subsection{Definitions}\label{subsection_cy}  
Let $(R,d_R,\nu_R)$ be a commutative valued DGA over $\K.$
In the following, $\mA$ stands for a small normed DG category over $R.$ 
We follow Shklyarov~\cite{Shklyarov}.

Recall the definition of the Hochschild complex $\CC_*(\mA)$ defined in Section~\ref{section:normed_cy_str}.
Let $d$ stands for the differential on the $\Hom$-complexes of $\mA.$ For an element $a_l\otimes\ldots\otimes a_1\in\CC_l(\mA)$ such that $a_i$ are homogeneous elements, define $\epsilon_i:=\sum_{j\ge i}|sa_j|.$
The differential of the complex $\CC_*(\mA)$ is given by $\dcc:=b(\mu_1)+b(\mu_2)$, where
\begin{multline*}
b(\mu_1)(a_l\otimes a_{l-1}\otimes\ldots\otimes a_1)=d( a_l) \otimes a_{l-1}\otimes\ldots\otimes a_1+\\
+\sum_{i=1}^{l-1}(-1)^{\epsilon_{i+1}}a_l\otimes a_{l-1}\otimes\ldots\otimes d_A(a_i)\otimes\ldots\otimes a_1,
\end{multline*}
\begin{multline}\label{b(mu)_definition}
    b(\mu_2)(a_l\otimes a_{l-1}\otimes\ldots\otimes a_1)=(-1)^{|a_l|}a_la_{l-1}\otimes a_{l-2}\otimes\ldots\otimes a_1-\\
    - \sum_{i=1}^{l-2}(-1)^{\epsilon_{i+1}}a_l\otimes a_{l-1}\otimes\ldots\otimes a_{i+1}a_i\otimes\ldots\otimes a_1-(-1)^{|sa_1|(\epsilon_2+1)}a_1a_l\otimes a_{l-1}\otimes\ldots\otimes a_2.
\end{multline}
Let $b'(\mu_2)$ denote the operator on $\CC_*(\mA)$ given by the second line in~\eqref{b(mu)_definition}, and recall the definition of the cyclic permutation $t$~\eqref{eq:cyclic_pernutation}.
The following lemma is proved in~\cite{Loday}.
\begin{lm}\label{lm:commutator_mu_1}
    \[
    t\circ b(\mu_1)=b(\mu_1)\circ t,\quad b(\mu_2)(\id-t)=(\id-t)b'(\mu_2).\]
    
\end{lm}
By Lemma~\ref{lm:commutator_mu_1} $\ima (\id-t)$ is $\dcc$-invariant, so the cyclic complex $\CC^\lambda_*(\mA)$ defined in Section~\ref{section:normed_cy_str} is well defined. 

Consider $\CC_*(\mA)$ as an $R$-module. We define a valuation $\nu$ on $\CC_*(\mA)$ as follows.
Let $a:=a_l\otimes\ldots \otimes a_1\in \CC_l(\mA),$ where $a_l\in(\Hom(X^l, X^1),\nu_l)$ and $a_i\in(\Hom(X^i, X^{i+1}),\nu_i)$ for $1\le i\le l-1,$ define $\nu(a):=\sum_{i=1}^l\nu_i(a_i).$
Let $\alpha_i\in \CC_{l_i}(\mA),$ define
$\nu(\sum_{i=1}^\infty\alpha_i)=\inf_i\{\nu(\alpha_i)\}.$ 
Since the Hom-complexes of $\mA$ are valued modules over $R,$ it follows that $(\CC_*(\mA),\nu)$ is a valued module over $R$ as well.
Note that the valuation $\nu$ descends to the cyclic complex $\CC_*^\lambda(\mA).$ 

Let $n\in\Z_{>0}$, and let $\triangle$ be a Delzant polytope such that $n=\dim\triangle$. Let $LG(\triangle,\sigma)$ be a valued Landau-Ginzburg model as defined in Section~\ref{Section:toric_construction_sub}, and let  $\MF(W_{\triangle,\sigma},w)$ be the associated normed matrix factorization category, where $w$ is admissible. 
Let $z_1,\ldots, z_n$ be the coordinates of Lemma~\ref{lm:u_coordinates}.
Abbreviate $\partial_i D=\partial_{z_i}D,$ and $\partial_i \Wtr=\partial_{z_i}\Wtr.$
The following theorem is mainly based on
Theorem 2.4 in~\cite{Shklyarov}, and proved in a more general setting in~\cite{Shklyarov}.
\begin{thm}\label{Theta}
Let $\Omega=\frac{dz_1\wedge\ldots\wedge dz_n}{z_1\cdots z_n}.$ The following operator $\Theta=\Theta_\Omega$ defines an $\infty$-trace on $\MF(\Wtr,w)$
\[
\Theta=\sum_{l\ge1}\Theta_l: \CC_*^\lambda(\MF(\Wtr,w))\to \Rtr
\]
such that $ \Theta_l=\sum_{x\in C_{\Wtr}}\Theta_{l,x}$ where
\begin{align*}
    \Theta_{l,x}(\Phi_l\otimes\cdots\otimes \Phi_1)&:=\frac{1}{(n+l-1)!}\sum_{\substack{k_1+\ldots+k_l=n-1\\ k_1,\ldots, k_l\ge 0}}(-1)^{k_1\epsilon_1+\ldots+k_l\epsilon_l}\sum_{i=1}^n(-1)^i\\
    &\sum_{\substack{r_1+\ldots+r_l=l\\ r_j\ge0 , j\ne i, r_i\ge1}}r_1!\cdots r_n!\sum_{(i^{(1)},\ldots,i^{(l)})\in \Lambda^l_n(r_1,\ldots ,r_n)}\\
    &\sum_{(j_1^{(l)},\ldots,j_{k_1}^{(l)},\ldots,j_1^{(1)},\ldots, j_{k_l}^{(1)})\in S^i_n}\mathrm{sgn}(j_1^{(l)},\ldots,j_{k_1}^{(l)},\ldots,j_1^{(1)},\ldots, j_{k_l}^{(1)})\\
    &\mathrm{Res}_x\big[ \frac{\mathrm{str}\big( \Phi_l\partial_{i^{(l)}}D^l\partial_{j_1^{(l)}}D^l\cdots \partial_{j_{k_l}^{(l)}}D^l\cdots \Phi_1\partial_{i^{(1)}}D^1\partial_{j_1^{(1)}}D^1\cdots \partial_{j_{k_1}^{(1)}}D^1\big)\wedge \Omega}{(\partial_1\Wtr)^{r_1+1}\cdots (\partial_i\Wtr)^{r_i}\cdots(\partial_n\Wtr)^{r_n+1}}  \big].
\end{align*}
\end{thm}
In this formula
\begin{enumerate}[label=(\arabic*)]
    \item 
    $\Phi_l\in \Hom(M^l,M^1)$ and $\Phi_i\in \Hom(M^i,M^{i+1})$, $1\le i\le l-1$,
    \item $\epsilon_i :=\sum_{j\ge i}|s\Phi_j|,$
    \item $\Lambda^l_n(r_1,\ldots ,r_n)$ denotes the subset in $\{1,\ldots, n\}^l$ of those multi-indices $(i^{(1)},\ldots, i^{(l)})$ that contains precisely $r_1$ copies of $1,$ $r_2$ copies of $2$ etc,
    \item $S^i_n$ where $1\le i\le n$ stands for the set of all permutations of $\{1,\ldots , n\}\backslash\{i\}.$ Given an element $(j_1,\ldots, j_{n-1})\in S^i_n,$ $\mathrm{sgn}(j_1,\ldots, j_{n-1})$ denotes the sign of the corresponding permutation,
    \item $C_{\Wtr}$ is the set of critical points of $\Wtr,$ and $\mathrm{str}$ is the supertrace,
    \item $\mathrm{Res}_x$ stands for the local residue at $x,$ that is
    \[
    \mathrm{Res}_x\bigg( \frac{\omega}{(\partial_1\Wtr)^{s_1}\cdots (\partial_n\Wtr)^{s_n}}\bigg):=\frac{1}{(2\pi i)^n}\int_\Gamma \frac{\omega}{(\partial_1\Wtr)^{s_1}\cdots (\partial_n\Wtr)^{s_n}},
    \]
    where $\Gamma$ is a cycle in a small neighborhood of $x$ defined by $\Gamma=\{|\partial_i\Wtr|=\epsilon_i\ll 1 \ \forall i\}$ and oriented by $d(\mathrm{arg}\partial_1 \Wtr)\wedge\ldots\wedge d(\mathrm{arg}\partial_n \Wtr). $
\end{enumerate}

\subsection{The exponent of a bounding cochain}\label{Section_subcomplex}
Let $(R,d_R,\nu_R)$ be a commutative valued DGA over $\K,$ and let $(A,d_A, 1_A,\nu_A)$ be a normed unital DGA over $R$.
Recall that we denote by $\CC_*(A)$ and $\CC_*^\lambda(A)$ the Hochschild and cyclic complexes of the category consisting of one object with endomorphism algebra $A.$
We identify $\CC_*(R)$ with the subcomplex of $\CC_*(A)$ generated by the elements $1_A^{\otimes l}$ where $l\ge1.$
\begin{rem}\label{rem:cyclic_cx_r} 
    Since $t(1_A^{\otimes 2l})=-1_A^{\otimes 2l},$ and $t(1_A^{\otimes 2l-1})=1_A^{\otimes 2l-1}$ for all $l\ge1,$ it follows that  
\[\CC_l^\lambda(R)=
    \left\{\begin{array}{ll}
        0, &  l\in 2\Z_{>0},\\
       \langle 1_A^{\otimes l}\rangle_R, & \mathrm{otherwise}.\\
     \end{array} \right.\]
\end{rem}

By abuse of notation, we will denote by $a_l\otimes\ldots\otimes a_1$ where $a_i\in A,$ an element in $\CC_*(A)$ and the class it induces in the cyclic complex $\CC_*^\lambda(A).$ The meaning will be clear from the context.

Let
\[
T(A):=\bigoplus_{k\ge0}A^{\otimes k}.
\]
For $a\in A$ such that $\nu_A(a)>0,$ abbreviate
\begin{equation}\label{eq:exponent}
\exp(a):=1_A \oplus a \oplus\frac{a\otimes a}{2}\oplus\frac{a\otimes a\otimes a}{3}\oplus\ldots \in \CC_*^\lambda(A).    
\end{equation}
For a valued or a normed ring $(\Upsilon,\nu_\Upsilon),$ let $F^E$ denote the filtration on $\Upsilon$ defined by $\nu_\Upsilon.$ That is,
\begin{equation}\label{eq:filtration:def}
    \nu_\Upsilon(\alpha)>E \iff \alpha\in F^E \Upsilon.
\end{equation}
Let $E\in\ima\nu_R.$ We will denote by
$\Mmm(A,c)_E$ the set of all bounding cochains modulo $F^EA$ with a fixed $c.$ So, if $b\in\Mmm(A,c)_E,$ then $d_A(b)-b^2\equiv c\cdot 1_A \pmod{F^EA}.$ Note that if $E=\infty,$ then $b$ is a bounding cochain in the usual sense of Definition~\ref{dfn:bc}.

\begin{lm}\label{lm:exp(b)_q}
Let $(R,d_R,\nu_R)$ be a commutative valued DGA over $\K,$ and
let
 $(A,d_A, 1_A, \nu_A)$ be a normed unital DGA over $R.$  Let $b,q\in A$ with $\nu_A(b),\nu_A(q)>0$ such that $d_A(b)-b^2\equiv q \pmod{F^EA}.$ Then,
    \[
    \dcc(\exp(b))=\sum_{l=1}^\infty (d_A( b)-b^2)\otimes b^{\otimes l-1}\equiv\sum_{l=1}^\infty q\otimes b^{\otimes l-1}\in \CC_*^\lambda(A) \pmod{F^EA}.
    \]   
\end{lm}
\begin{proof}
Since $|b|=1,$ it follows that $|sb|=0.$ Hence, for every $l\ge2 $ and $ 1\le i\le l-1$ we get
\[
t(b^{\otimes i}\otimes d_A( b)\otimes b^{\otimes l-i-1})= b^{\otimes i-1}\otimes d_A (b)\otimes b^{\otimes l-i},\quad
t(b^{\otimes i}\otimes b^2\otimes b^{\otimes l-i-1})= b^{\otimes i-1}\otimes b^2\otimes b^{\otimes l-i}  .
\]
Thus, for $b^{\otimes l}\in\CC_l^\lambda(A),$ we obtain
\begin{align*}
b(\mu_1)(b^{\otimes l})&=d_A( b)\otimes b^{\otimes l-1} +\sum_{i=1}^{l-1}(-1)^{\epsilon_{i+1}}b^{\otimes l-i}\otimes d_A( b)\otimes b^{\otimes i-1}\\
&=d_A( b)\otimes b^{\otimes l-1} +\sum_{i=1}^{l-1}b^{\otimes l-i}\otimes d_A( b)\otimes b^{\otimes i-1}\\
&=d_A( b)\otimes b^{\otimes l-1}+ \sum_{i=1}^{l-1}d_A (b) \otimes b^{\otimes l-1}\\
&= l\cdot d_A( b) \otimes b^{\otimes l-1},
\end{align*}
and
\begin{align*}
    b(\mu_2)(b^{\otimes l})&= -b^2\otimes b^{\otimes l-2}-\sum_{i=1}^{l-2}(-1)^{\epsilon_{i+1}}b^{\otimes l-i-2}\otimes b^2\otimes b^{\otimes i-1}- b^2\otimes b^{\otimes l-2}\\
    &=-b^2\otimes b^{\otimes l-2}-\sum_{i=1}^{l-2}b^{\otimes l-i-2}\otimes b^2\otimes b^{\otimes i-1}- b^2\otimes b^{\otimes l-2}\\
    &=-l\cdot b^2\otimes b^{\otimes l-2}.
\end{align*}
Hence,
\begin{align*}
    \dcc(\exp(b))&=(b(\mu_1)+b(\mu_2))\sum_{l=1}^\infty \frac{b^{\otimes l}}{l}\\
    &=\sum_{l=1}^\infty \frac{b(\mu_1)(b^{\otimes l})}{l}+\sum_{l=2}^\infty \frac{b(\mu_2)(b^{\otimes l})}{l}\\
    &= \sum_{l=1}^\infty d_A(b)\otimes b^{\otimes l-1}- \sum_{l=2}^\infty  b^2\otimes b^{\otimes l-2}\\
    &=\sum_{l=1}^\infty (d_A (b)-b^2)\otimes b^{\otimes l-1}.
\end{align*}
Since $d_A(b)-b^2\equiv q \pmod{F^EA},$ it follows that $\dcc(\exp(b))\equiv \sum_{l=1}^\infty q\otimes b^{\otimes l-1}\pmod{F^EA}.$ 
\end{proof}
\begin{cor}\label{lm_b(exp)_is_closed}
 Let $(R,d_R,\nu_R)$ be a commutative valued DGA over $\K,$ and
let
 $(A,d_A, 1_A, \nu_A)$ be a normed unital DGA over $R.$  Let $b\in \Mmm(A,c)_E.$ Then,
    \[
    \dcc(\exp(b))=\sum_{l=1}^\infty (d_A( b)-b^2)\otimes b^{\otimes l-1}\equiv\sum_{l=1}^\infty c\cdot1_A\otimes b^{\otimes l-1}\in \CC_*^\lambda(A) \pmod{F^EA}.
    \]   
\end{cor}
\begin{proof}
    Take $q:=c\cdot 1_A$ in Lemma~\ref{lm:exp(b)_q}.
\end{proof}
\subsection{Constructing the homotopy}\label{section:defining_the_homotopy}
Recall the definition of the element $y_b\in\mR^\lambda_*$ in Section~\ref{section:1.2.2}. 
Our goal is to establish the existence of the element $y_b$, and provide a canonical choice, which will be used consistently throughout this paper. 
To achieve this, we begin by defining a homotopy $H$ on a quotient complex of $\mR^\lambda_*.$ This homotopy will play a central role in determining the canonical choice of $y_b.$  In this section, we work in a more general setting than the one introduced in~\ref{section:1.2.2}.  

Let $(R,d_R,\nu_R)$ be a commutative valued DGA over $\K,$ and let $(A,d_A,1_A,\nu_A)$ be a normed unital DGA over $R.$ 
In this section we denote by ${\fCC}_*(A)$ the Hochschild complex of $A$ which is not complete, that is, ${\fCC}_*(A)$ consists finite sums of elements $a_l\otimes\ldots\otimes a_1,$ where $a_i\in A.$ Hence, the completion of ${\fCC}_*(A)$ is the complex $\CC_*(A).$  
Let ${\fmR}_*\subset{\fCC}_*(A)$ be the subcomplex generated by elements of the form $a_l\otimes\ldots\otimes a_1$ such that there exists $1\le i\le l$ such that $a_i\in R.$
Let ${\fmR}_*^\lambda\subset {\fCC}_*^\lambda(A)$ denote the subcomplex ${\fmR}_*/(\ima(\id-t)),$ and let ${\fmQ}_*^\lambda$ denote the complex ${\fmR}_*^\lambda/{\fCC}_*^\lambda(R).$ We denote by $\mR_*^\lambda,\mQ_*^\lambda$ the completions of ${\fmR}_*^\lambda,{\fmQ}_*^\lambda$ with respect to the valuation $\nu$ defined in Section~\ref{subsection_cy}.
The existence of $y_b$ proved in Corollary~\ref{cor:existence_y_b} relies on the following proposition. 
\begin{prop}\label{lm:null_homotopic}
    There exists an operator $H:\mQ_*^\lambda\to \mQ_*^\lambda$ satisfying $\dcc\circ H+ H\circ\dcc=\id.$
\end{prop}
We begin by proving a few lemmas that will be useful in the construction of the operator $H.$

Fix a projection $\pi:A\to R.$
Define a map $\widetilde{h}:{\fCC}_*(A)\to {\fCC}_*(A)$ acting on homogeneous elements as \[
\widetilde{h}(a_l\otimes\ldots\otimes a_1)= 1\otimes \pi(a_l)\otimes a_{l-1}\otimes\ldots\otimes a_1.
\]
Let $a=a_l\otimes\ldots\otimes a_1\in {\fCC}_l(A).$ Let $\e(a)$ be the number of $a_i$'s satisfying $a_i\in R$ if $a\ne0,$ and $\infty$ if $a=0.$ In this subsection we use the notation $\epsilon_i(a)$ instead of $\epsilon_i.$
\begin{lm}\label{lm:a'_mu_1}
Let $0\ne a=a_l\otimes\ldots\otimes a_1\in  {\fCC}_l(A),$ where $a_i$ are homogeneous elements. 
Then, 
\[\e\bigg(\big(\widetilde{h}\circ b(\mu_1)+b(\mu_1)\circ \widetilde{h}\big)(a)\bigg)>\e(a)+1.
\]
\end{lm}
\begin{proof}
We have 
\begin{align*}
\widetilde{h}\circ b(\mu_1)(a)&=\widetilde{h}\big(d_A(a_l)\otimes\ldots\otimes a_1\big)+\sum_{i=1}^{l-1}(-1)^{\epsilon_{i+1}(a)}\widetilde{h}\big( a_l\otimes \ldots\otimes d_A(a_i)\otimes\ldots\otimes a_1\big)\\
&=1_A\otimes \pi\big( d_A(a_l)\big)\otimes\ldots\otimes a_1+\sum_{i=1}^{l-1}(-1)^{\epsilon_{i+1}(a)}1_A\otimes\pi( a_l)\otimes \ldots\otimes d_A(a_i)\otimes\ldots\otimes a_1 .
\end{align*}
In addition,
\begin{align*}
b(\mu_1)\circ \widetilde{h}(a)&=   
b(\mu_1)(1_A\otimes\pi(a_l)\otimes a_{l-1}\otimes\ldots\otimes a_1)\\
&=-\sum_{i=1}^{l-1}(-1)^{\epsilon_{i+1}(a)}1_A\otimes \pi(a_l)\otimes a_{l-1}\otimes \ldots\otimes d_A(a_i)\otimes\ldots\otimes a_1.
\end{align*}
Therefore, \[
\big(\widetilde{h}\circ b(\mu_1)+b(\mu_1)\circ \widetilde{h}\big)(a)=1_A\otimes \pi\big( d(a_l)\big)\otimes a_{l-1}\otimes\ldots\otimes a_1,
\] 
 where $\e(1_A\otimes \pi\big( d(a_l)\big)\otimes a_{l-1}\otimes\ldots\otimes a_1)>\e(a)+1.$ 
\end{proof}
Recall
\[
b'(\mu_2)(a)=- \sum_{i=1}^{l-2}(-1)^{\epsilon_{i+1}(a)}a_l\otimes a_{l-1}\otimes\ldots\otimes a_{i+1}a_i\otimes\ldots\otimes a_1-(-1)^{|sa_1|(\epsilon_2(a)+1)}a_1a_l\otimes a_{l-1}\otimes\ldots\otimes a_2.
\]
\begin{lm}\label{lm:b_mu_2'}
    Let $0\ne a=a_l\otimes\ldots\otimes a_1\in  {\fCC}_l(A),$ 
    where $a_i$ are homogeneous elements.
    Then,    
\begin{multline*}
\big( b(\mu_2)\circ \widetilde{h}+\widetilde{h}\circ b'(\mu_2)\big)(a)=-(1-t)(-1)^{|sa_1|\epsilon_2(a)}a_1\otimes \pi(a_l)\otimes a_{l-1}\otimes\ldots\otimes a_2 \\-1_A\otimes \pi(a_l)a_{l-1}\otimes a_{l-2}\otimes \ldots\otimes a_1
-(-1)^{|sa_1|(\epsilon_2(a)+1)}1_A\otimes\pi(a_1a_l)\otimes a_{l-1}\otimes\ldots\otimes a_2.
\end{multline*}
\end{lm}
\begin{proof}
We have
\begin{align*}
    b(\mu_2)\circ \widetilde{h}(a)&=b(\mu_2)(1_A\otimes\pi(a_l)\otimes a_{l-1}\otimes\ldots\otimes a_1)\\
    &=\pi(a_l)\otimes a_{l-1}\otimes\ldots\otimes a_1+\sum_{i=1}^{l-2}(-1)^{\epsilon_{i+1}(a)}1_A\otimes \pi(a_l)\otimes a_{l-1} \ldots\otimes a_{i+1}a_i\otimes\ldots\otimes a_1\\
    &-1_A\otimes \pi(a_l)a_{l-1}\otimes a_{l-2}\otimes\ldots\otimes a_1-(-1)^{|sa_1|\epsilon_2(a)}a_1\otimes \pi(a_l)\otimes a_{l-1}\otimes\ldots\otimes a_2\\
    &=-(1-t)(-1)^{|sa_1|\epsilon_2(a)}a_1\otimes \pi(a_l)\otimes a_{l-1}\otimes\ldots\otimes a_2\\
    &-1_A\otimes \pi(a_l)a_{l-1}\otimes a_{l-2}\otimes\ldots\otimes a_1\\
    &+\sum_{i=1}^{l-2}(-1)^{\epsilon_{i+1}(a)}1_A\otimes \pi(a_l)\otimes a_{l-1} \ldots\otimes a_{i+1}a_i\otimes\ldots\otimes a_1,
\end{align*}
and
\begin{align*}
    \widetilde{h}\circ b'(\mu_2)(a)&=-\sum_{i=1}^{l-2}(-1)^{\epsilon_{i+1}(a)}1_A\otimes\pi(a_l)\otimes a_{l-1}\otimes\ldots\otimes a_{i+1}a_i\otimes\ldots\otimes a_1\\
    &-(-1)^{|sa_1|(\epsilon_2(a)+1)}1_A\otimes\pi(a_1a_l)\otimes a_{l-1}\otimes\ldots\otimes a_2.
\end{align*}

\end{proof}

Define the operator $N_l:=\sum_{i=0}^{l-1}t^i$ on $ {\fCC}_l(A).$ The following lemma is proved in\cite{Loday}.
\begin{lm}\label{lm:loday}
    $b'(\mu_2)\circ N_l=N_l\circ b(\mu_2),$ for all $l>0.$
\end{lm}
Let $a=a_l\otimes\ldots\otimes a_1\in{\fmR}_l$ be a generator such that $a\not\in {\fCC}_l(R).$
Define 
\[ 
U(a):=|\{ 1\le i< l| \ a_i\in R, a_{i+1}\not\in R
\}|+|\{l|\ a_l\in R, a_1\not\in R\}|,
\]
and
\[u(a)=
         \left\{\begin{array}{ll}
        1, &  a\in {\fCC}_l(R) ,\\
       U(a), & \mathrm{otherwise}.\\
        \end{array} \right.\]
Define a map
\[
h:{\fmR}_*\to{\fmR}_*,
\]
acting on generators as 
\[
h(a)=\frac{-1}{u(a)}\cdot\widetilde{h}\circ N_l(a).
\] 
\begin{rem}\label{rem:h_R}
Let $a=a_l\otimes\ldots\otimes a_1\in{\fmR}_l.$ Since $t^{1+l}(a)=t(a)$ and $u(a)=u(t(a)),$ it follows that 
    \[\frac{-1}{u(t(a))}\cdot\widetilde{h}\circ N_l\big(t(a)\big)=
   \frac{-1}{u(t(a))}\sum_{i=0}^{l-1}\widetilde{h}\circ t^i\big(t(a)\big)=\frac{-1}{u(a)}\sum_{i=0}^{l-1}\widetilde{h}\circ t^i(a)=\frac{-1}{u(a)}\cdot\widetilde{h}\circ N_l(a).
    \]
Hence, $h(t(a))=h(a).$       
\end{rem}
By Remark~\ref{rem:h_R}, $h$ is determined by the values of $1_A^{\otimes k}$ where $k\in\Z_{>0},$ and by elements of the form
\begin{equation}\label{eq:defining_q}
q=1_A^{\otimes k_1}\otimes A_1\otimes\ldots \otimes 1_A^{\otimes k_m}\otimes A_m, 
\end{equation}
where $k_i>0,$ $A_i:=a_{i_1}\otimes\ldots\otimes a_{i_{n_i}},$ such that $n_i\in\Z_{>0},$ and $a_{i_j}\in A$ satisfying $\pi(a_{i_j})=0.$ Write $M=\sum_{i=1}^m (k_i+n_i),$ so $q\in {\fCC}_M(A).$

Let $a=a_l\otimes\ldots\otimes a_1\in{\fCC}_l(A).$ Define 
\[
S_1(a):=(-1)^{|sa_1|\epsilon_2(a)}a_1\otimes \pi(a_l)\otimes a_{l-1}\otimes\ldots\otimes a_2 , \quad S_2(a):=1_A\otimes \pi(a_l)a_{l-1}\otimes a_{l-2}\otimes \ldots\otimes a_1,
\]
\[
S_3(a):=(-1)^{|sa_1|(\epsilon_2(a)+1)}1_A\otimes\pi(a_1a_l)\otimes a_{l-1}\otimes\ldots\otimes a_2.
\]

\begin{lm}\label{lm:h_not_cyclic}
    Let $q$ be a nonzero element as in equation~\eqref{eq:defining_q}. 
    Then, there exist $q'_0,\ldots, q'_{M-1}\in{\fCC}_{M+1}(A)$ and $ q_{1_1},\ldots,q_{1_{n_1-1}},\ldots, q_{m_1},\ldots, q_{m_{n_m-1}}\in \fCC_M(A)$ such that $\e(q'_i)-1,\e(q_{i_j})>\e(q),$ and satisfying
\begin{multline*}
\big( \dcc\circ h+h\circ \dcc\big)(q)= \frac{1}{m} \sum_{j=1}^mt^{\sum_{r=1}^{j-1}(k_r+n_r)}(q)\\
+\frac{1}{m}\sum_{i=0}^{M-1}q'_i-\frac{1}{m}(1-t)\circ\sum_{j=1}^mt^{k_j-1+\sum_{r=1}^{j-1}(k_r+n_r)}(q)+\frac{1}{m}\sum_{i=1}^m\sum_{j=1}^{n_i-1}q_{i_j}.
\end{multline*}
\end{lm}

\begin{proof}   
Write
\[
q'_i=\frac{-1}{m}\big( b(\mu_1)\circ \widetilde{h}+\widetilde{h}\circ b(\mu_1)\big)(t^i(q)), \quad 0\le i\le M-1.
\]
By Lemma~\ref{lm:a'_mu_1} we obtain that $\e(q'_i)>\e(q)+1$ and 
    \begin{align*}
\big(b(\mu_1)\circ h+ h\circ b(\mu_1)\big)(q)&\overset{\ref{lm:commutator_mu_1}}{=} 
\frac{-1}{m} \sum_{i=0}^{M-1} \big( b(\mu_1)\circ \widetilde{h}+\widetilde{h}\circ b(\mu_1)\big)(t^i(q))\\
&=\frac{-1}{m}\sum_{i=0}^{M-1}q'_i.
\end{align*}
We have
\begin{align*}
\big( b(\mu_2)\circ h+h\circ b(\mu_2)\big)(q)&\overset{\ref{lm:loday}}{=}\frac{-1}{m} 
\big( b(\mu_2)\circ \widetilde{h}+\widetilde{h}\circ b'(\mu_2)\big)\circ N_M (q)\\
&\overset{\ref{lm:b_mu_2'}}{=}\frac{1}{m}\bigg((1-t)\sum_{i=0}^{M-1}S_1(t^i(q))+\sum_{i=0}^{M-1}S_2(t^i(q))+\sum_{i=0}^{M-1}S_3(t^i(q))\bigg).
\end{align*}
Since $\pi(a_{i_j})=0$ for all $i\in\{1,\ldots, m\}$ and $j\in\{1,\ldots, n_i\},$ it follows that $S_1(t^i(q)),S_2(t^i(q))\ne 0,$ only if $i=s+ \sum_{r=0}^{j-1} (k_r+n_r)$ where $0\le s\le k_j-1$ and $1\le j\le m.$ In this case, we get $S_1(t^i(q))=t^{i-1}(q)$ $,S_2(t^i(q))=t^i(q).$ Thus,
\[
\sum_{i=0}^{M-1} S_1(t^i(q))=\sum_{j=1}^m\sum_{s=0}^{k_j-1} t^{-1+s+\sum_{r=0}^{j-1} (k_r+n_r)} (q), \quad  \sum_{i=0}^{M-1} S_2(t^i(q))=\sum_{j=1}^m\sum_{s=0}^{k_j-1} t^{s+\sum_{r=0}^{j-1} (k_r+n_r)} (q).
\]
In addition, we have
\begin{align*}
    \sum_{i=0}^{M-1}S_3(t^i(q))&=
-\sum_{j=1}^m\sum_{s=1}^{k_j-1}t^{-1+s+ \sum_{r=0}^{j-1} (k_r+n_r)}\\  
&+\sum_{i=1}^{m-1}\sum_{j=1}^{n_i-1}(-1)^{\star_{i,j}}1_A\otimes \pi(a_{i_j}a_{i_{j+1}})\otimes a_{i_{j+2}}\otimes\ldots\otimes a_{i_{n_i}}\otimes 1_A^{\otimes k_{i+1}}\\
&\otimes\ldots
\otimes 1_A^{k_i}\otimes a_{i_1}\otimes\ldots \otimes a_{i_{j-1}}\\
&+\sum_{j=1}^{n_m-1}(-1)^{\star_{m,j}}1_A\otimes \pi(a_{m_j}a_{m_{j+1}})\otimes a_{m_{j+2}}\otimes\ldots\otimes a_{m_{n_m}}\otimes 1_A^{\otimes k_1}\\
&\otimes\ldots
\otimes 1_A^{k_m}\otimes a_{m_1}\otimes\ldots \otimes a_{m_{j-1}},
\end{align*}
where 
\begin{multline*}
\star_{i,j}=|sa_{i_j}|(\epsilon_1(q)+|sa_{i_j}|+1)+(\epsilon_1(q)+1)\sum_{s=1}^ik_s\\
+\sum_{r=1}^{i-1}\sum_{l=1}^{n_r}|sa_{r_l}|(\epsilon_1(q)+|sa_{r_l}|)+\sum_{l=1}^j|sa_{i_l}|(\epsilon_1(q)+|sa_{i_l}|),
\end{multline*}
for all $i\in\{1,\ldots, m\},$ $j\in\{1,\ldots, n_i-1\}.$
Hence, 
\begin{align*}
\frac{-1}{m}\big( b(\mu_2)\circ \widetilde{h}+\widetilde{h}\circ b'(\mu_2)\big)\circ N_M (q) &=\frac{1}{m}\bigg((1-t)\circ\sum_{j=1}^m\sum_{s=0}^{k_j-1}t^{-1+s+\sum_{r=1}^{j-1}(k_r+n_r)}(q)\\
&-(1-t)\circ\sum_{j=1}^m\sum_{s=1}^{k_j-1} t^{-1+s+\sum_{r=1}^{j-1}(k_r+n_r)}(q)\\
&+\sum_{j=1}^mt^{\sum_{r=0}^{j-1} (k_r+n_r)} (q)\\
&+\sum_{i=1}^{m-1}\sum_{j=1}^{n_i-1}(-1)^{\star_{i,j}}1_A\otimes \pi(a_{i_j}a_{i_{j+1}})\otimes a_{i_{j+2}}\otimes\ldots\otimes a_{i_{n_i}}\\
&\otimes 1_A^{\otimes k_{i+1}}\otimes\ldots
\otimes 1_A^{k_i}\otimes a_{i_1}\otimes\ldots \otimes a_{i_{j-1}}\\
&+\sum_{j=1}^{n_m-1}(-1)^{\star_{m,j}}1_A\otimes \pi(a_{m_j}a_{m_{j+1}})\otimes a_{m_{j+2}}\otimes\ldots\otimes a_{m_{n_m}}\\
&\otimes 1_A^{\otimes k_1}
\otimes\ldots
\otimes 1_A^{k_m}\otimes a_{m_1}\otimes\ldots \otimes a_{m_{j-1}}\bigg).
\end{align*}
Write
\[
q_{i,j}=(-1)^{\star_{i,j}}1_A\otimes \pi(a_{i_j}a_{i_{j+1}})\otimes a_{i_{j+2}}\otimes\ldots\otimes a_{i_{n_i}}\otimes 1_A^{\otimes k_{i+1}}\otimes\ldots
\otimes 1_A^{k_i}\otimes a_{i_1}\otimes\ldots \otimes a_{i_{j-1}}
\]
for all $i\in\{1,\ldots, m-1\},$ $j\in\{1,\ldots, n_i-1\},$ and
\[
q_{m,j}=(-1)^{\star_{m,j}}1_A\otimes \pi(a_{m_j}a_{m_{j+1}})\otimes a_{m_{j+2}}\otimes\ldots\otimes a_{m_{n_m}}\otimes 1_A^{\otimes k_1}\otimes\ldots
\otimes 1_A^{k_m}\otimes a_{m_1}\otimes\ldots \otimes a_{m_{j-1}}
\]
for all $j\in\{1,\ldots, k_m-1\}.$
Since $a_{i_j}\not\in R,$ it follows that $\e(q_{i,j})>\e(q).$ Hence, we obtain
\begin{align*}
\big( \dcc\circ h+h\circ \dcc\big)(q)&=\frac{-1}{m}\big( b(\mu_1)\circ \widetilde{h}+\widetilde{h}\circ b(\mu_1)\big)\circ N_M (q)-\frac{1}{m}\big( b(\mu_2)\circ \widetilde{h}+\widetilde{h}\circ b'(\mu_2)\big)\circ N_M (q)\\
&=\frac{1}{m} \sum_{j=1}^mt^{\sum_{r=1}^{j-1}(k_r+n_r)}(q)
+\frac{1}{m}\sum_{i=0}^{M-1}q'_i+\frac{1}{m}(1-t)\circ\sum_{j=1}^mt^{-1+\sum_{r=1}^{j-1}(k_r+n_r)}(q)\\
&+\frac{1}{m}\sum_{i=1}^m\sum_{j=1}^{n_i-1}q_{i_j}.
\end{align*}
\end{proof}
 By Remark~\ref{rem:h_R} the map $h$ is well defined on ${\fmR}^\lambda_*.$ Let $1_A^{\otimes l}\in \fmR_l,$ since $h(1_A^{\otimes l})=0$ if $l\in2\Z_{>0}$ and $h(1_A^{\otimes l})=-1_A^{\otimes l+1}$ if $l\in2\Z_{\ge0}+1,$ it follows that $h$ descends to ${\fmQ}_*^\lambda.$ 
Let $a\in{\fmR}^\lambda_*,$ we write $[a]$ for the element $a$ induces in ${\fmQ}^\lambda_*.$ 
Note that ${\fmQ}_*^\lambda$ is generated by elements of the form $[q],$ where $q$ as in equation \eqref{eq:defining_q}.
\begin{lm}\label{lm:h_is_cyclic}
Let $q$ be a nonzero element as in equation~\eqref{eq:defining_q}.
There exist $q'_0,\ldots, q'_{M-1}\in{\fCC}_{M+1}(A)$ and $q_{1_1},\ldots,q_{1_{n_1-1}},\ldots, q_{m_1},\ldots, q_{m_{n_m-1}}\in {\fCC}_M(A)$ with $\e(q'_i)-1,\e(q_{i_j})>\e(q),$ such that the
map $h:{\fmQ}_*^\lambda\to{\fmQ}_*^\lambda$ satisfies 
\[
(\dcc\circ h+h\circ \dcc)([q])=[q]+\frac{1}{m}\sum_{i=0}^{M-1}[q'_i]-\frac{1}{m}\sum_{i=1}^m\sum_{j=1}^{n_i-1}[q_{i_j}].
\]
\end{lm}
\begin{proof}
By Lemma~\ref{lm:h_not_cyclic} there exist $q'_0,\ldots, q'_{M-1}\in{\fCC}_{M+1}(A),$ $ q_{1_1},\ldots,q_{1_{n_1-1}},\ldots, q_{m_1},\ldots, q_{m_{n_m-1}}\in {\fCC}_M(A)$ such that $\e(q'_i)-1,\e(q_{i_j})>\e(q)$ and satisfying 
\begin{align*}
\big( \dcc\circ h+h\circ \dcc\big)([q])&= \big[\frac{1}{m} \sum_{j=1}^mt^{\sum_{r=1}^{j-1}(k_r+n_r)}(q)\\
&+\frac{1}{m}\sum_{i=0}^{M-1}q'_i+\frac{1}{m}(1-t)\sum_{j=1}^mt^{-1+\sum_{r=1}^{j-1}(k_r+n_r)}(q)+\frac{1}{m}\sum_{i=1}^m\sum_{j=1}^{n_i-1}q_{i_j}\big]\\
&=[q]+\frac{1}{m}\sum_{i=0}^{M-1}[q'_i]+\frac{1}{m}\sum_{i=1}^m\sum_{j=1}^{n_i-1}[q_{i_j}].
\end{align*}
\end{proof}
Let $q\in{\fmR}^\lambda_M$ as in equation~\eqref{eq:defining_q}.
It follows by Lemma~\ref{lm:h_is_cyclic} that
$(\id-\dcc\circ h-h\circ\dcc)^M([q])=0.$
Hence, $\sum_{i=0}^\infty (\id-\dcc\circ h- h\circ\dcc)^i$ is a well defined operator on ${\fmQ}_*^\lambda.$ So, 
    \[
    \widecheck{H}:=\frac{h}{\dcc\circ h+ h\circ\dcc}=\frac{h}{\id-(\id-\dcc\circ h- h\circ\dcc)}=h\circ \sum_{i=0}^\infty (\id-\dcc\circ h- h\circ\dcc)^i
    \]
is well defined on ${\fmQ}_*^\lambda.$ Let $H$ denote the extension of $\widecheck{H}$ to $\mQ_*^\lambda.$

\begin{proof}[Proof of Proposition~\ref{lm:null_homotopic}]
Since $\dcc$ commutes with $\dcc\circ h+ h\circ\dcc,$ it follows that $\dcc\circ H+H\circ\dcc=\id. $ 
\end{proof}

\begin{cor}\label{cor:existence_y_b}
    Let $b\in\Mmm(A,c).$ The element $y_b=H(\dcc(\exp(b)))\in\CC_*^\lambda(A)$ satisfies $\dcc(y_b)=\dcc(\exp(b))+z_b,$ where $z_b\in \CC_*^\lambda(R).$
\end{cor}
\begin{proof}
By Corollary~\ref{lm_b(exp)_is_closed} we have $[\dcc(\exp(b))]\in \mQ_*^\lambda.$ By Proposition~\ref{lm:null_homotopic} we obtain 
    \[
   \dcc\circ H([\dcc(\exp(b)])= (\dcc\circ H+ H\circ\dcc)([\dcc(\exp(b)])=([\dcc(\exp(b)])\in \mQ_*^\lambda.
    \]
Hence, there exists $z_b\in \CC_*^\lambda(R)$ such that 
$\dcc\circ H(\dcc(\exp(b)))=\dcc(\exp(b))+z_b\in \CC_*^\lambda(A).$    
\end{proof}

\subsection{The canonical choice}\label{section:constructing_y_b}
In the following assume $A$ is equipped with an $\infty$-trace $\gto.$ 
This subsection aims to give an explicit formula for the element $y_b=H(\dcc(\exp(b))).$
\begin{lm}\label{lm:y_b_doesnt_dep}
    Let $y,y'\in\mR_*^\lambda$ such that $\dcc(y-y')\in\CC_*^\lambda(R).$ Then, $\gto(y-y')=0.$ 
\end{lm}
\begin{proof}
    By Proposition~\ref{lm:null_homotopic} we get $H^*(\mQ^\lambda)=0.$ So, since $y,y'\in\mR_*^\lambda$ and $\dcc(y-y')\in\CC_*^\lambda(R),$ it follows that $[y-y']=[0]\in H^*(\mQ^\lambda).$ Thus, there exist $x\in \mR_*^\lambda$ and $z\in\CC_*^\lambda(R)
    $ such that $\dcc(x)=(y-y')+z.$ Since $\gto\circ\dcc=0$ and $\gto$ vanishes on $\CC_*^\lambda(R)$, it follows that $\gto(y-y')=0.$
\end{proof}
In dealing with the next results, we use the following notations.
Denote by $J\in \Z_{>0}^{m+1}$ an $m+1-$tuple $(j_0,\ldots,j_m)$ where $j_i\in\Z_{>0}$ for $0\le i\le m.$ 
We write $|J|=j_0+\ldots+j_m.$
Let $\sigma$ denote the cyclic permutation, that is, if $J=(j_0,\ldots, j_a),$ then $\sigma(J)=(j_1,\ldots, j_a,j_0).$ 
Define
\[
\Z_{m,k}:=\{J\in\Z_{>0}^{m+1}| \ \ |J|=k\},
\]
and let $\sim$ be an equivalence relation on $\Z_{m,k}$ given by $J\sim J'$ if there exists $i\in\Z$ such that $\sigma^i(J)=J'.$ Write $[J]$ for the equivalence class of $J,$ and let
$\Z_{m,k}^\sigma$ denote the set of equivalence classes. 

Recall that for $E\in\ima\nu_R,$ we denote by
$\Mmm(A,c)_E$ the set of all bounding cochains modulo $F^EA$ with a fixed $c,$ that is $d_A(b)-b^2\equiv c\cdot 1_A \pmod{F^EA}.$ 
Let  $b\in \Mmm(A,c)_E,$ and assume there exists a projection $\pi:A\to R$ such that $\pi(b)=0.$
Let $k,l\in\Z_{\ge0},$ define
\[
B_{k,l}:=(c\cdot1_A \otimes 1_A)^{\otimes k}\otimes b^{\otimes l}.
\]
For all multindices $J,$ let $C_J\in\R.$ 
Let $n\in\Z_{\ge0},$ define
   \[
   y_n= -\sum_{m=0}^n \sum_{[J]\in\zzs}\sum_{I\in\Z_{>0}^{m+1}}C_J
   \bigotimes_{k=0}^m B_{j_k,i_k}\in\mR_*^\lambda.
  \]
Let $0\le m\le n$ and  $[J]=[(j_0,\ldots,j_m)]\in\zzs,$ define  
\[
 x_{n,[J]}= -\sum_{I\in\Z_{>0}^{m+1}}
\bigotimes_{k=0}^m B_{j_k,i_k}\in\mR_*^\lambda.
\]
Note that since we work in the complex $\mR_*^\lambda,$ it follows that $y_n,$ $x_{n,[J]}$ and $C_J$ do not depend on the choice of a  representative of the equivalence class $[J]\in\zzs.$
   \begin{lm}\label{lm:dcc_y_n_a_q}
      \begin{align*}
          \dcc(x_{n,[J]})&\equiv-\sum_{s=0}^m \sum_{\substack{I\in\Z_{>0}^{m+1}\\ i_s=1}}\bigotimes_{k=0}^{s-1}B_{j_k,i_k}\otimes \big(B_{j_s,0}\otimes c\cdot1_A \big)\bigotimes_{k=s+1}^mB_{j_k,i_k}\\
    &-\sum_{s=0}^m\sum_{\substack{I\in\Z_{>0}^{m+1}\\ i_s>1}}\bigotimes_{k=0}^{s-1}B_{j_k,i_k}\otimes \big(B_{j_s,0}\otimes c\cdot 1_A\otimes b^{\otimes i_s-1}\big)\bigotimes_{k=s+1}^mB_{j_k,i_k}\\
&-\sum_{s=0}^m\sum_{I\in\Z_{>0}^{m+2}}\bigotimes_{k=0}^{s-1}B_{j_k,i_k}\otimes \big(B_{j_s,i_s}\otimes c\cdot1_A\otimes b^{\otimes i_{s+1}}\big)\bigotimes_{k=s+1}^mB_{j_k,i_{k+1}}\\
&-\sum_{s=0}^m\sum_{\substack{I\in\Z_{>0}^{m+1}\\ i_s>1}}\bigotimes_{k=0}^{s-1}B_{j_k,i_k}\otimes \big(B_{j_s,i_s-1}\otimes c\cdot1_A\big)\bigotimes_{k=s+1}^mB_{j_k,i_{k+1}}\\
     &+\sum_{s=0}^m\sum_{I\in\Z_{>0}^{m+
1}}\bigotimes_{k=0}^{s-1}B_{j_k,i_k}\otimes\big( c\cdot1_A\otimes B_{j_s-1,i_s}\big)
   \bigotimes_{k=s+1}^mB_{j_k,i_k}\pmod{F^EA}.
      \end{align*}
  \end{lm}
  \begin{proof}
      We have,
\begin{align*}
    \dcc(x_{n,[J]})&=- \sum_{s=0}^m\sum_{I\in\Z_{>0}^{m+
1}}\sum_{t=1}^{i_s}\bigotimes_{k=0}^{s-1}B_{j_k,i_k}\otimes\big( B_{j_s,t-1}\otimes \delta(b)\otimes b^{\otimes i_s-t}\big)\bigotimes_{k=s+1}^mB_{j_k,i_k}\\
    &- \sum_{I\in\Z_{>0}^{m+
1}}\big( c\cdot1_A\otimes B_{j_0-1,i_0}\big)\bigotimes_{k=1}^mB_{j_k,i_k}\\
   &+ \sum_{s=0}^m\sum_{I\in\Z_{>0}^{m+
1}}\bigotimes_{k=0}^{s-1}B_{j_k,i_k}\otimes\big(  c\cdot1_A\otimes B_{j_s-1,i_s}\big)\bigotimes_{k=s+1}^mB_{j_k,i_k}\\
   &+ \sum_{s=0}^m\sum_{I\in\Z_{>0}^{m+
1}}\sum_{t=2}^{i_s}\bigotimes_{k=0}^{s-1}B_{j_k,i_k}\otimes\big( B_{j_s,t-2}\otimes b^2\otimes b^{\otimes i_s-t}\big)\bigotimes_{k=s+1}^mB_{j_k,i_k}\\
    &+ \sum_{I\in\Z_{>0}^{m+
1}}\big(c\cdot b\otimes 1_A\otimes B_{j_0-1,i_0}\big)\bigotimes_{k=1}^{m-1}B_{j_k,i_k}\otimes B_{j_m,i_m-1}\\
&\equiv- \sum_{s=0}^m\sum_{I\in\Z_{>0}^{m+
1}}\sum_{t=1}^{i_s}\bigotimes_{k=0}^{s-1}B_{j_k,i_k}\otimes \big(B_{j_s,t-1}\otimes c\cdot1_A\otimes b^{\otimes i_s-t}\big)\bigotimes_{k=s+1}^mB_{j_k,i_k}\\
&+\sum_{s=0}^m\sum_{I\in\Z_{>0}^{m+
1}}\bigotimes_{k=0}^{s-1}B_{j_k,i_k}\otimes\big( c\cdot1_A\otimes B_{j_s-1,i_s}\big)\bigotimes_{k=s+1}^mB_{j_k,i_k}\pmod{F^EA}.
    \end{align*}
In addition, we have
\begin{align*}
    -\sum_{s=0}^m&\sum_{I\in\Z_{>0}^{m+1}}\sum_{t=1}^{i_s}\bigotimes_{k=0}^{s-1}B_{j_k,i_k}\otimes\big( B_{j_s,t-1}\otimes c\cdot1_A\otimes b^{\otimes i_s-t}\big)\bigotimes_{k=s+1}^mB_{j_k,i_k}(c,b)\\
    &=- \sum_{s=0}^m\sum_{\substack{I\in\Z_{>0}^{m+1}\\ i_s=1}}\bigotimes_{k=0}^{s-1}B_{j_k,i_k}\otimes \big(B_{j_s,0}\otimes c\cdot1_A \big)\bigotimes_{k=s+1}^mB_{j_k,i_k}(c,b)\\
    &-\sum_{s=0}^m\sum_{\substack{I\in\Z_{>0}^{m+1}\\ i_s>1}}\bigotimes_{k=0}^{s-1}B_{j_k,i_k}\otimes \big(B_{j_s,0}\otimes c\cdot1_A\otimes b^{\otimes i_s-1}\big)\bigotimes_{k=s+1}^mB_{j_k,i_k}\\
&-\sum_{s=0}^m\sum_{I\in\Z_{>0}^{m+2}}\bigotimes_{k=0}^{s-1}B_{j_k,i_k}\otimes \big(B_{j_s,i_s}\otimes c\cdot1_A\otimes b^{\otimes i_{s+1}}\big)\bigotimes_{k=s+1}^mB_{j_k,i_{k+1}}\\
&-\sum_{s=0}^m\sum_{\substack{I\in\Z_{>0}^{m+1}\\ i_s>1}}\bigotimes_{k=0}^{s-1}B_{j_k,i_k}\otimes \big(B_{j_s,i_s-1}\otimes c\cdot1_A\big)\bigotimes_{k=s+1}^mB_{j_k,i_{k+1}}.
\end{align*}
This completes the proof.
  \end{proof}

\begin{lm}\label{lm:id_hdcc_x}
    \begin{multline*}
   (\id-h\circ\dcc )(x_{n,[J]})\equiv -\frac{2}{m+1}\sum_{s=0}^mx_{n,[(j_0,\ldots,j_s+1,\ldots,j_m)]}-\frac{1}{m+2}\sum_{s=0}^mx_{n,[(j_0,\ldots, j_s,1,j_{s+1},\ldots,j_m)]}\\
   -\bigg(\frac{1-\delta_{m0}}{m-\delta_{m0}}\bigg )\bigg(
   x_{n,[(j_0+j_m+1,j_1,\ldots, j_{m-1})]}+\sum_{s=0}^{m-1}x_{n,[(j_0,\ldots, j_s+j_{s+1}+1,\ldots, j_m)]}\bigg)\pmod{F^EA}.
\end{multline*}
\end{lm}

\begin{proof}
By Lemma~\ref{lm:dcc_y_n_a_q} we get
\begin{align*}
h\circ\dcc(x_{n,[J]})&\equiv h\bigg(- \sum_{s=0}^m\sum_{\substack{I\in\Z_{>0}^{m+1}\\ i_s=1}}\bigotimes_{k=0}^{s-1}B_{j_k,i_k}\otimes \big(B_{j_s,0}\otimes c\cdot1_A \big)\bigotimes_{k=s+1}^mB_{j_k,i_k}\\
    &-\sum_{s=0}^m\sum_{\substack{I\in\Z_{>0}^{m+1}\\ i_s>1}}\bigotimes_{k=0}^{s-1}B_{j_k,i_k}\otimes \big(B_{j_s,0}\otimes c\cdot1_A\otimes b^{\otimes i_s-1}\big)\bigotimes_{k=s+1}^mB_{j_k,i_k}\\
&-\sum_{s=0}^m\sum_{I\in\Z_{>0}^{m+2}}\bigotimes_{k=0}^{s-1}B_{j_k,i_k}\otimes \big(B_{j_s,i_s}\otimes c\cdot1_A\otimes b^{\otimes i_{s+1}}\big)\bigotimes_{k=s+1}^mB_{j_k,i_{k+1}}\\
&-\sum_{s=0}^m\sum_{\substack{I\in\Z_{>0}^{m+1}\\ i_s>1}}\bigotimes_{k=0}^{s-1}B_{j_k,i_k}(c,b)\otimes \big(B_{j_s,i_s-1}\otimes c\cdot1_A\big)\bigotimes_{k=s+1}^mB_{j_k,i_{k+1}}\bigg)\\
       &- \frac{1}{m+1}\sum_{s=0}^m\sum_{I\in\Z_{>0}^{m+1}}B_{j_s,i_s}\bigotimes_{k=s+1}^mB_{j_k,i_k}\bigotimes_{k=0}^{s-1}B_{j_k,i_k} \pmod{F^EA}.
\end{align*}  
Since we work in the cyclic complex $\CC_*^\lambda(A),$ it follow that 
\[
x_{n,[J]}\equiv-\frac{1}{m+1}\sum_{s=0}^m\sum_{I\in\Z_{>0}^{m+1}}B_{j_s,i_s}\bigotimes_{k=s+1}^mB_{j_k,i_k}\bigotimes_{k=0}^{s-1}B_{j_k,i_k} \pmod{F^EA}.
\]
Note that if $m=0,$ then
\[
\sum_{s=0}^m\sum_{\substack{I\in\Z_{>0}^{m+1}\\ i_s=1}}\bigotimes_{k=0}^{s-1}B_{j_k,i_k}\otimes \big(B_{j_s,0}\otimes c\cdot1_A \big)\bigotimes_{k=s+1}^mB_{j_k,i_k}= c^{|J|+1}\cdot 1_A^{\otimes 2|J|+1}.
\]
By Remark~\ref{rem:cyclic_cx_r} we have $\CC_l^\lambda(R)=0$ for all $l\in2\Z_{>0},$ so
$
h\big(c^{|J|+1}\cdot 1_A^{\otimes 2|J|+1}\big)=0.
$
Hence,
\begin{align*}
    (\id-h\circ\dcc)(x_{n,[J]})
    &\equiv-\sum_{s=0}^m\sum_{\substack{I\in\Z_{>0}^{m+1}\\i_s=1}}\frac{1-\delta_{m0}}{m-\delta_{m0}}\bigotimes_{k=0}^{s-1}B_{j_k,i_k}\otimes B_{j_s+1,0}\bigotimes_{k=s+1}^mB_{j_k,i_k}\\
    &-\sum_{s=0}^m\sum_{\substack{I\in\Z_{>0}^{m+1}\\ i_s>1}} \frac{1}{m+1}\bigotimes_{k=0}^{s-1}B_{j_k,i_k}\otimes B_{j_s+1,i_s-1}\bigotimes_{k=s+1}^mB_{j_k,i_k}\\
    &-\sum_{s=0}^m\sum_{I\in\Z_{>0}^{m+2}}\frac{1}{m+2}\bigotimes_{k=0}^sB_{j_k,i_k}\otimes  B_{1,i_{s+1}}\bigotimes_{k=s+1}^mB_{j_k,i_{k+1}}\\
    &-\sum_{s=0}^{m-1}\sum_{\substack{I\in\Z_{>0}^{m+1}\\ i_s>1}}\frac{1}{m+1}\bigotimes_{k=0}^{s-1}B_{j_k,i_k}\otimes \big(B_{j_s,i_s-1}\otimes B_{j_{s+1}+1,i_{s+1}}\big)\bigotimes_{k=s+2}^mB_{j_k,i_{k+1}}\\
    &-\sum_{\substack{I\in\Z_{>0}^{m+1}\\ i_m>1}}\frac{1}{m+1}B_{j_0+1,i_0}\bigotimes_{k=1}^{m-1}B_{j_k,i_k}\otimes B_{j_m,i_m-1} \pmod{F^EA}.
\end{align*}
By changing variables, we obtain:
\begin{enumerate}[label=(\arabic*)]
    \item 
    \begin{align*}
\sum_{s=0}^m&\sum_{\substack{I\in\Z_{>0}^{m+1}\\i_s=1}}\frac{1}{m}\bigotimes_{k=0}^{s-1}B_{j_k,i_k}\otimes B_{j_s+1,0}\bigotimes_{k=s+1}^mB_{j_k,i_k}\\
    &=\sum_{s=0}^{m-1}\sum_{\substack{I\in\Z_{>0}^{m+1}\\i_s=1}}\frac{1}{m}\bigotimes_{k=0}^{s-1}B_{j_k,i_k}\otimes B_{j_s+j_{s+1}+1,i_{s+1}}\bigotimes_{k=s+2}^mB_{j_k,i_k}\\
    &+ \sum_{\substack{I\in\Z_{>0}^{m+1}\\i_s=1}}\frac{1}{m}B_{j_m+j_0+1,i_0}\bigotimes_{k=1}^{s-1}B_{j_k,i_k}\\
    &=\frac{1}{m}\sum_{s=0}^{m-1}x_{n,[(j_0,\ldots, j_s+j_{s+1}+1,\ldots, j_m)]}+\frac{1}{m}x_{n,[(j_0+j_m+1,j_1,\ldots, j_{m-1})]}\pmod{F^EA},
    \end{align*}
\item
\begin{align*}
\sum_{s=0}^m&\sum_{\substack{I\in\Z_{>0}^{m+1}\\ i_s>1}} \frac{1}{m+1}\bigotimes_{k=0}^{s-1}B_{j_k,i_k}\otimes B_{j_s+1,i_s-1}\bigotimes_{k=s+1}^mB_{j_k,i_k}\\
    &+\sum_{s=0}^{m-1}\sum_{\substack{I\in\Z_{>0}^{m+1}\\ i_s>1}}\frac{1}{m+1}\bigotimes_{k=0}^{s-1}B_{j_k,i_k}\otimes \big(B_{j_s,i_s-1}\otimes B_{j_{s+1}+1,i_{s+1}}\big)\bigotimes_{k=s+2}^mB_{j_k,i_{k+1}}\\
    &+\sum_{\substack{I\in\Z_{>0}^{m+1}\\ i_m>1}}\frac{1}{m+1}B_{j_0+1,i_0}\bigotimes_{k=1}^{m-1}B_{j_k,i_k}\otimes B_{j_m,i_m-1}\\
    &=2\sum_{s=0}^m\sum_{I\in\Z_{>0}^{m+1}}\frac{1}{m+1}\bigotimes_{k=1}^{s-1}B_{j_k,i_k}\otimes B_{j_s+1,i_s}\bigotimes_{k=s+1}^m B_{j_k,i_k}\\
    &=\frac{2}{m+1}\sum_{s=0}x_{n,[(j_0,\ldots,j_s+1,\ldots,j_m)]} \pmod{F^EA},
\end{align*}
    \item 
    \begin{align*}
        \sum_{s=0}^m&\sum_{I\in\Z_{>0}^{m+2}}\frac{1}{m+2}\bigotimes_{k=0}^sB_{j_k,i_k}\otimes B_{1,i_{s+1}}\bigotimes_{k=s+1}^mB_{j_k,i_{k+1}}\\
    &=\frac{1}{m+2}\sum_{s=0}^mx_{n,[(j_0,\ldots, j_s,1,j_{s+1},\ldots,j_m)]}\pmod{F^EA}
    .
    \end{align*}
\end{enumerate}
This completes the proof.
\end{proof}

\begin{lm}\label{lm:dcc_dcc}
For all $n\in\Z_{\ge0},$
\[
(\id-\dcc\circ h-h\circ \dcc)^n(\dcc(\exp(b)))=(\id-\dcc\circ h)^n(\dcc(\exp(b))).
\]    
\end{lm}
\begin{proof}
    We prove it by induction. For $n=0,$ the statement is a tautology. Assume the statement holds for a general $n.$
Let $S$ be the linear operator on $\mR^\lambda_*$ satisfying $(\id-\dcc\circ h)^n=\id+\dcc\circ S.$  Hence, \begin{align*}
    (\id-\dcc\circ h-h\circ \dcc)^{n+1}(\dcc(\exp(b)))&=(\id-\dcc\circ h-h\circ \dcc)\circ(\id-\dcc\circ h)^n(\dcc(\exp(b)))\\
    &=(\id-\dcc\circ h-h\circ \dcc)\circ(\id+\dcc\circ S)(\dcc(\exp(b)))\\
    &=(\id-\dcc\circ h)\circ(\id+\dcc\circ S)(\dcc(\exp(b)))\\
    &=(\id-\dcc\circ h)^{n+1}(\dcc(\exp(b))).
\end{align*}    
\end{proof}

Define
\[
y_{b,n}:=h\circ(\id-\dcc\circ h-h\circ\dcc)^n(\dcc(\exp(b))).
\]
\begin{lm}\label{lm:y_b_n_computation}
There exist  $C_J\in\R$ for all $J\in\zzs$ such that 
    \[
   y_{b,n}\equiv y_n:= - \sum_{m=0}^n \sum_{[J]\in\zzs}\sum_{I\in\Z_{>0}^{m+1}}C_J
   \bigotimes_{k=0}^mB_{j_k,i_k} \pmod{F^EA}.
  \]
\end{lm}
\begin{proof}
    We prove it by induction. For $n=0,$ we have
    \[
    y_{b,0}=-\sum_{l=1}^\infty c\cdot 1_A\otimes 1_A\otimes b^{\otimes l}=-\sum_{l=1}^\infty B_{1,l} \pmod{F^EA},
    \]
so $C_{[(1)]}=1.$ Assume the statement holds for a general $n.$ We have,
\begin{align*}
y_{b,{n+1}}&=h\circ(\id-\dcc\circ h-h\circ\dcc)^{n+1}(\dcc(\exp(b)))\\
&\overset{\ref{lm:dcc_dcc}}{=}h\circ(\id-\dcc\circ h)^{n+1}(\dcc(\exp(b)))\\
&=h\circ(\id-\dcc\circ h)\circ(\id-\dcc\circ h)^n(\dcc(\exp(b))\\
&=(\id-h\circ\dcc)\circ h\circ(\id-\dcc\circ h)^n(\dcc(\exp(b))\\
&=(\id-h\circ\dcc)(y_{b,n}).
\end{align*}
Hence, by Lemma~\ref{lm:id_hdcc_x} we get
\begin{multline*}
    y_{b,n+1}\equiv\sum_{m=0}^n \sum_{J\in\zzs}C_J\bigg(-\frac{2}{m+1}\sum_{s=0}^mx_{n,[(j_0,\ldots,j_s+1,\ldots,j_m)]}-\frac{1}{m+2}\sum_{s=0}^mx_{n,[(j_0,\ldots, j_s,1,j_{s+1},\ldots,j_m)]}\\
   -\bigg(\frac{1-\delta_{m0}}{m-\delta_{m0}}\bigg )\bigg(
   x_{n,[(j_0+j_m+1,j_1,\ldots, j_{m-1})]}+\sum_{s=0}^{m-1}x_{n,[(j_0,\ldots, j_s+j_{s+1}+1,\ldots, j_m)]}\bigg)  \bigg)\pmod{F^EA}.
\end{multline*}

Thus, we can write
\[
y_{b,n+1}\equiv-\sum_{m=0}^{n+1} \sum_{J\in\zzs}\sum_{I\in\zz}K_J
   \bigotimes_{k=0}^mB_{j_k,i_k}\pmod{F^EA},
\]
where $K_J$ are linear combinations of the constants $C_J$ of $y_{b,n}.$
\end{proof}
 \begin{cor}\label{cor:y_b}
     If $y_b=H(\dcc(\exp(b))),$ then there exist  $C_J\in\R$ for all $J\in\zzs,$  such that 
      \[
   y_b\equiv -\sum_{n=0}^\infty \sum_{m=0}^n \sum_{J\in\zzs}\sum_{I\in\Z_{>0}^{m+1}}C_J
\bigotimes_{k=0}^mB_{j_k,i_k} \pmod{F^EA}.
  \]
 \end{cor}

Note that the constants $C_{[(j_0,\ldots, j_m)]}$ do not depend on the bounding cochain $b,$ but only on the equivalence class $[(j_0,\ldots, j_m)].$ We call to $C_{[(j_0,\ldots, j_m)]}$ where $j_0+\ldots+ j_m=k+1$ a constant of level $k.$ 
\begin{rem}\label{rem:constant_c}
It follows from the proof of Lemma~\ref{lm:y_b_n_computation} that the constants $C_{[(j_0,\ldots, j_m)]}$ of $y_b=H(\dcc(\exp(b)))$ can be computed inductively.

Consider the graph in Figure~\ref{fig:graph}.  
Let $C_J$ be a constant in the graph of level $k,$ let $C_{J_1},\ldots, C_{J_t}$ denote the constants of level $k-1$ that are connected to $C_J,$ and let $\alpha_{J_1},\ldots, \alpha_{J_t}$ be the weights on the edges that connect $C_J$ to these constants respectively. Then,
$C_J=\sum_{i=1}^t \alpha_{J_i}C_{J_i}.$
For example,
\[
C_{[(2)]}=2\cdot C_{[(1)]}=2, \quad C_{[(1,1]}=\frac{1}{2}\cdot C_{[(1)]}=\frac{1}{2},\] 
\[ C_{[(3)]}=2\cdot C_{[(2)]}+2\cdot C_{[(1,1)]}=5,\quad
C_{[(2,1)]}=\frac{1}{2}\cdot C_{[(2)]}+ 2\cdot C_{[(1,1)]}=2.
\]
\begin{figure}[H]
     \centering
     \includegraphics[scale=0.38]{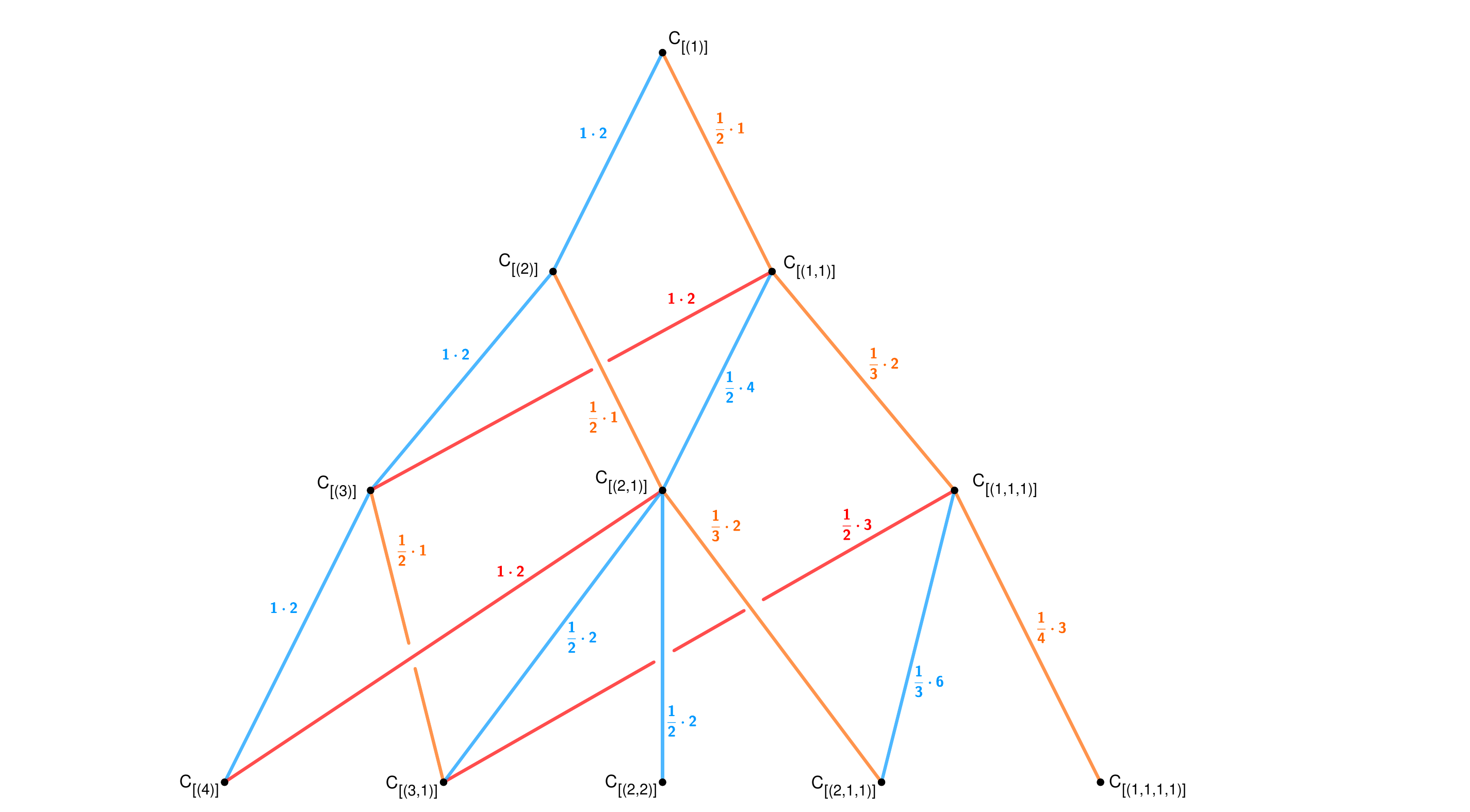}
     \caption{The constants of level $1,2,3,4$.}
     \label{fig:graph}
 \end{figure}    
\end{rem}

Let $E, E'\in \R_{>0}$ such that $E'>E.$ Let
$b\in\Mmm(A,c)_E$ and $q\in (A)^2$ satisfying $\nu(q)\ge E'$ and 
$d_A(b)-b^2\equiv c\cdot 1_A+q \pmod{F^{E'}A}.$ Assume there is a projection $\pi:A\to R$ such that $\pi(b)=0.$
\begin{lm}\label{lm:d_a_q0}
    $d_A(q)\equiv0\pmod{F^{E'}A}.$
\end{lm}
\begin{proof}
\begin{align*}
    0&=-d_A(b^2)-bd_A(b)+d_A(b)b\\
    &=d_A(d_A(b)-b^2)-b(d_A(b)-b^2)+(d_A (b)-b^2)b\\
    &\equiv d_A(c\cdot1_A +q)- b(c\cdot1_A +q)+(c\cdot1_A +q)b\\
    &\equiv d_A(q)- bq+qb\\
    &\equiv d_A(q) \pmod{F^{E'}A} .
\end{align*}

\end{proof}    

\begin{lm}\label{lm:dcc_y_b_mod_e}
Let 
\[
y_b=H(\sum_{l=1}^\infty c\cdot1_A\otimes b^{\otimes l-1}).
\] 
Then, there exists $z_b\in\CC_*^\lambda(R)$ such that 
\[
\dcc(y_b)\equiv \dcc(\exp(b))+z_b-q\pmod{F^{E'}A}. 
\]
\end{lm}

\begin{proof}

By Lemma~\ref{lm:exp(b)_q} we have     
 \[
\dcc(\exp(b))=\sum_{l=1}^\infty (d_A( b)-b^2)\otimes b^{\otimes l-1}\equiv q+ \sum_{l=1}^\infty c\cdot 1_A\otimes b^{\otimes l-1} \pmod{F^{E'}A}.
    \]
Hence, it follows by Lemma~\ref{lm:d_a_q0} that
\[
\dcc(\sum_{l=1}^\infty c\cdot1_A\otimes b^{\otimes l-1}) \equiv 0\pmod{F^{E'}A}.
\]
By the definition of $H,$ there exists $z
_b\in\CC_*^\lambda(R)$ such that    
\begin{align*}
\dcc\circ H(\sum_{l=1}^\infty c\cdot1_A\otimes b^{\otimes l-1})&=\sum_{l=1}^\infty c\cdot1_A\otimes b^{\otimes l-1}-H\circ\dcc(\sum_{l=1}^\infty c\cdot1_A\otimes b^{\otimes l-1})+z_b\\
&\equiv \sum_{l=1}^\infty c\cdot1_A\otimes b^{\otimes l-1}+z_b\pmod{F^{E'}A}.
\end{align*}
Thus, \[
\dcc(y_b)\equiv \dcc(\exp(b))+z_b -q\pmod{F^{E'}A}. 
\]

\end{proof}
Let $b\in\Mmm(A,c),$ and let 
$y_b\in\mR_*^\lambda(A)$ such that $\dcc(y_b)=\dcc(\exp(b))+z_b,$ for  $z_b\in\CC_*^\lambda(R).$
Define
\[
    \expp(b):=\exp(b)-y_b.
\]
By Lemma~\ref{lm:y_b_doesnt_dep} we get that $\gto(\expp(b))$ does not depend on the choice of $y_b.$  

\section{Bounding cochains in matrix factorizations }\label{section_bc_mf}

\subsection{Algebraic setting}\label{section:algebraic_setting} 

Let $n\in\Z_{>0}$ be odd. Let
$\big((R,\nu_R),(S,\{\nu^S_\lambda\}_{\lambda\in \Lambda}) , W \big)$ be a valued Landau-Ginzburg model, and
let $\MF(W,w)$ be the associated  normed matrix factorization category defined in Section~\ref{Section:normed_mf_category_def}. Let $(M,D_M,\{\nu^M_\lambda\}_{\lambda\in \Lambda})$ be an object of $\MF(W,w).$
Recall that $\End(M)$ is equipped with the norm $\nu_M$ defined in Section~\ref{Section_valuations}.  

Let $s$ be a formal variable of degree $1-n$ and consider the ring $\R[[s]].$ 
Equip $\R[[s]]$ with the trivial differential and with the valuation $\nu_s$ given by 
\[
\nu_s(\sum_{i=0}^\infty x_i s^i)=\min\{i| \ x_i\ne 0\}.
\]
Let $\Rs=\R[[s]]\otimes R$ and $\ems=\R[[s]]\otimes\End(M)$ be 
equipped with the valuation $\nu_{\Rs}:=\nu_s\otimes \nu_R,$ and the norm
$\nu_{\Ms}:=\nu_s\otimes\nu_M$ respectively. 
Recall $\Rs$ is equipped with the trivial differential. It follows by Corollary~\ref{lm:end_val_dga_2} that $(\ems, \delta, \id_{\Ms}, \nu_{\Ms})$ is a normed unital DGA over $\Rs.$

Assume $\MF(W,w)$ is equipped with a dimension $n$ normed $\infty$-trace 
\[
\gto:\CC_*^\lambda(\MF(W,w))\to R.
\]
Extend $\gto$ to $\End(M)^s$ by extension of scalars.

Recall the definition of the filtration $F^E$~\eqref{eq:filtration:def}.
Recall $R$ is equipped with a homomorphism $h_G:G\to \R$ such that $\ima h_G$ is a discrete subset of $\R.$ Hence, by the definition of $\nu_{\Rs},$ it follows that $\ima\nu_{\Rs}\backslash\{\infty\}$ is a discrete subset of $\R.$ It follows by  Lemma~\ref{lm:ima_m_eq_ima_r} that $\ima\nu_{\Ms}=\ima\nu_{\Rs}.$ 
Let $\alpha_1,\dots, \alpha_m\in F^0\Upsilon$ where $\Upsilon=\Rs,$ or $\Upsilon=\ems.$ Define 
\[
\mG(\alpha_1,\ldots,\alpha_m):=\big\{x\in\ima\nu_\Upsilon\backslash\{\infty\}  |\ x\ge\min\{ \nu_\Upsilon(\alpha_1),\ldots, \nu_\Upsilon(\alpha_m)\}\big\}.
\]
Thus, we can write
\begin{equation}\label{sabab_def}
\mG(\alpha_1,\ldots,\alpha_m)=\{E_1, E_2, E_3,\ldots\}, \qquad E_i<E_{i+1} \quad \forall i,
\end{equation}
where $E_1=\min\{ \nu_\Upsilon(\alpha_1),\ldots, \nu_\Upsilon(\alpha_m)\}.$

For a valued or a normed ring $(\Upsilon,\nu_\Upsilon),$ 
define $F^{\ge E} \Upsilon:=\{\alpha\in  \Upsilon| \ \nu_\Upsilon(\alpha)\ge E\},$
and $\Upsilon_E:=F^{\ge E} \Upsilon/F^E \Upsilon.$ 

\subsection{Existence of bounding cochains}\label{subsection_existence_bc}
This section describes an inductive construction of a bounding cochain in the algebra $\ems,$ where $M$ is a normed Calabi-Yau spherical object.
The obstruction theory 
introduced in~\cite{FOOO1} will play an essential role 
in this construction.

We will denote by $\id$ both $\id_M$ and $\id_{\Ms}.$ Define
\[
mc(\alpha):=\delta(\alpha)-\alpha^2, \quad \alpha\in \ems.
\] 
Fix $r\in F^0\Rs,$ and consider the subset $\mG(r)\subset\R_{>0}$ ordered as in~\eqref{sabab_def}.
Let $l\ge1.$ Suppose we have $b_{(l)}\in F^0\ems$ with $|b_{(l)}|=1,$ 
$\nu_{\Ms}(b_{(l)})\in \mG(r),$ 
and 
\[
mc(b_{(l)})\equiv c_{(l)}\cdot \id \pmod{F^{E_l}\ems}, \quad c_{(l)}\in (\Rs)^2, \quad \nu_{\Rs}(c_{(l)})>0.
\]
Define the obstruction cochain $o_l\in F^{\ge E_{l+1}}\ems$ to be an endomorphism satisfying
\begin{equation}\label{obstruction_chain}
mc(b_{(l)})\equiv c_{(l)}\cdot \id+ o_l \pmod{F^{E_{l+1}}\ems}.
\end{equation}
\begin{lm}\label{lemma_o_closed}
$
\delta(o_l)\equiv0 \pmod{F^{E_{l+1}}\ems}.
$
\end{lm} 
\begin{proof}
This follows from Lemma~\ref{lm:d_a_q0}.

\end{proof}

\begin{lm}\label{lm:h_0}
   Let $M$ be a spherical object. Then, $H^*(\eems_E)\cong H^*(S^n;\Rs_E)$ for all $E\in\ima\nu_{\Ms}\backslash\{\infty\}.$
\end{lm}
\begin{proof}
    Let $E\in\ima\nu_{\Ms}\backslash\{\infty\} .$ Since $\ima\nu_{\Ms}=\ima\nu_{\Rs},$ there exists $r_E\in {\Rs}$ such that $\nu_{\Rs}(r_E)=E.$ Hence, multiplication by $r_E$ induces isomorphisms $H^*(\eems_0)\cong H^*(\eems_E)$ and $H^*(S^n;\Rs_0)\cong H^*(S^n;\Rs_E).$ Thus, 
    $H^*(\eems_E)\cong H^*(S^n;\Rs_E).$
\end{proof}
Let $M$ be a normed Calabi-Yau object.
Let $h_0$ denote an element in $\End^n(M) $ representing a generator of $H^n(\End(M)^s_0)$ with $\nu_M(h_0)=0$ such that $\gto_1(h_0)=1.$ Fix a projection $\pi:\ems\to\Rs.$ Since $\Rs$ consists elements of even degree, it follows that $\pi(\Phi)=0$ for all $\Phi\in\ems$ of odd degree.
Let $H$ be the homotopy defined with respect to $\pi$ in Section~\ref{section:defining_the_homotopy}. 
Define
\begin{equation}\label{eq:y_b_l}
   y_{b_{(l)}}:=H\bigg(\sum_{l=1}^\infty c_{(l)}\cdot\id\otimes b_{(l)}^{\otimes l-1}\bigg).
\end{equation}
Since $b_{(l)}\in\Mmm(A,c_{(l)})_{E_l}$ and $\pi(b_{(l)})=0,$ it follows 
by Corollary~\ref{lm_b(exp)_is_closed} that 
\[
y_{b_{(l)}}\equiv H(\dcc(\exp(b_{(l)})))\pmod{F^{E_l}\ems}.
\]
Define
$\expp(b_{(l)}):=\exp(b_{(l)})-y_{b_{(l)}}.$ 

\begin{lm}\label{lm:computation}
Let $M$ be a normed Calabi-Yau spherical object, and let $r\in F^0\Rs$ with $|r|=1-n.$
Suppose 
\[
\gto(\expp(b_{(l)}))\equiv r \pmod{F^{E_l}\Rs},
\]
and suppose 
there exist $c_l\in F^{\ge E_{l+1}}\Rs$ and $b_l\in F^{\ge E_{l+1}}\End(M)^s$ with $|c_l|=2,$ $|b_l|= 1,$ such that
\[
\delta(b_l)\equiv -o_l+c_l\cdot \id \pmod{F^{E_{l+1}}\ems}.
\]
Then, 
there exists $\theta_l\in F^{\ge E_{l+1}}\Rs$ such that  
$b_{(l+1)}:=b_{(l)}+b_l-\theta_lh_0\in \End^1(M)^s$ 
and $c_{(l+1)}:=c_{(l)}+c_l\in (\Rs)^2$ 
satisfy $\nu_{\Ms}(b_{(l+1)})\in\mG(r),$ and  
\[
mc(b_{(l+1)})\equiv c_{(l+1)}\cdot\id \pmod{F^{E_{l+1}}\ems},\quad \nu_{\Rs}(c_{(l+1)})>0,
\]
\[
\gto(\expp(b_{(l+1)}))\equiv r \pmod{F^{E_{l+1}}\Rs}.
\]
\end{lm}
\begin{proof}
   Since $\nu_{\Rs}(c_{(l)})>0$ and $\nu_{\Rs}(c_l)\ge E_{l+1},$ it follows that $\nu_{\Rs}(c_{(l+1)})>0.$ 
Let $\theta_l\in \Rs$ such that 
\[\theta_l\equiv  \gto(\exp(b_{(l)}))+\gto_1(b_l)-\gto(y_{b_{(l)}})-r \pmod{F^{E_{l+1}}\Rs}.\]
Since $\gto$ is normed and by the assumption 
\[
\gto(\expp(b_{(l)}))\equiv r \pmod{F^{E_l}\Rs},
\] it follows that $\theta_l\in F^{\ge E_{l+1}}\Rs.$ In addition, since $\gto$ is an operator of dimension $n,$ $|r|=1-n$ and $|\exp(b_{(l)})|=|b_l|=|y_{b_{(l)}}|=1,$ it follows that $|\theta_l|=1-n.$ 
Hence,
\[
b_{(l+1)}=b_{(l)}+b_l-\theta_lh_0\in\End^1(\Ms)
\]
satisfies $\nu_{\Ms}(b_{(l+1)})\in\mG(r),$ and
\begin{align*}
mc(b_{(l+1)})&\equiv mc(b_{(l)})+\delta(b_l)-\delta(\theta_lh_0)\\
&\equiv c_{(l)}\cdot \id+o_l+\delta(b_l)\\
&\equiv c_{(l)}\cdot \id+ c_l\cdot\id\\
&\equiv c_{(l+1)}\cdot \id  \pmod{F^{E_{l+1}}\ems}.
\end{align*}
Since $\nu_{\Rs}({c_{(l+1)}})>0$ and $ \nu_{\Ms}({b_{(l+1)}})>0 ,$ it follows that $y_{b_{(l)}}\equiv y_{b_{(l+1)}}\pmod{F^{E_{l+1}}\ems}.$ Therefore, since $\gto$ is normed, we get
\begin{align*}
\gto(\expp(b_{(l+1)}))&\equiv \gto(\exp(b_{(l+1)})-y_{b_{(l+1)}})\\
&\equiv \gto(\exp(b_{(l)}))+\gto_1(b_l)-\gto_1(\theta_lh_0)-\gto(y_{b_{(l+1)}})\\
&\equiv \gto(\exp(b_{(l)}))+\gto_1(b_l)-\theta_l-\gto(y_{b_{(l)}})\\
&\equiv r\quad \pmod{F^{E_{l+1}}\Rs}.
\end{align*}
\end{proof}

\begin{prop}\label{theorem:existence_bc}
Let $M$ be a normed Calabi-Yau spherical object.
Then, for any $r\in F^0\Rs$ with $|r|=1-n,$ there exists a bounding cochain $b\in \ems$ such that $\gto(\expp(b))=r.$   
\end{prop}
\begin{proof}
Fix $r\in F^0\Rs$ with $|r|=1-n.$ Write $\mG(r)$ in the form of a list as in~\eqref{sabab_def}.

Define $b_{(1)}:=r\cdot h_0.$ Hence, $|b_{(1)}|=1$ and
\[
mc(b_{(1)})\equiv 0 =c_{(1)}\cdot \id, \pmod{F^{E_1}\ems}, \quad c_{(1)}=0.
\]
In addition, $\nu_{\Rs}(c_{(1)})>0$ and
$\nu_{\Ms}(b_{(1)})>0.$
Since $c_{(1)}=0,$ we get
$y_{b_{(1)}}= 0.$ Thus,
\[
\gto(\expp(b_{(1)}))=\gto(\exp(b_{(1)}))\equiv\gto_1(b_{(1)})= r \pmod{F^{E_1}\Rs}.\]
We proceed by induction. Suppose we have $b_{(l)}\in \ems$ with $|b_{(l)}|=1,$ $\nu_{\Ms}(b_{(l)})\subset \mG(r),$  and
\[
mc(b_{(l)})\equiv c_{(l)}\cdot \id,  \pmod{F^{E_l}\ems} \quad c_{(l)}\in (\Rs)^2,\quad \nu_{\Rs}(c_{(l)})>0, \quad d_{\Rs}(c_{(l)})=0,
\]
\[\gto(\expp(b_{(l)}))\equiv r \pmod{F^{E_l}\ems}.
\]
Define the obstruction cochain $o_l$ as in~\eqref{obstruction_chain}. 
By Lemma~\ref{lemma_o_closed} we have 
\[
\delta(o_l)\equiv0 \pmod{F^{E_{l+1}}\ems},
\]
so $[o_l]\in H^*(\eems_{E_{l+1}}).$
By Lemma~\ref{lm:h_0} we have
$H^*(\eems_{E_{l+1}})\cong H^*(S^n;\Rs_{E_{l+1}}).$
Hence,
there exist $c_l\in F^{\ge E_{l+1}}\Rs$ and $b_l\in F^{\ge E_{l+1}}\End(M)^s$ with 
$|c_l|=2,$ $|b_l|=1,
$
such that 
\[\delta(b_l)\equiv -o_l+c_l\cdot \id \pmod{F^{E_{l+1}}\ems}.\]
Lemma~\ref{lm:computation} guarantees there exists $b_{(l+1)}\in \End^1(M)^s$ satisfying $\nu_{\Ms}(b_{(l+1)})\in\mG(r)$ and
\[
mc(b_{(l+1)})\equiv c_{(l+1)}\cdot \id \pmod{F^{E_{l+1}}\ems}, \quad c_{(l+1)}\in (\Rs)^2, \quad\nu_{\Rs}(c_{(l+1)})>0,
\] 
\[
\gto(\expp(b_{(l+1)}))\equiv r\pmod{F^{E_{l+1}}\Rs}, \quad b_{(l+1)}\equiv b_{(l)}\pmod{F^{E_l}\ems}. 
\]
Note that since $d_{\Rs}$ is the trivial differential, it follows that $d_{\Rs}(c_{(l+1)})=0.$ 
Thus, the inductive process gives rise to a convergent sequence $\{b_{(l)}\}_{l=1}^\infty$ where $b_{(l)}$ is a bounding cochain modulo $F^{E_l}\ems.$
Taking the limit as $l$ goes to infinity, we obtain 
\[
b:=\lim_l b_{(l)}\in \End^1(M)^s,
\quad c:=\lim_l c_{(l)}\in (\Rs)^2, \quad y_b:=\lim_ly_{b_{(l)}}\in \End^2(M)^s
\]
satisfying
$mc(b)=c\cdot \id,$ $y_b=H(\dcc(\exp(b)))$ and $\gto(\expp(b))=s.$  
In addition,
$d_{\Rs}(c)=0,$ $\nu_{\Rs}(c)>0$ and  $\nu_{\Ms}(b)>0.$ 

\end{proof}
\subsection{Pseudoisotopy}\label{section:pseudoisotopy}
Recall Definition~\ref{dfn:pseudoisotopy}. In this section, we construct a pseudoisotopy between $\eems$ and itself. This pseudoisotopy will be used in Section~\ref{section:gauge} in order to prove that the map $\varrho$ in Theorem~\ref{thm:bc_equivalence} is injective.

Write $I=[0,1],$ and let $A^*(I;\R)$ denote the complex of differential forms on $I.$ Let $\mathrm{ev}_0,\mathrm{ev}_1:A^*(I;\R)\to \R$ denote the evaluation maps at zero and one. Abbreviate $\eta|_i=\mathrm{ev}_i(\eta).$ Define the subcomplex
\[
A^*(I;\R)_{\mathrm{rel}}:=\{\eta\in A^*(I;\R)| \ \eta|_0=\eta|_1=0\}.
\]   
Let $H^*(I;\R)_{\mathrm{rel}}$ denote the cohomology associated to this subcomplex.
\begin{lm}\label{lm:cohomology_I}
    \[H^i(I;\R)_{\mathrm{rel}}=
         \left\{\begin{array}{ll}
        0, &  i\ne1,\\
        \langle [dt]\rangle_{\R}, & i=1.\\
        \end{array} \right.\]
\end{lm}
\begin{proof}
Consider the short exact sequence 
\[
0\longrightarrow A^*(I;\R)_{\mathrm{rel}}\hookrightarrow A^*(I;\R)\overset{\mathrm{ev_0}\oplus \mathrm{ev_1}}{\longrightarrow} \R^{\oplus 2}\longrightarrow0,
\]
and consider the long exact sequence it induces
\[
0\longrightarrow H^0(I;\R)_{\mathrm{rel}}\longrightarrow H^0(I;\R)\longrightarrow \R^{\oplus 2}\rightarrow H^1(I;\R)_{\mathrm{rel}}\longrightarrow H^1(I;\R)\longrightarrow\ldots.
\]
Since 
\[
H^0(I;\R)\cong \R, \qquad H^1(I;\R)\cong0,
\]
it follows that 
\[
H^0(I;\R)_{\mathrm{rel}}\cong 0, \qquad H^1(I;\R)_{\mathrm{rel}}\cong \R.
\]
Hence, it suffices to show  that $dt$ represents a non-trivial cohomology class of $H^1(I;\R)_{\mathrm{rel}}.$ Since $dt$ is closed we need to show it is non-exact.
Suppose there exists $f(t)\in A^0(I;\R)_{\mathrm{rel}}$ such that $df=dt.$ In particular, $[f(t)-t]\in H^0(I;\R).$ 
Thus, there exists $c\in\R$ such that $f(t)-t=c.$ But, this implies $f(0)=f(1)-1$ which is a contradiction since $f(t)\in A^0(I;\R)_{\mathrm{rel}}.$ 
\end{proof}
Define 
\[
\fUs_{\mathrm{rel}}:= A^*(I;\R)_{\mathrm{rel}}\otimes \ems, \quad \fRs_{\mathrm{rel}}:=A^*(I;\R)_{\mathrm{rel}}\otimes \Rs,
\]
\[
\fUs:= A^*(I;\R)\otimes \ems, \quad \fRs:=A^*(I;\R)\otimes \Rs,
\] 
Equip 
$\fUs_{\mathrm{rel}},$ $\fRs_{\mathrm{rel}}, $ $\fUs, \fRs$ with the valuations and the norms
$\nu_{\fUs_{\mathrm{rel}}},\nu_{\fRs_{\mathrm{rel}}}, \nu_{\fUs}, \nu_{\fRs} $ induced by $\nu_{\Ms}, \nu_{\Rs}$ and the trivial valuations on $A^*(I;\R)_{\mathrm{rel}}$ and $A^*(I;\R).$ 
These algebras are equipped with the differentials induced by the tensor products. This makes $\fRs_{\mathrm{rel}},$ $\fRs$ into commutative valued DGAs over $\R,$ and $\fUs_{\mathrm{rel}},$ $\fUs$ into normed unital DGAs over $\fRs_{\mathrm{rel}},$ $\fRs$ respectively.
The differential $\fpar$ and the multiplication on both $\fRs$ and $\fRs_{\mathrm{rel}}$ are given by
\[\fpar(\eta\otimes r):=d\eta\otimes r, \quad
(\eta_1\otimes r_1 )\cdot( \eta_2\otimes r_2):= \eta_1\wedge\eta_2\otimes r_1r_2,
\]
and the differential $\widetilde{\delta}$ and the multiplication on both $\fUs$ and  $\fUs_{\mathrm{rel}}$ are given by
\[
\widetilde{\delta}(\eta\otimes \alpha):=d\eta \otimes\alpha+(-1)^{|\eta|} \eta\otimes\delta(\alpha),\quad (\eta_1 \otimes\alpha_1) \cdot (\eta_2\otimes \alpha_2):=(-1)^{|\eta_2||\alpha_1|}\eta_1\wedge\eta_2 \otimes\alpha_1\alpha_2.
\]
 
Define $\mathrm{eval}^i:\fRs\to \Rs$ by $\mathrm{eval}^i:=\mathrm{ev}_i\otimes \id_{\Rs},$  
and 
$\mathfrak{eval}^i:\fUs\to \ems$ by $\mathfrak{eval}^i=\mathrm{ev}_i\otimes\id_{\Ms},$ for $i=0,1.$    
    \begin{lm}\label{lm:dga_morphism}
        $\mathrm{eval}^i$ are DGA morphisms, and $\mathfrak{eval}^i$ are unital DGA morphisms over $\mathrm{eval}^i.$
    \end{lm}
\begin{proof}
        Let $\eta_1,\eta_2\in A^*(I;\R)$ and $r_1,r_2\in\Rs.$  We get
   \begin{align*}
       \mathrm{eval}^i\big((\eta_1\otimes r_1) \cdot( \eta_2\otimes r_2)\big)&=\mathrm{eval}^i(\eta_1\wedge\eta_2\otimes r_1r_2)\\
       &=\eta_1|_i\eta_2|_ir_1r_2\\
     &= (\eta_1|_i r_1)\cdot (\eta_2|_i r_2)\\
     &=\mathrm{eval}^i(\eta_1\otimes r_1)\cdot \mathrm{eval}^i(\eta_2\otimes r_2), 
   \end{align*}
   and 
   \[
   \mathrm{eval}^i\circ \fpar(\eta_1\otimes r_1)= \mathrm{eval}^i(d\eta_1\otimes r_1)=0=d_{\Rs}\circ \mathrm{eval}^i(\eta_1\otimes r_1).
   \]
Hence, $\mathrm{eval}^i$ are DGA morphisms. Let $ \alpha_1,\alpha_2\in \ems.$ We get
    \begin{align*}
     \eval^i\big((\eta_1\otimes\alpha_1)\cdot( \eta_2\otimes\alpha_2)\big)&= 
     \eval^i((-1)^{|\eta_2||\alpha_1|}\eta_1\wedge \eta_2 \otimes \alpha_1 \alpha_2)\\
     &=(-1)^{|\eta_2||\alpha_1|}\eta_1|_i\eta_2|_i\alpha_1\alpha_2\\
     &=\eta_1|_i\eta_2|_i\alpha_1\alpha_2\\
     &= (\eta_1|_i \alpha_1)\cdot (\eta_2|_i \alpha_2)\\
     &=\eval^i(\eta_1\otimes \alpha_1)\cdot \eval^i(\eta_2\otimes \alpha_2).
 \end{align*}
Since $\eta|_i=0$ for every $\eta\in A^1(I;\R),$ it follows that $(-1)^{|\eta|}\eta|_i=\eta|_i$ for every $\eta\in A^*(I;\R).$
    Thus,
    \begin{align*}
        \mathfrak{eval}^i\circ \widetilde{\delta}(\eta_1\otimes \alpha_1)&=\mathfrak{eval}^i\big(d\eta_1\otimes \alpha_1+(-1)^{|\eta|}\eta_1\otimes \delta(\alpha_1) \big)\\
    &=(-1)^{|\eta_1|}\eta_1|_i \delta(\alpha_1)\\
    &=\eta_1|_i \delta(\alpha_1)\\
    &=\delta\circ \mathrm{eval}^i(\eta_1\otimes \alpha_1).
    \end{align*}
Hence, $\eval^i$ are DGA morphisms.
Note that $\eval^i|_{\fRs}=\mathrm{eval}^i.$ Hence, $\eval^i$ are morphisms over $\mathrm{eval}^i.$  
    In addition $\eval^i(1\otimes \id_{\Ms})=\id_{\Ms},$ so $\eval^i$ are  unital morphisms.
\end{proof}
\begin{lm}\label{lm:eval_maps_homotopic}
   The maps $\mathrm{eval}^0$ and $\mathrm{eval}^1$ are homotopic.
\end{lm}
\begin{proof}
Define a map $h:\fRs\to \Rs$ acting on pure tensors as
     \[h(r)=
         \left\{\begin{array}{ll}
        0, &  r=f(t)\otimes r',\\
        r'\int_0^1f(t)dt, & r=f(t)dt\otimes r'.\\
        \end{array} \right.\]
        Since $d_{\Rs}$ is the trivial differential, 
        we get $h\circ\fpar+d_{\Rs}\circ h=\mathrm{eval}^1-\mathrm{eval}^0.$ Hence, $h$ is a homotopy between $\mathrm{eval}^0$ and $\mathrm{eval}^1.$
\end{proof}
In the following we write $\eta\alpha$ for an element $\eta\otimes\alpha\in \fUs,$ and by abuse of notations we denote also by $\id$ the identity element of both $\fUs_{\mathrm{rel}}$ and $\fUs.$ 
Consider the cyclic complex $(\CC^\lambda_*(\fUs),\dcc).$ Let $\eta\in A^*(I;\R).$ We have 
\begin{equation}\label{eq:commutative}
\eta\cdot\alpha_l\otimes\ldots\otimes\alpha_1= (-1)^{\sum_{j=i+1}^l|\eta|(|\alpha_j|+1)}\alpha_l\otimes\ldots\otimes\eta\alpha_i\otimes\ldots\otimes\alpha_1.
\end{equation}
Define an operator $\thet:\CC^\lambda_*(\fUs)\to \fRs$ acting on generators by
\[
\thet_l(\eta\cdot \alpha_l\otimes\ldots \otimes \alpha_1 ):=\eta\cdot \gto_l(\alpha_l\otimes\ldots\otimes \alpha_1)
\]
for every $l\ge1.$ 
\begin{lm}
    The operator $\thet$ is well defined.
\end{lm}
\begin{proof}
    We need to show that $\thet\circ (1-t)=0.$
    We have,
    \begin{align*}
        \thet(\eta_l\alpha_l\otimes\ldots\otimes \eta_1\alpha_1)&=\thet\big((-1)^{\star_d}\eta_l\wedge\ldots\wedge\eta_1\cdot \alpha_l\otimes\ldots\otimes\alpha_1\big)\\
        &=(-1)^{\star_d}\eta_l\wedge\ldots\wedge\eta_1\cdot \gto(\alpha_l\otimes\ldots\otimes\alpha_1),
    \end{align*}
    where $\star_d=\sum_{i=1}^{l-1}\sum_{j=i+1}^l(|\eta_i||\alpha_j|+|\eta_i|).$
    In addition,
    \begin{align*}
        \thet\circ t(\eta_l\alpha_l\otimes\ldots\otimes \eta_1\alpha_1)&=(-1)^A \thet(\eta_{l-1}\alpha_{l-1}\otimes\ldots\otimes\eta_1\alpha_1\otimes \eta_l\alpha_l)\\
        &=(-1)^{A+B}\thet(\eta_{l-1}\wedge\ldots\wedge\eta_1\wedge\eta_l\cdot \alpha_{l-1}\otimes \ldots\otimes\alpha_1\otimes\alpha_l)\\
        &=(-1)^{A+B+C}\thet(\eta_l\wedge\ldots\wedge\eta_l\cdot \alpha_{l-1}\otimes \ldots\otimes\alpha_1\otimes\alpha_l)\\
        &=(-1)^{A+B+C+D}\thet\big(\eta_l\wedge\ldots\wedge\eta_l\cdot t(\alpha_l\otimes \ldots\otimes\alpha_1)\big)\\
        &=(-1)^{A+B+C+D}\eta_l\wedge\ldots\wedge\eta_l\cdot \gto(\alpha_l\otimes \ldots\otimes\alpha_1),
    \end{align*}
    where we used in the last equality that $\gto\circ(1-t)=0,$ and
\[
A=(|\eta_l|+|\alpha_l|+1)\big( \sum_{i=1}^l(|\eta_i|+|\alpha_i|+1)+ (|\eta_l|+|\alpha_l|+1)    \big),\]
\[
B=\sum_{i=1}^{l-2}\sum_{j=i+1}^{l-1}(|\eta_i||\alpha_j|+|\eta_i|)+\sum_{i=1}^l(|\eta_l||\alpha_j|+|\eta_l|),
\]
\[
C=|\eta_l|\sum_{i=1}^{l-1}|\eta_i|, \quad D=(|\alpha_l|+1)\big( \sum_{i=1}^l(|\alpha_i|+1)+ (|\alpha_l|+1)    \big).
\]    
Since $\star_d\equiv A+B+C+D\pmod{2},$ it follows that $\thet=\thet\circ t.$ 
\end{proof}
Our goal is to show that $\fUs$ is a pseudoisotopy between $\ems$ and itself. So, we need to show that $\fpar\circ\thet=\thet\circ\widetilde{\dcc}.$ In order to prove this,  we need the following lemma.
\begin{lm}\label{lm:fdcc_equal}
   \begin{multline*}
\dcc(\eta_l\alpha_l\otimes\ldots\otimes\eta_1\alpha_1)=(-1)^{\star_d+\sum_{j=1}^l|\eta_j|}\eta_l\wedge\ldots\wedge\eta_1\cdot\dcc(\alpha_l\otimes\ldots\otimes\alpha_1)\\
+(-1)^{\star_d}d\eta_l\wedge\ldots\wedge\eta_1\cdot \alpha_l\otimes\ldots\otimes\alpha_1
 +(-1)^{\star_d}\sum_{i=1}^{l-1}(-1)^{\sum_{j=i+1}^l|\eta_j|}\eta_l\wedge\ldots\wedge d\eta_i\wedge\ldots\wedge\eta_1\cdot\alpha_l\otimes \ldots \otimes\alpha_1,
\end{multline*}
where $\star_d=\sum_{j=1}^{l-1}\sum_{k=j+1}^l(|\eta_j||\alpha_k|+|\eta_j|).$
\end{lm}
\begin{proof}
We have
    \begin{multline*}\label{b_mu_1}
        b(\mu_1)(\eta_l\alpha_l\otimes\ldots\otimes \eta_1\alpha_1)=(d\eta_l\alpha_l+(-1)^{|\eta_l|}\eta_l\delta(\alpha_l))\otimes\ldots\otimes \eta_1\alpha_1+\\
        +\sum_{i=1}^{l-1}(-1)^{\epsilon_{i+1}}\eta_1\alpha_1\otimes \ldots\otimes(d\eta_i\alpha_i+(-1)^{|\eta_i|}\eta_i\delta(\alpha_i)) \otimes \ldots \otimes\eta_1\alpha_1.
     \end{multline*}
First, consider the first summand of $b(\mu_1)$. We obtain
\begin{multline*}
    (d\eta_l\alpha_l+(-1)^{|\eta_l|}\eta_l\delta(\alpha_l))\otimes\ldots\otimes \eta_1\alpha_1=\\
    = (-1)^{\star_\delta}\eta_l\wedge\ldots\wedge\eta_1\cdot \delta(\alpha_l)\otimes\ldots\otimes\alpha_1 +(-1)^{\star_d}d\eta_l\wedge\ldots\wedge\eta_1\cdot \alpha_l\otimes\ldots\otimes\alpha_1,
\end{multline*}
where 
\[
    \star_d=\sum_{j=1}^{l-1}\sum_{k=j+1}^l(|\eta_j||\alpha_k|+|\eta_j|),
\]
and 
\begin{align*}
    \star_\delta&=|\eta_l|+ \sum_{j=1}^{l-2}\sum_{k=j+1}^{l-1}(|\eta_j||\alpha_k|+|\eta_j|)+\sum_{j=1}^{l-1}|\eta_j||\alpha_l|\\
    &\equiv+ \sum_{j=1}^{l-1}\sum_{k=j+1}^l(|\eta_j||\alpha_k|+|\eta_j|)+\sum_{j=1}^l|\eta_j| \\
    &\equiv \star_d+\sum_{j=1}^l|\eta_j|\pmod{2}.
\end{align*}
Next, consider the second summand of $b(\mu_1)$. We obtain
\begin{multline*}
    \sum_{i=1}^{l-1}(-1)^{\epsilon_{i+1}}\eta_l\alpha_l\otimes \ldots\otimes(d\eta_i\alpha_i+(-1)^{|\eta_i|}\eta_i\delta(\alpha_i)) \otimes \ldots \otimes\eta_1\alpha_1=\\
    \sum_{i=1}^{l-1}(-1)^{\diamond_{\delta_i}}\eta_1\wedge\ldots\wedge\eta_l\cdot\alpha_1\otimes \ldots\otimes \delta(\alpha_i)\otimes \ldots \otimes\alpha_1+\\
    +\sum_{i=1}^{l-1}(-1)^{\diamond_{d_i}}\eta_1\wedge\ldots\wedge d\eta_i\wedge\ldots\wedge\eta_l\cdot\alpha_l\otimes \ldots\otimes  \alpha_1,
\end{multline*}
where 
\begin{align*}
    \diamond_{\delta_i}&=\epsilon_{i+1}+|\eta_i|+\sum_{j=i}^{l-1}\sum_{k=j+1}^l(|\eta_j||\alpha_k|+|\eta_j|)+\sum_{j=1}^{i-1}\sum_{\substack{k=j+1\\ k\ne i}}^l(|\eta_j||\alpha_k|+|\eta_j|)+ \sum_{j=1}^{i-1}|\eta_j||\alpha_i|\\
    &\equiv\star_d+\sum_{j=i+1}^l(|\alpha_j|+1)+\sum_{j=1}^l|\eta_j| \pmod{2},
\end{align*}
and
\begin{align*}
\diamond_{d_i}&=\epsilon_{i+1}+\sum_{j=i+1}^{l-1}\sum_{k=j+1}^l(|\eta_j||\alpha_k|+|\eta_j)|+ \sum_{k=i+1}^l(|\eta_i||\alpha_k|+|\alpha_k|+|\eta_i|+1)\\
&+\sum_{j=1}^{i-1}\sum_{k=j+1}^l(|\eta_j||\alpha_k|+|\eta_j|)\\
&\equiv \star_d+\sum_{j=i+1}^l |\eta_j| \pmod{2}.
\end{align*}
So,
\begin{multline*}
    b(\mu_1)(\eta_l\alpha_l\otimes\ldots\otimes\eta_1\alpha_1)= (-1)^{\star_d+\sum_{j=1}^l|\eta_j|}\eta_l\wedge\ldots\wedge\eta_1\cdot b(\mu_1)(\alpha_l\otimes\ldots\otimes \alpha_1)+\\
    + (-1)^{\star_d}d\eta_l\wedge\ldots\wedge\eta_1\cdot \alpha_l\otimes\ldots\otimes\alpha_1+ (-1)^{\star_d}\sum_{i=1}^{l-1}(-1)^{\sum_{j=i+1}^l|\eta_j|}\eta_1\wedge\ldots\wedge d\eta_i\wedge\ldots\eta_l\cdot\alpha_1\otimes \ldots\otimes\alpha_1.
\end{multline*}
With similar arguments rely on~\eqref{eq:commutative}, one can show
\[
 b(\mu_2)(\eta_l\alpha_l\otimes\ldots\otimes\eta_1\alpha_1)=(-1)^{\star_d+\sum_{j=1}^l|\eta_j|}\eta_l\wedge\ldots\wedge\eta_1\cdot b(\mu_2)(\alpha_l\otimes\ldots\otimes\alpha_1).
\]
This completes the proof.
\end{proof}
\begin{lm}\label{lm:thet_commuting_differentials}
   $\fpar\circ\thet=\thet\circ \widetilde{\dcc}.$
\end{lm}
\begin{proof}
By~\eqref{eq:commutative}, we have
\[
\eta_l\alpha_l\otimes \ldots\otimes \eta_1\alpha_1= (-1)^{\star_d}\eta_l\wedge\ldots\wedge\eta_1\cdot \alpha_l\otimes\ldots\otimes \alpha_1,
\]
where $\star_d={\sum_{j=1}^{l-1}\sum_{k=j+1}^l(|\eta_j||\alpha_k|+|\eta_j|}).$
Hence, 
\begin{align*}
    \fpar\circ \thet_l (\eta_l\alpha_l\otimes \ldots\otimes \eta_1\alpha_1)&=(-1)^{\star_d} d(\eta_l\wedge\ldots\wedge\eta_1)\cdot\gto_l(\alpha_l\otimes\ldots\otimes \alpha_1)\\
    &=(-1)^{\star_d}\bigg(d\eta_l\wedge\ldots\wedge\eta_1+\\
    &+\sum_{i=1}^{l-1}(-1)^{\sum_{j=i+1}^l|\eta_j|}\eta_l\wedge\ldots \wedge d\eta_i\wedge\ldots\wedge\eta_1
    \bigg)\cdot\gto_l(\alpha_l\otimes\ldots\otimes \alpha_1)\\
    &=\thet_l\bigg((-1)^{\star_d+\sum_{j=1}^l|\eta_j|}\eta_l\wedge\ldots\wedge\eta_1\cdot\dcc(\alpha_l\otimes\ldots\otimes\alpha_1)+\\
&+(-1)^{\star_d}d\eta_l\wedge\ldots\wedge\eta_1\cdot \alpha_l\otimes\ldots\otimes\alpha_1
 +\\
 &+(-1)^{\star_d}\sum_{i=1}^{l-1}(-1)^{\sum_{j=i+1}^l|\eta_j|}\eta_l\wedge\ldots\wedge d\eta_i\wedge\ldots\wedge\eta_1\cdot\alpha_l\otimes \ldots \otimes\alpha_1\bigg)\\
    &\overset{\ref{lm:fdcc_equal}}= \thet_l\circ\dcc(\eta_l\alpha_l\otimes \ldots\otimes \eta_1\alpha_1),
\end{align*}
where the third equality follows from $\gto\circ\dcc=0.$

\end{proof}
\begin{lm}\label{lm:pseudo_thet}
    $\fUs$ is a unital pseudoisotopy between $\ems$ and itself.
\end{lm}
\begin{proof}
Since $\gto$ is cohomologically unital, $\thet$ is cohomologically unital as well. In addition, it follows from the definitions of $\thet$ and $\mathrm{eval}^i,$ $\eval^i$ that $\gto\circ\eval^i=\mathrm{eval}^i\circ\thet.$ Hence, the claim follows from Lemmas~\ref{lm:dga_morphism},~\ref{lm:eval_maps_homotopic},~\ref{lm:thet_commuting_differentials}.
\end{proof}
Note that since $\gto$ is normed and $A^*(I;\R)$ is equipped with the trivial valuation, it follows that $\thet$ is a normed operator.  
Define 
\[
\widetilde{mc}(a):= \widetilde{\delta}(a)-a^2, \quad a\in \fUs.
\]
\subsection{Gauge equivalence of bounding cochains}\label{section:gauge}
The goal of the present section is to prove that the map $\varrho$ in Theorem~\ref{thm:bc_equivalence} is well defined and injective. Together with Proposition~\ref{theorem:existence_bc}, this proves Theorem~\ref{thm:bc_equivalence}.

We first prove that the map $\varrho$ is well defined. Recall Definition~\ref{dfn_gauge_equivalence}. 

\begin{proof}[Proof of Lemma~\ref{building_equiv_2}]
By Lemma~\ref{remain_c} we have $[c_0]=[c_1]\in H^*(\Rs;d_{\Rs})$. Since $d_{\Rs}$ is the trivial differential, it follows that $c_0=c_1.$

Let $\widetilde{b}$ be a pseudoisotopy between $b_0$ and $b_1,$
and let $h:\fS\to \Rs$ denote the homotopy between $\mathrm{eval}^0$ and $\mathrm{eval}^1,$ where $(\fS,\fpar_{\fS},\nu_{\fS})$ is a valued DGA over $\Rs.$ Thus, 
\[
h\circ\fpar_\fS=h\circ\fpar_\fS+ d_{\Rs}\circ h= \text{eval}^1- \text{eval}^0.
\]
Let $(\fB,\widetilde{m}_\fB,\nu_{\fB})$ denote the normed DGA over $\fS,$ and
let $\widetilde{\sigma}:\CC_*^\lambda(\fB)\to\fS$ denote the cohomologically unital operator satisfying $\fpar_\fS\circ\widetilde{\sigma}=\widetilde{\sigma}\circ\dcc,$ and $\gto\circ\eval^i=\mathrm{eval}^i\circ\widetilde{\sigma}.$ 
Hence,
\begin{align*}
    \gto(\expp(b_1))- \gto(\expp(b_0))&= (\gto\circ \eval^1 -\gto\circ \eval^0)(\expp(\widetilde{b}))\\
    &=((\mathrm{eval}^1-\mathrm{eval}^0)\circ\widetilde{\sigma})(\expp(\widetilde{b}))\\
    &= h\circ\fpar_\fS \circ\widetilde{\sigma}(\expp(\widetilde{b}))\\
    &=h\circ \widetilde{\sigma}\circ\dcc(\expp(\widetilde{b}))\\
    &=h\circ \widetilde{\sigma}\circ\dcc(\exp(\widetilde{b})-y_{\widetilde{b}})\\
    &=-h\circ \widetilde{\sigma}(z_{\widetilde{b}})\\
    &=0.
\end{align*}
\end{proof}

Let $\alpha_0,\alpha_1\in F^0\ems,$ and consider $\mG(\alpha_0,\alpha_1)\subset \R_{>0}$ ordered as in~\eqref{sabab_def}.
Let $l\ge1.$ Suppose we have $\widetilde{b}_{(l)}\in F^0\fUs$ with $|\widetilde{b}_{(l)}|=1,$ $\nu_{\fUs}(\widetilde{b}_{(l)})\in \mG(\alpha_0,\alpha_1),$ 
and 
\[
\widetilde{mc}(\widetilde{b}_{(l)})\equiv \widetilde{c}_{(l)}\cdot \id,  \pmod{F^{E_l}\fUs}, \quad \widetilde{c}_{(l)}\in (\fRs)^2,\quad 
\fpar\widetilde{c}_{(l)}=0,\quad \nu_{\fRs}(\widetilde{c}_{(l)})>0. 
\]
Define the obstruction cochain $F^{\ge E_{l+1}}\widetilde{o}_l\in \fUs$ to be an endomorphism satisfying
\begin{equation}\label{o_tilde_j}
\widetilde{mc}(\widetilde{b}_{(l)})\equiv \widetilde{c}_{(l)}\cdot \id+\widetilde{o}_l \pmod{F^{E_{l+1}}\fUs}.
\end{equation} 
Let $a\in \fUs,$ abbreviate $a|_i=\eval^i(a)$ where $i=0,1.$
The proof of Proposition~\ref{building_equiv_1} which shows the map $\varrho$ is injective is based on the following lemmas.
\begin{lm}\label{lm:o_tilde_is_closed}
$
\widetilde{\delta}(\widetilde{o}_l)\equiv0 \quad \pmod{F^{E_{l+1}}\fUs}.
$
\end{lm}
\begin{proof}
   The follows from Lemma~\ref{lm:d_a_q0}.
\end{proof}
Define
\[
y_{\widetilde{b}_{(l)}}:=H\big(\sum_{l=1}^\infty\widetilde{c}_{(l)}\cdot \id\otimes \widetilde{b}_{(l)}^{\otimes l-1}\big),
\]
and $\expp(\widetilde{b}_{(l)})=\exp(\widetilde{b}_{(l)})-y_{\widetilde{b}_{(l)}}.$
Let $b\in\Mmm(\ems,c).$ In this section we choose $y_b=H(\dcc(\exp(b))).$
\begin{lm}\label{lm:o_is_closed_reduced}
Let $b_i\in \Mmm(\ems,c_i)$ for $i=0,1$ such that $\gto(\expp(b_0))=\gto(\expp(b_1)).$
Suppose $\widetilde{b}_{(l)}|_0=b_0,$  $\widetilde{b}_{(l)}|_1=b_1.$  Then,
$\int_0^1\thet_1(\widetilde{o}_l)\equiv 0 \pmod{F^{E_{l+1}}\Rs}.$
\end{lm}
\begin{proof}
By Lemma~\ref{lm:dcc_y_b_mod_e} there exists $z_{\widetilde{b}}\in\CC_*^\lambda(\fRs)$ such that
\[
\dcc(y_{\widetilde{b}_{(l)}})\equiv \dcc\exp(\widetilde{b}_{(l)})+z_{\widetilde{b}_{(l)}}-\widetilde{o}_l  \pmod{F^{E_{l+1}}\fUs}.
\]
Hence, since $\thet$ is normed and vanishes on $\CC^\lambda_*(\fRs),$ we obtain

\begin{align*}
    \int_0^1\thet_1(\widetilde{o}_l)&\equiv\int_0^1\thet(\dcc\exp(\widetilde{b}_{(l)}))-\int_0^1\thet(\dcc(y_{\widetilde{b}_{(l)}}))+ \int_0^1\thet(z_{\widetilde{b}_{(l)}})\\
    &\equiv \int_0^1\thet(\dcc\expp(\widetilde{b}_{(l)}))\\
     &\equiv \int_0^1\fpar\thet(\expp(\widetilde{b}_{(l)}))\\
    &\equiv\gto(\expp(\widetilde{b}_{(l)}|_1))-\gto(\expp(\widetilde{b}_{(l)}|_0))\\
    &\equiv \gto(\expp(b_1))-\gto(\expp(b_0))\\
    &=0
    \pmod{F^{E_{l+1}}\Rs}.
\end{align*}

\end{proof}

\begin{lm}\label{lm:o_c}
   Let $M$ be a spherical object, 
    and suppose $\widetilde{o}_l|_0,\widetilde{o}_l|_1\in \Rs\cdot \id.$ Then,
    $\widetilde{o}_l|_0\equiv \widetilde{o}_l|_1 \pmod{F^{E_{l+1}}\ems}.$
\end{lm}
\begin{proof}
    Write
    \[
    \widetilde{o}_l=\sum_{i=1}^{m_0}f_i(t)\alpha_i+\sum_{i=1}^{m_1}g_i(t)dt\cdot\alpha'_i, \quad \alpha_i,\alpha'_i\in \ems, \quad f_i(t),g_i(t)dt\in A^*(I;\R).
    \]
By Lemma~\ref{lm:o_tilde_is_closed},
\[
0\equiv \widetilde{\delta}(\widetilde{o}_l)= \sum_{i=1}^{m_0}\frac{\partial f_i}{\partial t}dt\cdot \alpha_i+\sum_{i=1}^{m_0}f_i(t)\delta(\alpha_i)-\sum_{i=1}^{m_1}g_i(t)dt\cdot \delta(\alpha'_i)\pmod{F^{E_{l+1}}\fUs},
\]
so
\[
\sum_{i=1}^{m_0}\frac{\partial f_i}{\partial t}dt \cdot\alpha_i\equiv\sum_{i=1}^{m_1}g_i(t)dt\cdot\delta(\alpha'_i) \pmod{F^{E_{l+1}}\fUs}.
\]
We get
\[
\widetilde{o}_l|_1-\widetilde{o}_l|_0=
\sum_{i=1}^{m_0}\big(f_i(1)-f_i(0)\big)\alpha_i\equiv\delta \big(  \sum_{i=1}^{m_1}\alpha'_i\int_0^1g_i(t)dt \big) \pmod{F^{E_{l+1}}\Rs}. 
\]
Hence, $[0]=[\widetilde{o}_l|_1-\widetilde{o}_l|_0]\in \Rs_{E_{l+1}}\cdot \id.$ 
Therefore,
\[
\widetilde{o}_l|_1\equiv \widetilde{o}_l|_0 \pmod{F^{E_{l+1}}\Rs}.
\]
\end{proof}
Let $E\in\ima\nu_{\Ms}\backslash\{\infty\}.$ Since $\gto$ is normed, it descends to $\HH^\lambda_*(\ems_E).$ We denote by $\gto_1^E$ the induced operator.

\begin{lm}\label{lm:co_rel_0}
 Let $M$ be a normed Calabi-Yau spherical object. Then, $\ker\gto_1^E=[\id]\cdot \Rs_E.$ 
\end{lm}
  
\begin{proof}
The map $H^*(\ems_0)\to H^*(\ems_E)$ given by multiplication by $r_E\in\Rs$ where $\nu_{\Rs}(r_E)=E$ is an isomorphism. Hence,
since $M$ is a spherical object, it follows that $H^*(\ems_E)\cong [\id]\cdot \Rs_E\oplus [h_0]\cdot \Rs_E.$ Hence,  $\ker\gto_1^E=[\id]\cdot \Rs_E.$ 
\end{proof}
Let $E\in\ima\nu_{\Ms}\backslash\{\infty\}.$ Since $\thet$ is normed, 
it descends to $\HH^\lambda_1(\fUs_{\mathrm{rel}_E}).$ We denote by $\thet_1^E$ the induced operator.
 
\begin{lm}\label{lm:primitive}
Let $M$ be a normed Calabi-Yau spherical object. Let $E\in\ima\nu_{\Ms}\backslash\{\infty\},$ and let $[\widetilde{a}]\in H^{2i}(\fUs_{\mathrm{rel}_E})$ for $i\in\Z$ such that $\int_0^1\thet_1^E([\widetilde{a}])=0.$ Then, $[\widetilde{a}]=[0].$
\end{lm}
\begin{proof}
By Lemma~\ref{lm:cohomology_I}
it follows that
$\ker\int_0^1\thet_1^E$ is generated by elements of the form $dt\cdot [\widetilde{\alpha}],$ where $[\widetilde{\alpha}]\in \ker\gto_1^E.$
By Lemma~\ref{lm:co_rel_0} we have  $\ker\gto_1^E=R_E\cdot[\id].$ So, since $R_E$ consists elements of even degree, it follows that $\ker\int_0^1\thet_1^E$ consists elements of odd degree. Therefore, $[\widetilde{a}]=[0].$ 
\end{proof}

\begin{lm}\label{lm:primitive_form} 
Let $M$ be a normed Calabi-Yau spherical object, and suppose $\widetilde{b}_{(l)}|_i$ are bounding cochains for $i=0,1.$ 
Then, 
there exist
$k_l\in F^{\ge E_{l+1}}\Rs$ and $\widetilde{b}_l\in F^{\ge E_{l+1}}\fUs_{\mathrm{rel}}$ with $
|k_l|=2,$ $|\widetilde{b}_l|=1,
$
such that $k_l\cdot\id=\widetilde{o}_l|_1,$ and 
\[\widetilde{\delta}(\widetilde{b}_l)\equiv -\widetilde{o}_l+k_l\cdot \id \pmod{F^{E_{l+1}}\fUs_{\mathrm{rel}}}.
    \]
\end{lm} 
\begin{proof}
Since $\widetilde{b}_{(l)}|_i$ are bounding cochains, it follows that $\widetilde{o}_l|_0, \widetilde{o}_l|_1\in \Rs\cdot \id.$ So, there exists $k_l\in F^{\ge E_{l+1}}\Rs$ with $|k_l|=2$ satisfying $k_l\cdot \id=\widetilde{o}_l|_1.$ In addition, by Lemma~\ref{lm:o_c} we have
\[
k_l\cdot\id:=\widetilde{o}_l|_1\equiv\widetilde{o}_l|_0\pmod{F^{E_{l+1}}\ems}.
\] 
Thus, $\widetilde{o}_l-k_l\cdot \id\in F^{\ge E_{l+1}}\fUs_{\mathrm{rel}}.$
By
Lemma~\ref{lm:o_tilde_is_closed} $\widetilde{o}_l-k_l\cdot \id$ represents a cohomology class of $H^2(\fUs_{\mathrm{rel}_{E_{l+1}}}),$   
and by Lemma~\ref{lm:o_is_closed_reduced} we have $\widetilde{o}_l-k_l\cdot \id\in\ker\int_0^1\thet_1^{E_{l+1}}.$ Hence,
by Lemma~\ref{lm:primitive} we get $[\widetilde{o}_l-k_l\cdot\id]=[0]$ which completes the proof.
\end{proof}

\begin{lm}\label{lm:c}
  Let $M$ be a normed Calabi-Yau spherical object.
Suppose $\widetilde{b}_{(l)}|_i\in \Mmm(\ems,c_i)$  
for $i=0,1.$ 
Then, there exist $k_l\in F^{\ge E_{l+1}}\Rs$ with $|k_l|=2$ such that 
$k_l\cdot\id=\widetilde{o}_l|_1$ and
\[\widetilde{c}_{(l+1)}:=\widetilde{c}_{(l)}+k_l\in(\Rs)^2 \] 
satisfies
$\widetilde{c}_{(l+1)}|_i\equiv c_i \pmod{F^{E_{l+1}}\Rs}$ for $i=0,1.$ 
\end{lm}
\begin{proof}
Since $\widetilde{b}_{(l)}|_i\in \Mmm(\ems,c_i),$ it follows that $\widetilde{c}_{(l)}|_i\equiv c_i \pmod{F^{E_l}\Rs},$  $\widetilde{o}_l|_1,\widetilde{o}_l|_0 \in \Rs\cdot \id$ and
\[
c_i\equiv \widetilde{c}_{(l)}+\widetilde{o}_l|_i \pmod{F^{E_{l+1}}\Rs}, \quad i=0,1.
\]
Let $k_l\in F^{\ge E_{l+1}}\Rs$ with $|k_l|=2$ such that $k_l\cdot \id=\widetilde{o}_l|_1.$
Write
\[\widetilde{c}_{(l+1)}:=\widetilde{c}_{(l)}+k_l\in (\Rs)^2.
\] 
By Lemma~~\ref{lm:o_c} we have
$k_l\cdot\id:=\widetilde{o}_l|_1\equiv\widetilde{o}_l|_0\pmod{F^{E_{l+1}}\ems}.$
Hence, 
\[
\widetilde{c}_{(l+1)}|_i\equiv c_i \pmod{F^{E_{l+1}}\Rs}, \quad i=0,1.
\]
\end{proof}

\begin{prop}\label{building_equiv_1}
 Let $r\in F^0\Rs$ with $|r|=1-n.$ Let $M$ be a normed Calabi-Yau spherical object.
Let $b_i\in \Mmm(\ems,c_i)$ where $i=0,1,$ such that $\gto(\expp({b_i}))=r.$ Then $b_0\sim_{\fUs}b_1.$
\end{prop}

\begin{proof}
We construct a pseudoisotopy $\widetilde{b}.$ Write $\mG(b_0,b_1)$ in the form of a list as in~\eqref{sabab_def}.

First, note that $\delta(b_1-b_0)\equiv (c_1-c_0)\cdot \id \pmod{F^{E_1}\ems}.$
Thus, $[0]=[(c_1-c_0)\cdot \id]\in R_{E_1}\cdot\id$ which  
implies that $c_0\equiv c_1 \pmod{F^{E_1}R}.$ So, $\delta(b_1-b_0)\equiv0\pmod{F^{E_1}\ems}.$
Since $\gto$ is normed we obtain
\begin{align*}
    0&\equiv \gto(\expp(b_1))-\gto(\expp(b_0))\\
    &\equiv \gto(\exp(b_1))-\gto(\exp(b_0))\\
    &\equiv \gto_1(b_1-b_0)\pmod{F^{E_1}\Rs}.
\end{align*}
Thus, since $\gto_1$ is non-degenerate, it follows that $[b_1-b_0]=[0]\in H^1(\eems_{E_1}).$ 
Hence, there exists $\eta\in F^{\ge 0}\End^0(M)^s$ with $|\eta|=0$ such that $\delta(\eta)\equiv b_1-b_0 \pmod{F^{E_1}\ems}.$
Define
\[
\widetilde{b}_{(1)}:= t\cdot b_1+ (1-t)\cdot b_0 +dt\cdot\eta ,   
\]
so $\nu_{\fUs}(\widetilde{b}_{(1)})\in\mG(b_0,b_1),$ and
\begin{align*}
    \widetilde{mc}(\widetilde{b}_{(1)})&\equiv \widetilde{\delta}(t\cdot b_1) +\widetilde{\delta}((1-t)\cdot b_0)+\widetilde{\delta}(dt\cdot \eta )\\
    &\equiv dt \cdot b_1+ t\cdot \delta(b_1) -dt \cdot b_0 +(1-t)\cdot \delta(b_0)- dt\cdot \delta(\eta)\\
    &\equiv t\cdot\delta(b_1-b_0)+\delta(b_0)\\
    &\equiv c_0\cdot \id
    \pmod{F^{E_1}\fUs}.
\end{align*}
Define $ \widetilde{c}_{(1)}:=c_0\in \Rs,$ so $\fpar\widetilde{c}_{(1)}=0,$ $\nu_{\Rs}(\widetilde{c}_{(1)})>0$ and $\widetilde{c}_{(1)}|_i\equiv c_i \pmod{F^{E_1}\Rs}$ for $i=0,1.$ In addition $\widetilde{b}_{(1)}|_0=b_0,$ and $\widetilde{b}_{(1)}|_1=b_1.$
Suppose by induction that we have constructed $\widetilde{b}_{(l)}\in \fUs$ with $|\widetilde{b}_{(l)}|=1,$
$\nu_{\fUs}(\widetilde{b}_{(l)})\in \mG(b_0,b_1),$ 
such that
\[
\widetilde{mc}(\widetilde{b}_{(l)})\equiv \widetilde{c}_{(l)}\cdot \id,  \pmod{F^{E_l}\fUs} \quad \widetilde{c}_{(l)}\in (\fRs)^2,  \quad \fpar\widetilde{c}_{(l)}=0, \quad \nu_{\fRs}(\widetilde{c}_{(l)})>0
\]
and 
\[ \widetilde{c}_{(l)}|_i\equiv c_i \pmod{F^{E_l}\Rs}, \quad\widetilde{b}_{(l)}|_i=b_i,\quad  \quad i=0,1.
\]
Define the obstruction cochain $\widetilde{o}_l\in\fUs$ as in~\eqref{o_tilde_j}.
By Lemma~\ref{lm:primitive_form} and Lemma~\ref{lm:c}, there exist $k_l\in F^{\ge E_{l+1}}\Rs$ and $\widetilde{b}_l\in F^{\ge E_{l+1}}\fUs_{\mathrm{rel}}$  with $|k_l|=2,$ $|\widetilde{b}_l|=1,$ such that 
\[
\widetilde{\delta}(\widetilde{b}_l)\equiv -\widetilde{o}_l+k_l\cdot \id \pmod{F^{E_{l+1}}\fUs_{\mathrm{rel}}},
    \]
    and  \[
    \widetilde{c}_{(l+1)}:=\widetilde{c}_{(l)}+k_l\in (\Rs)^2
    \]
    satisfies 
    $\widetilde{c}_{(l+1)}|_i\equiv  c_i \pmod{F^{E_{l+1}}\Rs}$ for $i=0,1.$
Define \[\widetilde{b}_{(l+1)}:=\widetilde{b}_{(l)}+b_l\in (\fUs)^1.\]
Thus, $\nu_{\Rs}(\widetilde{c}_{(l+1)})>0$ and $\nu_{\fUs}(\widetilde{b}_{(l+1)})\in \mG(b_0,b_1).$
Since $k_l\in \Rs,$ it follows that $\fpar\widetilde{c}_{(l+1)}=\fpar\widetilde{c}_{(l)}=0.$ Since
$\widetilde{b}_l\in \fUs_{\mathrm{rel}},$ it follows that  
\[
\widetilde{b}_{(l+1)}|_i=\widetilde{b}_{(l)}|_i=b_i, \quad i=0,1.
\]
We get,
\begin{align*}
    \widetilde{mc}(\widetilde{b}_{(l+1)})&\equiv \widetilde{mc}(\widetilde{b}_{(l)})+ \widetilde{\delta}(\widetilde{b}_l)\\
    &\equiv \widetilde{c}_{(l)}\cdot \id+ \widetilde{o}_l+\widetilde{\delta}(\widetilde{b}_l)\\
    &\equiv \widetilde{c}_{(l)}\cdot \id+k_l\cdot \id \\
    &\equiv \widetilde{c}_{(l+1)}\cdot \id \pmod{F^{E_{l+1}}\fUs} .
\end{align*}
Take the limits $\widetilde{b}=\lim_l \widetilde{b}_{(l)},$ $\widetilde{c}=\lim_l \widetilde{c}_{(l)}.$
So, $\fpar\widetilde{c}=0,$ $\nu_{\fRs}(\widetilde{c})>0$ and $\nu_{\fUs}(\widetilde{b})>0.$ 
We obtain a bounding cochain $\widetilde{b}\in \Mmm(\fUs,\widetilde{c})$ that satisfies $\widetilde{b}|_0=b_0$ and $\widetilde{b}|_1=b_1.$
In addition, $\widetilde{c}$ satisfies $\widetilde{c}|_0=c_0,$ $ \widetilde{c}|_1=c_1.$ 
Therefore, $b_0\sim_{\fUs}b_1.$

\end{proof}

\begin{proof}[Proof of Theorem~\ref{thm:bc_equivalence}]
    By Lemma~\ref{building_equiv_2} the map $\varrho$ is well defined. It follows from Proposition~\ref{theorem:existence_bc} that $\varrho$ is surjective, and from Proposition~\ref{building_equiv_1} that $\varrho$ is injective.
\end{proof}

\section{The normed matrix factorization category for \texorpdfstring{$\cp^n$}{the projective space}}\label{Section:example}
The objective of this section is to 
introduce the example of the normed matrix factorization category associated to $\cp^n$ where $n$ is odd. 
We begin by constructing this category and the object $\Msf$ described in Section~\ref{sectin:lg_model_introduction}.
Then, we prove that $\Msf$ is a normed Calabi-Yau spherical object.
Finally, we consider the case $n=1.$ We prove Proposition~\ref{prop:bounding cochain} which gives an explicit bounding cochain of $\Msf.$ Then, we compute the numerical invariants of $\Msf$ and prove Theorem~\ref{thm:gw_inv_numerical_invariants}.  

 \subsection{Constructing \texorpdfstring{$\MF(\Wsf,w)$}{the category}}\label{section:4}

\subsubsection{The valued Landau-Ginzburg model of \texorpdfstring{$\cp^n$}{ the projective space}}\label{section:mf_algebra_c}

We follow the construction of a valued Landau-Ginzburg model given in Section~\ref{Section:toric_construction_sub}.

Let $n\in\Z_{>0}$ be odd.
Let $(\lambda_0,\ldots ,\lambda_n)\in \Q^{n+1}$ where $\lambda_i=0$ for $0\le i\le n-1,$ and $\lambda_n=1.$ 
Let $\Tilde{g}:\Z^{n+1}\to \Z$ be defined by $\Tilde{g}(m_0,\ldots, m_n)=2(m_0+\ldots+ m_n).$
Let $w_0,\ldots,w_n$ denote the standard basis of $\R^n,$ and let
$v_0,\ldots, v_n\in \Z^n$ be defined by $v_i:=-w_i$ for $0\le i\le n-1,$ and $v_n:= w_0+\ldots+ w_{n-1}.$ 
Define
\[
\trian:=\{x\in \R^n| \ \langle x, v_i\rangle\le \lambda_i\}= \{(x_0,\ldots, x_{n-1})\in\R^n| \ x_0+\ldots+x_{n-1}\le 1 \  , \ 0\le x_i \ \forall i  \}.
\]
We have $J=\Z.$ The inclusion $i:J\to \Z^{n+1}$ is given by $m\mapsto (m,\ldots, m).$ 
Hence,
\[
g_J(m)=2(n+1)m,  \quad  h_J(m)=m, \quad m\in J,
\] 
and $G=\frac{1}{n+1}\Z.$ For $n>1,$ the quadratic form $q$ satisfies
\[
q(m)=\tilde{q}(i(m))=\sum_{i\le j}m^2\equiv \frac{n(n+1)}{2}m \pmod{2}, \quad m\in J,
\]
and for $n=1,$ $q$ is the zero form.
Hence, $q: J\to \Z/2$ is a homomorphism. 
We get
\[
    R_{\triangle^n}=\bigg\{\sum_{i=0}^\infty a_iT^{\beta_i}| \ a_i\in\R, \ \beta_i\in \frac{1}{n+1}\Z, \  \lim_i\beta_i=\infty \bigg\}. 
    \]
The valuation on $R_{\triangle^n}$ is given by
\[
\nu_{R_{\triangle^n}}(\sum_{i=0}^\infty a_iT^{\beta_i})=\inf_{a_i\ne0}\{\beta_i\},
\]    
and the grading is given by declaring $|T^\beta|=2(n+1)\beta.$ Equip $R_{\triangle^n}$ with the trivial differential $d_{R_{\triangle^n}}.$
For $n>1,$ let $\sigma=q\in\Hom(J,\Z/2),$ and for $n=1$ let $\sigma$ be the non-trivial element, that is, $\sigma(1)=1.$ Thus,
\[
S_{\triangle^n}= R_{\triangle^n}\langle z_0,\ldots, z_n\rangle/ \langle (-1)^{\frac{n(n+1)}{2}}z_0\cdots z_n-T\rangle.
\]
A grading on $S_{\triangle^n}$ is defined by declaring $|z_i|=2.$ 
The family of valuations $\{\nu_\zeta^{S_{\triangle^n}}\}_{\zeta\in\interior{\triangle^n}}$ on $S_{\triangle^n}$ is defined by
\[
\nu_\zeta^{S_{\triangle^n}}\big(z_i)= \zeta_i,\quad 0\le i\le n-1, \qquad \nu_\zeta^{S_{\triangle^n}}\big(z_n)=1-\zeta_0-\ldots-\zeta_{n-1}.
\]
The superpotential is $W_{\triangle^n}=z_0+\ldots+z_d.$
Therefore, 
$\big((R_{\triangle^n},\nu_{R_{\triangle^n}}),(S_{\triangle^n},\{\nu^{S_{\triangle^n}}_\zeta\}_{\zeta\in \interior{\triangle}}) , W_{\triangle^n} \big)$ is the valued Landau-Ginzburg model associated to $\cp^n$ for $n\in\Z_{>0}$ odd.

\subsubsection{The Dirac matrix factorization of the projective space}\label{section:4.1.2}
Let $T(\widetilde{L}_{S_{\triangle^n}})$ be the tensor algebra as defined in Section~\ref{ssec:DiracMF}, and recall the grading is defined by declaring $|e_i|=1.$ Let $\{\nu^T_\zeta\}_{\zeta\in \interior{\triangle^n}}$ be a family of valuations on $T(\widetilde{L}_{S_{\triangle^n}})$ defined by
\[
\nu^T_\zeta(e_i)=\frac{1}{2}\zeta_i, \quad 0\le i\le n-1, \quad \nu^T_\zeta(e_i)=\frac{1}{2}-\frac{1}{2}\zeta_0-\ldots-\frac{1}{2}\zeta_{n-1}.
\]
Abbreviate $e_ie_j=e_i\otimes e_j.$
For $n>1$ let $\sigma'=0,$ and for $n=1$ let $\sigma'\in \Hom(J,\Z/2)$ be the homomorphism defined by $\sigma'(1,1)=1.$ Thus, it follows from the definition of $\varepsilon_{ij}$ in equation~\eqref{eq:vareps} that
\[
P=\langle e_ie_j+ e_je_i=2\delta_{ij}z_i| \  i,j=0,\ldots ,n \rangle.
\]
So,
\[
\Cl=T(\widetilde{L}_{S_{\triangle^n}})/\langle e_ie_j+ e_je_i=2\delta_{ij}z_i| \  i,j=0,\ldots ,n \rangle.
\]
Choose $\s : J \to \Zt$ to be the zero map, so 
\[
L_{\triangle^n}=\langle e_0\cdots e_n-T^{1/2}\rangle.
\] 
Write $M_{\triangle^n}=\Cl/L_{\triangle^n},$ and define $D_{\triangle^n}=e_0+\ldots+e_n.$
It follows from the definitions of $g_J$ and $\sigma$  that $w_{\triangle^n}=0.$ Let $\{\nu^{M_{\triangle^n}}_\zeta\}_{\zeta\in \interior{\triangle^n}}$ be the family of valuations on $\Msf$ induced by $\{\nu^T_\zeta\}_{\zeta\in \interior{\triangle^n}}$. 
Hence, $(M_{\triangle^n}, D_{\triangle^n}, \{\nu^{M_{\triangle^n}}_\zeta\}_{\zeta\in \interior{\triangle^n}})$ is an object of $\MF(\Wsf,0)$.
\subsubsection{Changing the coordinates}\label{section:changing_coordinates}
Throughout this section, we will work with an equivalent valued Landau-Ginzburg model of $\cp^n$ obtained by considering an isomorphic $R_{\triangle^n}$-algebra to $S_{\triangle^n}.$

Abbreviate $\Rsf=R_{\triangle^n}.$
Define 
\[
\Ssf:= \Rsf\langle z_0,\ldots, z_n\rangle/ \langle z_0\cdots z_n-T\rangle.
\]
The map defined by
\[
z_0\mapsto (-1)^{\frac{n(n+1)}{2}}z_0, \quad z_i\mapsto z_i, \quad 1\le i\le n, \quad r\mapsto r ,\quad r\in \Rsf,
\]
is an isomorphism from $S_{\triangle^n}$ to $\Ssf.$
Let $\Wsf$ denote the image of $W_{\triangle^n}$ under this isomorphism. Hence, 
\[
\Wsf=(-1)^{\frac{n(n+1)}{2}}z_0+z_1+\ldots+z_n.
\]
Write $\pi^*_\R=\pi^*\otimes\id_\R,$ and consider the map
\[
\lambda-\pi_\R^*: \R^n \to \R^{n+1},
\] 
\[
(x_0,\ldots ,x_{n-1})\mapsto (x_0,\ldots, x_{n-1}, 1-x_0-\ldots- x_{n-1}).
\]
Let
\[
\simp:=(\lambda-\pi^*_\R)(\triangle^n)= \{(x_0,\ldots, x_n)\in \R^{n+1}| \ x_0+\ldots+x_n=1, \  0\le x_i \ \forall i \}.
\]
We will parametrize the valuations $\nu_{\zeta}^{\Ssf}$ by $\interior{\widetilde{\triangle}}$ instead of $\interior{\triangle}.$ Thus, we have $\nu_\zeta^{\Ssf}\big(z_i)=\zeta_i$ for all $0\le i\le n$ where $\zeta_0+\ldots+\zeta_n=1$ and $\zeta_i\ge0.$

We consider the following valued Landau-Ginzburg model
\[
\big((\Rsf,\nu_{\Rsf}),(\Ssf,\{\nu^{\Ssf}_\zeta\}_{\zeta\in \interior{\simp}}) , \Wsf \big).
\]
Let $\MF(\Wsf,0)$ be the associated normed matrix factorization category.
Let $I$ be the ideal generated by the sets $\{e_0^2-(-1)^{\frac{n(n+1)}{2}}z_0\},$ $\{e_i^2-z_i| \ 1\le i\le n\},$ $\{e_ie_j+e_je_i| \ 0\le i<j\le n\}.$
After changing the coordinates, we have $\Cl=T(L_{\Ssf})/I.$ Write $\Msf=\Cl/\langle e_0\cdots e_n- T^{1/2}\rangle.$ The left $\Cl$-module $\Msf$ is equipped with the family of valuations $\{\nu^{\Msf}_\zeta\}_{\zeta\in \interior{\simp}}$ determined by 
$\nu_\zeta^{\Msf}\big(e_i)=\frac{1}{2}\zeta_i,$ for all $\zeta\in \interior{\simp}.$ Abbreviate $\Dsf=D_{\triangle^n}.$ Hence, $(\Msf, \Dsf, \{\nu^{\Msf}_\zeta\}_{\zeta\in \interior{\simp}})$ is an object of $\MF(\Wsf,0)$

\subsection{Spherical object}\label{section:spherical_object}
Our goal is to prove Theorem~\ref{lm:cohomology}.
Let
\[
F^{\ge0}\End(\Msf)=\{\Phi\in\End(\Msf)| \ \nu_{\Msf}(\Phi)\ge0\},
\]
\[
F^{\ge0}\Ssf=\{s\in\Ssf| \ \inf_{\zeta\in\interior{\simp}}\nu_\zeta^{\Ssf}(s)\ge0 \}.
\]
Abbreviate $\End(\Msf)_{\ge0}=F^{\ge0}\End(\Msf),$ and $({\Ssf})_{\ge0}=F^{\ge0}\Ssf.$
The following proposition will play an essential role in the proof of Theorem~\ref{lm:cohomology}.
\begin{prop}\label{prop:coho}
 \[H^i(\End(\Msf)_{\ge0})=
         \left\{\begin{array}{ll}
       T^{\frac{i}{2(n+1)}}\cdot \R, &  i=0,2,4,\ldots, n-1\\
       0, & \mathrm{otherwise}.\\
        \end{array} \right.\]
    
\end{prop}
By Remark~\ref{rem:cl_endomorphism} we can consider $\Cl$ as a subalgebra of $\End(\Msf).$ 
In dealing with the next results we use the following notations.
Let $n\ge1$ be an odd number. 
Define
\[
Y(i,n):=\{(j_1,\ldots, j_i)\in\{0,\ldots,n\}^i| \  j_1<\ldots<j_i  \},
\]
where $1\le i\le n+1,$ and $Y(0,n):=\{\emptyset\}.$
Let $I=(j_1,\ldots,j_i)\in Y(i,n).$ We denote by $I\backslash\{j_k\}$ the $i-1$ tuple $(j_1,\ldots,j_{k-1},j_{k+1},\ldots, j_i)$ where $1\le k\le i.$ If $I=(j_1),$ then by $I\backslash\{j_1\}$ we mean the empty set.
We write $e_I$ for the element $e_{j_1}\cdots e_{j_i}$ in $\Cl,$ where by $e_\emptyset$ we mean to the element $1.$
Similarly, if $V$ is a vector space with a fixed basis $v_0,\ldots,v_n,$ we write $v_I$ for $v_{j_1}\wedge\ldots\wedge v_{j_i}.$ 
Given $k\in\{0,\ldots, n\},$ we write 
\[
\mu_{I,k}=|\{j\in I | \ j<k\}| +\frac{n(n+1)}{2}\delta_{k0}.
\]
Define
\[
Y_0(i,n):=\{(j_1,\ldots,j_i)\in\{1\ldots, n\}^i| \ j_1<\ldots<j_i \}\subset Y(i,n),
\]
for $1\le i\le n,$ and $Y_0(0,n)=\{\emptyset\}.$

Fix $J\in Y_0(k,n)$ such that $0\le k\le n-1$ is even. We begin by defining three complexes that will help us compute the cohomology of $H^*(\End(M)_{\ge0}).$

Define the complex ${E_J}_*$ as follows. For every $0\le i\le n+1,$   let ${E_J}_i$ be the free $({\Ssf})_{\ge0}$-module generated by 
$\{e_{I,J}\}_{I\in Y(i,n)}.$
For $i>n+1$ and $i<0,$ we set ${E_J}_i=0.$
The differential $d_{E_J}$ of ${E_J}_*$ is given by
\[
d_{E_J}(e_{I,J})
=2\sum_{j\in I\cap J}(-1)^{\mu_{I,j}} e_{I\backslash\{j\},J}+ 2\sum_{j\in I\backslash J}(-1)^{\mu_{I,j}} z_je_{I\backslash\{j\},J}.
\]

We define an additional complex ${F_J}_*$ as follows.
For every $0\le i\le n+1,$ let ${F_J}_i$ be the free $({\Ssf})_{\ge0}$-module generated by  
$\{f_{I,J}\}_{I\in Y(i,n)}.$
For $i>n+1$ and $i<0$ set ${F_J}_i=0.$ 
The differential $d_{F_J}$ of ${F_J}_*$ is given by
\[
d_{F_J}(f_{I,J})
=2\sum_{j\in I\cap J}(-1)^{\mu_{I,j}} z_jf_{I\backslash\{j\},J},+ 2\sum_{j\in I\backslash J}(-1)^{\mu_{I,j}} f_{I\backslash\{j\},J}.
\]

Define
\[
Q:=\Rsf\langle x_0,\ldots,x_n ,y_0,\ldots, y_n\rangle.
\]
Let $V$ be an $n+1$-dimensional vector space over $\Rsf,$ and fix a basis $v_0,\ldots,v_n$ of $V.$ Let $K_*$ be the Koszul complex defined by $K_i=Q\otimes \bigwedge^i V,$ for $0\le i\le n+1.$
Set $K_{-1}= Q/\langle y_0,\ldots, y_n\rangle$ and $K_i=0$ for all $i<-1$ and $i>n+1.$
The differential $d_V$ is given by
\[
d_V(v_I)=2\sum_{j\in I}(-1)^{\mu_{I,j}}y_jv_{I\backslash\{j\}}.
\]
Since $K_*$ is a Koszul complex and $y_0,\ldots, y_n$ is a regular sequence, it follows that $K_*$ is exact.

\begin{lm}\label{lm:co_e}
    Let $J\in Y_0(k,n)$ where $0\le k\le n-1$ is even. 
    \begin{enumerate}[label=(\arabic*)]
        \item If $0<k\le n-1,$ then  $H^*({E_J})=0,$
        \item if $k=0,$ then  
    \[H^i(E_\emptyset)=
         \left\{\begin{array}{ll}
        \bigoplus_{i=0}^n T^{\frac{
        i}{n+1}}\cdot \R, &  i=0,\\
        0, & \mathrm{otherwise}.\\
        \end{array} \right.\]
    \end{enumerate}
\end{lm}
\begin{proof}
Define
\[
\phi_J: Q\to \Ssf
\]
by
\[
\phi_J(x_i)=
         \left\{\begin{array}{ll}
        1, &  i\not\in J \\
        z_i, & i\in J
        \end{array}\right. ,\quad 
 \phi_J(y_i)= \left\{\begin{array}{ll}
        z_i, &  i\not\in J \\
        1, & i\in J
        \end{array}\right. 
        , \quad \phi_J(r)=r \quad\forall r\in \Rsf.
\]
Consider the cochain complex $K_*\otimes_Q ({\Ssf})_{\ge0}.$
Since $K_i$ are flat $Q$-modules, it follows that  $K_*\otimes_Q ({\Ssf})_{\ge0}$ is an exact complex. 
Define
\[
\Psi_J: {E_J}_*\to K_*\otimes_Q ({\Ssf})_{\ge0},
\]
\[
e_{I,J}\mapsto v_I.
\]
It follows by the definition of $\phi_J$ that $\Psi_J$ is a cochain map.

Assume $0<k\le n-1.$ Since $J\ne\emptyset,$ it follows by the definition of $\phi_J$ that $K_{-1}\otimes_Q ({\Ssf})_{\ge0}=0,$ which implies that $\Psi_J$ is an isomorphism. Thus,
${E_J}_*$ is an exact complex, and so $H^*(E_J)=0$.

Assume $k=0,$ so $J=\emptyset.$
Hence, $K_{-1}\otimes_Q ({\Ssf})_{\ge0}=({\Ssf})_{\ge0}/\langle z_0,\ldots, z_n\rangle.$
Define the complex ${\overline{E}_\emptyset}_*$ as follows. Set ${\overline{E}_\emptyset}_i={E_\emptyset}_i$ for $i\ne-1,$ and ${\overline{E}_\emptyset}_{-1}=({\Ssf})_{\ge0}/\langle z_0,\ldots, z_n\rangle.$ Thus, the map  
\[
\overline{\Psi}_\emptyset: {\overline{E}_\emptyset}_*\to K_*\otimes_Q ({\Ssf})_{\ge0},
\]
given by $\overline{\Psi}_\emptyset(e_{I,\emptyset})=v_I$ for $I\in Y(i,n)$ where $0\le i\le n+1,$ and $\overline{\Psi}_\emptyset(x)=x$ for $x\in {\overline{E}_\emptyset}_{-1},$ is an isomorphism. Hence, ${\overline{E}_\emptyset}_*$ is an exact complex. Therefore, $H^i(E_\emptyset)\ne0$ only if $i=0,$ and $H^0(E_\emptyset)=({\Ssf})_{\ge0}/\langle z_0,\ldots, z_n\rangle.$
The claim follows since 
\[
({\Ssf})_{\ge0}/\langle z_0,\ldots, z_n\rangle=\R[[T^{1/n+1}]]/\langle T\rangle =\bigoplus_{i=0}^n T^{\frac{
        i}{n+1}}\cdot \R.
\]
\end{proof}
\begin{lm}\label{lm:co_f}
     Let $J\in Y_0(k,n)$ where $0\le k\le n-1$ is even. Then, $H^*(F_J)=0.$
\end{lm}
\begin{proof}
Define
\[
\phi_J: Q\to \Ssf
\]
by
\[
\phi_J(x_i)=
         \left\{\begin{array}{ll}
        z_i, &  i\not\in J \\
        1, & i\in J
        \end{array}\right. ,\quad 
 \phi_J(y_i)= \left\{\begin{array}{ll}
        1, &  i\not\in J \\
        z_i, & i\in J
        \end{array}\right.
        , \quad \phi_J(r)=r \quad\forall r\in \Rsf.
\]

Consider the cochain complex $K_*\otimes_Q ({\Ssf})_{\ge0}.$
Since $K_i$ are flat $Q$-modules, it follows that  $K_*\otimes_Q ({\Ssf})_{\ge0}$ is an exact complex. Since $J\ne (0,\ldots, n),$ it follows that $K_{-1}\otimes_Q ({\Ssf})_{\ge0}=0.$ Hence, the map $\phi_J$ induces the following isomorphism 
\[
\Psi_J: {F_J}_*\to K_*\otimes_Q ({\Ssf})_{\ge0},
\]
\[
f_{I,J}\mapsto v_I.
\]
Thus,
${F_J}_*$ is an exact complex, and so $H^*(F_J)=0$.

\end{proof}

    We change the grading of the cochain complexes ${E_J}_*$ and ${F_J}_*$ such that the new grading includes the grading of $({\Ssf})_{\ge0}.$
    \begin{rem}\label{rem:cohomology_e_j}
    By Lemma~\ref{lm:co_e} the cohomology of ${E_\emptyset}_*$ is given by
\[H^i(E_\emptyset)=
         \left\{\begin{array}{ll}
        T^{\frac{
        i}{2(n+1)}}\cdot\R, &  i=0,2,4,\ldots, 2n,\\
        0, & \mathrm{otherwise}.\\
        \end{array} \right.\]
\end{rem}
Let $J\in Y_0(k,n)$ where $0\le k\le n-1$ is even.
Let $\Cl\otimes e_J$ be the complex that generated by the elements $e_I\otimes e_J,$ where $I\in Y(i,n)$ such that $0\le i\le n+1,$ and equipped with the differential
\[
d_J(e_I\otimes e_J)= \delta(e_I)\otimes e_J.
\]
Note that where $k=0,$ the complex is simply $\Cl.$  We equip $\Cl\otimes e_J$ with the norm $\nu_{\Msf}\otimes \nu_{\Msf}.$ 
Abbreviate $(\Cl\otimes e_J)_{\ge0}=F^{\ge0}\big(\Cl\otimes e_J\big).$
\begin{lm}\label{lm:coho_cl_e_j}
    Let $J\in Y_0(k,n)$ where $0\le k\le n-1$ is even. 
    \begin{enumerate}[label=(\arabic*)]
        \item If $0<k\le n-1,$ then  $H^*((\Cl\otimes e_J)_{\ge0})=0,$
        \item if $k=0,$ then  
    \[H^i(\Cl_{\ge0})=
         \left\{\begin{array}{ll}
        T^{\frac{i}{2(n+1)}}\cdot\R, &  i=0,2,4,\ldots, n-1,\\
        0, & \mathrm{otherwise}.\\
        \end{array} \right.\]
    \end{enumerate}
\end{lm}
\begin{proof}
Consider the short exact sequence
\[
0\longrightarrow (\Cl\otimes e_J)_{\ge0}[-n-1]\overset{\varphi}{\longrightarrow} {E_J}_*\oplus {F_J}_*\overset{\psi}{\longrightarrow} (\Cl\otimes e_J)_{\ge0}\longrightarrow 0,
\]
where
\[
\varphi(a_I e_I\otimes e_J)= (-a_I\prod_{i\in I\cap J}z_i T^{1/2}e_{I,J}, \ a_I\prod_{i\in I\cup J}z_i f_{I,J} ),
\]
such that $a_I\in \Ssf$ such that $a_I e_I\otimes e_J\in (\Cl\otimes e_J)_{\ge0},$ and
\[
\psi(e_{I,J},0)=\frac{e_I\otimes e_J}{\prod_{j\in I\cap J}z_j}, \quad \psi(0,f_{I,J})=\frac{T^{1/2}e_I\otimes e_J}{\prod_{j\in I\cup J}z_j}.
\]
Consider the long exact sequence that induced by this short exact sequence.
If $k>0,$ then by Lemmas~\ref{lm:co_e},~\ref{lm:co_f} we get $H^*(E_J)=H^*(F_J)=0.$ Hence, $H^*((\Cl\otimes e_J)_{\ge0})=0.$

Let $k=0.$
Note that since $H^*(F_\emptyset)=0,$ $\varphi$ induces the following maps 
\[
\varphi_i: H^{i-n-1}(\Cl_{\ge0})\rightarrow H^i(E_J)\oplus H^i(F_J), \quad i\in \Z,
\]
\[
[x]\mapsto ([T^{1/2}\cdot x],0). 
\]
The cohomology $H^i(\Cl_{\ge0})$ is computed as follows:
\begin{enumerate}[label=(\arabic*)]
    \item\label{case1_cohomology_clifford} Since $\Cl_{\ge0}$ consists elements of non-negative degree, we get $H^i(\Cl_{\ge0})=0$ for $i<0,$
    \item\label{case3_cohomology_clifford}  let $0\le i< n,$ and
   consider the exact sequence
    \[
    \ldots\rightarrow  H^{i-n-1}(\Cl_{\ge0})\rightarrow H^i(E_\emptyset)\oplus H^i(F_\emptyset) \rightarrow  H^i(\Cl_{\ge0})\rightarrow H^{i-n} (\Cl_{\ge0})\ldots.
    \]
    Since $i-n<0,$ it follows by~\ref{case1_cohomology_clifford} that $H^{i-n-1}(\Cl_{\ge0}), H^{i-n}(\Cl_{\ge0})=0.$ Hence, $H^i(E_\emptyset)\oplus H^i(F_\emptyset) \cong H^i(\Cl_{\ge0}).$
     Assume $i$ is odd.
    By Lemma~\ref{lm:co_f} we have $H^i(F_\emptyset)=0,$ and
    since $i$ is odd we get by Remark~\ref{rem:cohomology_e_j} that $H^i(E_\emptyset)=0.$
    Hence, $H^i(\Cl_{\ge0})=0.$
    Assume now $i$ is even.
    By Remark~\ref{rem:cohomology_e_j} and Lemma~\ref{lm:co_f}, we have $H^i(E_\emptyset)\cong T^{\frac{i}{2(n+1)}}\cdot \R, H^i(F_\emptyset)=0$. Hence,
     $H^i(\Cl_{\ge0})\cong T^{\frac{i}{2(n+1)}}\cdot \R,$
    \item let $i=n,n+1,$ and consider the exact sequence
      \[
    \ldots\rightarrow H^n(\Cl_{\ge0})\rightarrow H^0(\Cl_{\ge0})\rightarrow H^{n+1}(E_\emptyset)\oplus H^{n+1}(F_\emptyset)\rightarrow H^{n+1}(\Cl_{\ge0})\rightarrow \ldots.
    \]
    By Remark~\ref{rem:cohomology_e_j} we have $H^{n+1}(E_\emptyset)\cong T^{1/2}\cdot \R$ and by~\ref{case3_cohomology_clifford} we have $H^0(\Cl_{\ge0})\cong\R.$ Thus, $\varphi_{n+1}$ is an isomorphism, so $H^n(\Cl_{\ge0})=H^{n+1}(\Cl_{\ge0})=0.$
    \item\label{case4:cohomology_clifford} let $n+1<i\le 2n.$ Assume $i$ is even and consider 
     the exact sequence
    \[
    \ldots\rightarrow H^{i-n-1}(\Cl_{\ge0})\rightarrow H^i(E_\emptyset)\oplus H^i(F_\emptyset)\rightarrow H^i(\Cl_{\ge0})\rightarrow \ldots.
    \]
    By~\ref{case3_cohomology_clifford} we have $H^{i-n-1}(\Cl_{\ge0})\cong T^{\frac{i-n-1}{2(n+1)}}\cdot\R,$ and by Remark~\ref{rem:cohomology_e_j} we have $H^i(E_\emptyset)\cong T^{\frac{i}{2(n+1)}}\cdot\R.$ Hence, $\varphi_i$ is an isomorphism, so $H^i(\Cl_{\ge0})=0.$
    Assume $i$ is odd and consider 
    the exact sequence
    \[
    \ldots\rightarrow H^i(\Cl_{\ge0})\rightarrow H^{i-n}(\Cl_{\ge0})\rightarrow H^{i+1}(E_\emptyset)\oplus H^{i+1}(F_\emptyset)\rightarrow  \ldots.
    \]
     By~\ref{case3_cohomology_clifford} we have $H^{i-n}(\Cl_{\ge0})\cong T^{\frac{i-n}{2(n+1)}}\cdot\R,$ and by Remark~\ref{rem:cohomology_e_j} we have $H^{i+1}(E_\emptyset)\cong T^{\frac{i+1}{2(n+1)}}\cdot\R.$ Hence, $\varphi_{i+1}$ is an isomorphism, so $H^i(\Cl_{\ge0})=0.$
    \item if $i>2n,$ then since by Remark~\ref{rem:cohomology_e_j} and  Lemma~\ref{lm:co_f} $H^i(E_\emptyset)=H^i(F_\emptyset)=0,$ it follows that $H^i(\Cl_{\ge0})=0.$
\end{enumerate}

\end{proof}
Recall the definition of $\ccle$ in Section~\ref{section:azumaya_algebra}. Let $\ccle^{\mathrm{op}}$ denote the opposite algebra of $\ccle.$
Define 
\[F:\Cl\otimes\ccle^{\mathrm{op}}\rightarrow \End(\Msf),
\]
\[
x\otimes y\mapsto \phi_{x,y},
\]
where $\phi_{x,y}(m)=xmy,$ $m\in\Msf.$
Note that $F$ preserves norms since
\begin{multline*}
\nu_{\Msf}(\phi_{x,y})=\inf_{\zeta\in\interior{\simp}}\inf_{m\in \Msf\backslash\{0\}}\{\nu_\zeta^{\Msf}(xmy)-\nu_\zeta^{\Msf}(m)\}\\
=\inf_{\zeta\in\interior{\simp}}\inf_{m\in \Msf\backslash\{0\}}\{\nu_\zeta^{\Msf}(x)+\nu_\zeta^{\Msf}(y)\}=\nu_{\Msf}\otimes\nu_{\Msf}(x\otimes y).
\end{multline*}
\begin{lm}\label{lm:F_isomorphism}
    $F$ is an isomorphism.
\end{lm}
\begin{proof}
    Let $\mathfrak{m}\subset \Ssf$ be a maximal ideal and let $k=\Ssf/\mathfrak{m}.$ 
We have
\[
\dim_k\Cl\otimes k=2^{n+1}, \quad \dim_k \ccle\otimes k=2^{n-1}, \quad \dim_k\End(\Msf)\otimes k=2^{2n}.
\]
Write $F_k=F\otimes \id_k.$
First, we want to show that $F_k$ is an isomorphism. Note that
\begin{equation}\label{eq:clifford_times_k}
\Cl\otimes\ccle^{\mathrm{op}}\otimes k\cong (\Cl\otimes k)\otimes(\ccle^{\mathrm{op}}\otimes k) .
\end{equation}
Let 
\[G_k:(\Cl\otimes k)\otimes(\ccle^{\mathrm{op}}\otimes k)\to\End(\Msf)\otimes k
\]
denote the composition of $F_k$ with a map induces the isomorphism in equation~\eqref{eq:clifford_times_k}. Hence, in order to show that $F_k$ is an isomorphism, it suffices to show that $G_k$ is an isomorphism.
By Lemma~\ref{lm:clifford_is_csga}  $\Cl\otimes k$ is a CSGA, and by Lemma~\ref{lm:ccle_csga} $\ccle\otimes k$ is a CSGA. 
Thus, it follows by Lemma~\ref{lm:tensor_producy_csga} that  $(\Cl\otimes k)\otimes(\ccle\otimes k)$ is a CSGA, which implies that $G_k$ is injective.  In addition,
since \[
\dim_k (\Cl\otimes k)\otimes(\ccle^{\mathrm{op}}\otimes k)=\dim_k\End(\Msf)\otimes k,
\]
it follows that $G_k$ is  surjective and so an isomorphism.
Hence,
for every field $k=\Ssf/\mathfrak{m}$ such that $\mathfrak{m}\subset \Ssf$ is a maximal ideal we get that $F_k$ is an isomorphism.
By Nakayam's lemma this implies that $F$ is surjective. Since $\End(\Msf)$ is projective, we obtain that $\Cl\otimes\ccle\cong\End(\Msf)\oplus \ker F. $  Hence, again by Nakayama's lemma, we get $\ker F=0.$    
Hence, $F$ is an isomorphism.
\end{proof}
\begin{proof}[Proof of Proposition~\ref{prop:coho}]
By Lemma~\ref{lm:F_isomorphism} we have
\[
\big(\Cl\otimes\ccle^{\mathrm{op}}\big)_{\ge0}\cong \End(\Msf)_{\ge0}.
\]
Since, 
\[
(\Cl\otimes \ccle^{\mathrm{op}})_{\ge0}=\bigoplus_{\substack{J\in Y_0(k,n)\\ 0\le k\le n-1 \mathrm{\ is \ even}}}(\Cl\otimes e_J)_{\ge0},\]
it follows by Lemma~\ref{lm:coho_cl_e_j} that 

\[
H^i((\End(\Msf))_{\ge0})\cong
         \left\{\begin{array}{ll}
        T^{\frac{i}{2(n+1)}}\cdot\R, &  i=0,2,4,\ldots, n-1,\\
       0, & \mathrm{otherwise}.\\
        \end{array} \right.\]
\end{proof}

\begin{proof}[Proof of Theorem~\ref{lm:cohomology}]
    Consider the short exact sequence
    \[
    0\longrightarrow \End(\Msf)_{\ge0}[-2]\overset{\phi}{\longrightarrow} \End(\Msf)_{\ge0}\overset{\pi}{\longrightarrow }\End(\Msf)_0\longrightarrow0,
\]
where $\phi(\Phi)=T^{\frac{1}{n+1}}\Phi$ and $\pi$ is the projection. Note that the maps $\phi_i:H^{i-2}\End(\Msf)_{\ge0}\rightarrow H^i(\End(\Msf)_{\ge0})$ are given by multiplication by $T^{\frac{1}{n+1}}.$
Consider the induced long exact sequence 
\[
 \ldots\rightarrow H^{i-2}(\End(\Msf)_{\ge0})\rightarrow H^{i}(\End(\Msf)_{\ge0})\rightarrow H^{i}(\End(\Msf)_0)\rightarrow H^{i-1}(\End(\Msf)_{\ge0})\rightarrow\ldots.
\]
\begin{enumerate}[label=(\arabic*)]
    \item Let $i<0.$ By Proposition~\ref{prop:coho} we have $H^i(\End(\Msf)_{\ge0})=0,$ so $H^i(\End(\Msf)_0)=0,$ 
    \item let $i=0,$ by Proposition~\ref{prop:coho} we have $H^j(\End(\Msf)_{\ge0})=0$ for $j=-1,-2.$ Thus, $H^0(\End(\Msf)_{\ge0})\cong H^0(\End(\Msf)_0).$ Hence, by Proposition~\ref{prop:coho} we get $H^i(\End(\Msf)_0)\cong\R,$ 
    \item let $0<i<n.$ Assume $i$ is even. Then, by Proposition~\ref{prop:coho} we have $H^{i-2}(\End(\Msf)_{\ge0})\cong T^{\frac{i-2}{2(n+1)}}\cdot\R$ and $H^i(\End(\Msf)_{\ge0})\cong T^{\frac{i}{2(n+1)}}\cdot\R.$ Hence, $\phi_i$ is an isomorphism. Thus, $H^i(\End(\Msf)_0)=0.$ Assume $i$ is odd. By Proposition~\ref{prop:coho} 
    we have  $H^{i-1}(\End(\Msf)_{\ge0})\cong T^{\frac{i-1}{2(n+1)}}\cdot\R$ and $H^{i+1}(\End(\Msf)_{\ge0})\cong T^{\frac{i+1}{2(n+1)}}\cdot\R.$ Thus,
    the map $\phi_{i+1}$ is an isomorphism. Hence, $H^i(\End(\Msf)_0)=0,$
    \item let $i=n.$ We have the following exact sequence
    \[
 \ldots\rightarrow H^n(\End(\Msf)_{\ge0})\rightarrow H^n(\End(\Msf)_0)\rightarrow H^{n-1}(\End(\Msf)_{\ge0})\rightarrow H^{n+1}(\End(\Msf)_{\ge0})\rightarrow\ldots.
\]
By Proposition~\ref{prop:coho} we have $H^j(\End(\Msf)_{\ge0})=0$ for $j=n,n+1.$ Hence, we get $H^n(\End(\Msf)_0)\cong H^{n-1}(\End(\Msf)_{\ge0}).$ Thus, by Proposition~\ref{prop:coho}  $H^n(\End(\Msf)_0)\cong T^{\frac{n}{2(n+1)}}\cdot \R,$
\item let $i>n.$ By Proposition~\ref{prop:coho} we
have $H^i(\End(\Msf)_{\ge0})=0.$ Thus, $H^i(\End(\Msf)_0)=0.$
\end{enumerate}

Since $({\Rsf})_0=\R,$ it follows that $\Msf$ is a spherical object.
\end{proof}

\subsection{The normed \texorpdfstring{$\infty$}{infinity}-trace  \texorpdfstring{$\Theta$}{Theta}}\label{section:tight_mf_cpn}
Recall the definition of $\Theta$ given in Theorem~\ref{Theta}. 
Our goal is to prove Theorem~\ref{thm:norm} which states that $\Theta$ is a normed $\infty$-trace on $\MF(\Wsf).$ Consider a general normed matrix factorization category $\MF(W_{\triangle,\sigma},w),$ where $w$ is admissible, as defined in Section~\ref{Section:toric_construction_sub}. We begin by proving that $\Theta$ is of dimension $\dim \triangle,$ and when $\dim \triangle$ is odd, that $\Theta$ is cohomologically unital.
Then, we prove that $\Theta$ on $\MF(\Wsf,w)$ is normed.

Let $\triangle$ be a Delzant polytope of dimension $n=\dim\triangle$.
Let $LG(\triangle,\sigma)$ be a valued Landau-Ginzburg model as defined in Section~\ref{Section:toric_construction_sub}, and let $\MF(W_{\triangle,\sigma},w)$ be the associated normed matrix factorization category, where $w$ is admissible. 
We may assume that $z_1,\ldots, z_n$ are the coordinates of Lemma~\ref{lm:u_coordinates}.
In the following we use the notations of Theorem~\ref{Theta}.
Let  
$l\ge1$ and fix $1\le i\le n.$ Let $r_1,\ldots, r_n\in \Z_{\ge0}$ such that $r_i\ge1,$ $r_1+\ldots+r_n=l,$ and let $k_1,\ldots,k_l\in \Z_{\ge0}$ such that $k_1+\ldots+k_l=n-1.$  Write
\[
\Vec{j}=(j_1^{(l)},\ldots,j_{k_1}^{(l)},\ldots, j_1^{(1)},\ldots,j_{k_l}^{(l)})\in S_n^i ,\quad \Vec{i}=(i^{(1)},\ldots,i^{(l)})\in \Lambda_n^l(r_1,\ldots,r_n), \quad 
\]
\[
\Vec{z}=(z_1,\ldots,z_n), \quad d\Vec{z}=dz_1\wedge\ldots\wedge dz_n.
\]
Let $(M^i,D^i, \{\nu_\zeta^{M^i}\}_{\zeta\in\interior{\triangle}})$ for $i=1,\ldots ,l$ be objects of $\MF(W_{\triangle,\sigma},w).$
We write $\Vec{\Phi}=(\Phi_1,\ldots,\Phi_l)$ for 
$\Phi_l\in \Hom(M^l,M^1)$ and $\Phi_j\in \Hom(M^j,M^{j+1}),$ $1\le j\le l-1,$
and let $\Omega=\frac{dz_1\wedge\ldots \wedge dz_n}{z_1\cdots z_n}.$  Define
    \[
    f_{l,\Vec{j},\Vec{i},\Vec{\Phi}}(\Vec{z}):=\mathrm{str}\big( \Phi_l\partial_{i^{(l)}}D^l\partial_{j_1^{(l)}}D^l\cdots \partial_{j_{k_1}^{(l)}}D^l\cdots \Phi_1\partial_{i^{(1)}}D^1\partial_{j_1^{(1)}}D^1\cdots \partial_{j_{k_l}^{(1)}}D^1\big),
    \]
    \[
    F_{l,\Vec{j},\Vec{i},\Vec{\Phi}}(\Vec{z}):=\frac{f_{l,\Vec{j},\Vec{i},\Vec{\Phi}}(\Vec{z})\wedge\Omega}{(\partial_1\Wtr)^{r_1+1}\cdots (\partial_i\Wtr)^{r_i}\cdots(\partial_n\Wtr)^{r_n+1}}.
    \]
In this section we prove properties concerning $\Theta$ as an operator on $\MF(W_{\triangle,\sigma},w).$ 
Since $\Theta_l(\Phi_l\otimes\ldots\otimes \Phi_1)$ is a linear combination of residues of differential forms of the form $ F_{l,\Vec{j},\Vec{i},\Vec{\Phi}}(\Vec{z}),$ it will be suffices in the proofs to talk in the level of  $F_{l,\Vec{j},\Vec{i},\Vec{\Phi}}(\Vec{z})$ instead of $\Theta.$ Hence, in the following we fix $i,r_1,\ldots, r_n,k_1,\ldots k_l,$ and $\Vec{j},$ $\Vec{i},$ $\Vec{\Phi}$ satisfying the conditions above. 

\subsubsection{\texorpdfstring{$\Theta$}{Theta} is cohomologically unital}
Recall we denote by $C_{\Wtr}$  the set of critical points of the superpotential $\Wtr.$ 
Note that since $|D^j|=1$ and $|z_i|=2,$ it follows that $|\partial_iD^j|=-1.$ 
In addition, if $\Phi\in\End(M)$ such that $|\Phi|$ is odd, then $\mathrm{str}(\Phi)=0.$
\begin{lm}\label{lm:vanish_theta}
 If $n$ is odd, then  $\Theta$ is cohomologically unital on $\MF(W_{\triangle,\sigma},w).$
\end{lm}
\begin{proof}
Let $(M,D, \{\nu_\zeta^M\}_{\zeta\in\interior{\triangle}})$ be an object of $\MF(\Wtr,w).$ 
By Remark~\ref{rem:cyclic_cx_r}, we need to show that $\Theta_{l,x}(\id_M^{\otimes l})=0$ for $l$ odd and $x\in C_{\Wtr}.$ 
    Since $n,l$ are odd numbers, it follows that
        \begin{align*}
|\id_M\partial_{i^{(l)}}D^l\partial_{j_1^{(l)}}D^l\cdots \partial_{j_{k_l}^{(l)}}D^l\cdots \id_M\partial_{i^{(1)}}D^1\partial_{j_1^{(1)}}D^1\cdots \partial_{j_{k_1}^{(1)}}D^1|&=-(k_1+\ldots +k_l)+l\\
&=-n+1+l\in 2\Z+1.
        \end{align*}
Hence, 
    \[
    f_{l,\Vec{j},\Vec{i},\Vec{\id}_M}(\Vec{z})=0 \quad\Longrightarrow \quad\Theta_{l,x}(\id_M^{\otimes l})=0.
    \]
    \end{proof}

\subsubsection{Dimension of  \texorpdfstring{$\Theta$}{Theta}}
For $x\in C_{\Wtr},$ we write $x=(x_1,\ldots,x_n).$ 
Note that for $\Phi\in\End(M)$ such that $|\Phi|$ is even, we have
$|\mathrm{str}(\Phi)|=|\Phi|.$  
\begin{lm}\label{lm:degree_theta}
    $\Theta$ is an operator of dimension $n$.
    \end{lm}
\begin{proof}
     Let $l\ge1.$ We have
     \[
|\Phi_l\otimes\ldots\otimes\Phi_1|=|\Phi_l|+\sum_{j=1}^{l-1}(|\Phi_j|-1)=\sum_{j=1}^l|\Phi_j|- (l-1), 
    \] 
where $\Phi_l\in \Hom(M^l,M^1)$ and $\Phi_j\in \Hom(M^j,M^{j+1}),$ $1\le j\le l-1.$
It suffices to show that if $\mathrm{Res}_x(F_{l,\Vec{j},\Vec{i},\Vec{\Phi}})\ne0,$ then $|\mathrm{Res}_x(F_{l,\Vec{j},\Vec{i},\Vec{\Phi}})|=|\Phi_l\otimes\ldots\otimes\Phi_1|-n$ where  $x\in C_{\Wtr}.$ 
Assume $f_{l,\Vec{j},\Vec{i},\Vec{\Phi}}\ne 0,$ then we have $|f_{l,\Vec{j},\Vec{i},\Vec{\Phi}}|=\sum_{j=1}^l|\Phi_j| -(n+l-1).$
    Since $|\Wtr|=2$ and $|z_i|=2,$ it follows that
    $|\partial_i\Wtr|=0.$ Since $|\Omega|=-2n,$ it follows that
    \[
    |F_{l,\Vec{j},\Vec{i},\Vec{\Phi}}|=|f_{l,\Vec{j},\Vec{i},\Vec{\Phi}}|-2n= \sum_{j=1}^l|\Phi_j|-(l-1)-3n.
    \]
Consider the Laurent series of $F_{l,\Vec{j},\Vec{i},\Vec{\Phi}}.$ We obtain
\[
|F_{l,\Vec{j},\Vec{i},\Vec{\Phi}}|=|\mathrm{Res}_x(F_{l,\Vec{j},\Vec{i},\Vec{\Phi}})|+|\prod_{j=1}^n(z_j-x_j)^{-1}|= |\mathrm{Res}_x(F_{l,\Vec{j},\Vec{i},\Vec{\Phi}})|- 2n.
\]
Therefore, if the residue of $F_{l,\Vec{j},\Vec{i},\Vec{\Phi}}$ does not vanish, we get
\[
|\mathrm{Res}_x(F_{l,\Vec{j},\Vec{i},\Vec{\Phi}})|=|F_{l,i,\Vec{j},\Vec{i},\Vec{\Phi}}|+2n=\sum_{j=1}^l|\Phi_j|-(l-1)-n=|\Phi_l\otimes\ldots\otimes\Phi_1|-n.
\]
\end{proof}

\subsubsection{\texorpdfstring{$\Theta$}{Theta} is normed}
 
In this section, we consider the normed matrix factorization category $\MF(\Wsf,w)$ associated to $\cp^n$ where $n$ is odd. We take $z_1,\ldots, z_n$ to be the coordinates of Lemma~\ref{lm:u_coordinates}.

We begin by proving lemmas concerning residues of a certain family of differential forms. These lemmas will be the central tool in the proof of Theorem~\ref{thm:norm}.

Recall for fixed $1\le i\le n$ and $l\ge 1,$ we fixed $r_1,\ldots, r_n\in \Z_{\ge0}$ such that $r_i\ge1$ and $r_1+\ldots+r_n=l.$ Define $l_1,\ldots,l_n\in \Z_{\ge0}$ by
    \[l_j=
         \left\{\begin{array}{ll}
        r_j+1, &  j\ne i,\\
       r_i, & j=i.\\
        \end{array} \right.\]
So, $l_1+\ldots+l_n=n+l-1.$ Let $\Vec{m}=(m_1,\ldots, m_n)\in \Z^n.$ In the following, write $\TT=(-1)^{\frac{n(n+1)}{2}}T.$
Define
\[
g_{\Vec{m}}(\Vec{z}):=\frac{\prod_{j=1}^n z_j^{m_j}}{\prod_{j=1}^n(z_j\prod_{k=1}^nz_k- \TT)^{l_j}}d\Vec{z}.
\]

\begin{lm}\label{lm:residue_1}
    Let $n\in\Z_{>0},$ $l\ge1$ and   
   $m_1,\ldots,m_n\in\Z.$ 
    Then,
   \[
\mathrm{Res}_\infty(g_{\Vec{m}}(\Vec{z}))=K_{\Vec{m},n}\cdot T^{-(n+l-1)+\frac{n+\sum_{j=1}^nm_j}{n+1}},\quad K_{\Vec{m},n}\in\C.
   \]
\end{lm}
\begin{proof} 
Since
    \[
\mathrm{Res}_\infty(g_{\Vec{m}}(\Vec{z}))=-\mathrm{Res}_0 (\frac{1}{z_1^2\cdots z_n^2}g_{\Vec{m}}(\frac{1}{\Vec{z}})),
    \]
    it suffices to compute the coefficient of $\frac{1}{z_1\cdots z_n}$ in the Laurent series of $-\frac{1}{z_1^2\cdots z_n^2}g_i(\frac{1}{\Vec{z}})$ at $0.$
We have
\begin{align*}
-\frac{1}{z_1^2\cdots z_n^2}g_{\Vec{m}}(\frac{1}{\Vec{z}})&=\frac{-1}{z_1^2\cdots z_n^2}\cdot \frac{\prod_{j=1}^nz_j^{-m_j}\cdot \prod_{j=1}^n\big(z_j\prod_{k=1}^nz_k\big)^{l_j}}{\prod_{j=1}^n(1-\TT z_j\prod_{k=1}^nz_k)^{l_j}}d\Vec{z}\\
&=-\frac{\prod_{j=1}^nz_j^{-m_j-2+l_j}\cdot \prod_{k=1}^nz_k^{\sum_{j=1}^nl_j}}{\prod_{j=1}^n(1-\TT z_j\prod_{k=1}^nz_k)^{l_j}}d\Vec{z}\\
&=-\frac{\prod_{j=1}^n z_j^{n+l-3+l_j-m_j}}{\prod_{j=1}^n(1-\TT z_j\prod_{k=1}^nz_k )^{l_j}}d\Vec{z}\\
&=-\prod_{j=1}^n z_j^{n+l-3+l_j-m_j}\cdot \prod_{j=1}^n\big( \sum_{a=0}^\infty (\TT z_j\prod_{k=1}^nz_k)^a
\big)^{l_j}d\Vec{z}\\
&=-\prod_{j=1}^n z_j^{n+l-3+l_j-m_j}\cdot \prod_{j=1}^n\bigg( \prod_{t=1}^{l_j} \big(\sum_{a_{j_t}=0}^\infty (\TT z_j\prod_{k=1}^nz_k)^{a_{j_t}}
\big)\bigg)d\Vec{z}.
\end{align*}
Hence, the coefficient of $\frac{1}{z_1\cdots z_n}$ is $K_{\Vec{m}}\cdot \TT^{\sum_{t=1}^{l_1}a_{j_t}+\ldots+\sum_{t=1}^{l_n}a_{j_t} },$ where $K_{\Vec{m}}\in\C$ and
\[
\sum_{t=1}^{l_1}a_{1_t}+\ldots+ 2\sum_{t=1}^{l_j}a_{j_t}+\ldots+\sum_{t=1}^{l_n}a_{n_t}=-(n+l-3+l_j-m_j)-1,\quad \forall j\in\{1,\ldots, n\}.
\]
Thus, 
\[
\sum_{t=1}^{l_1}a_{1_t}+\ldots+\sum_{t=1}^{l_n}a_{n_t}=-(n+l-1)+\frac{n+\sum_{j=1}^nm_j}{n+1},
\]
which implies
$\mathrm{Res}_\infty(g_{\Vec{m}}(\Vec{z}))=K_{\Vec{m},n}\cdot T^{-(n+l-1)+\frac{n+\sum_{j=1}^nm_j}{n+1}}$ for some $K_{\Vec{m},n}\in\C$ that depends on $\Vec{m}$ and $n.$
Note that since
$a_{j_t}$ are non-negative integers, it follows that if $-(n+l-1)+\frac{n+\sum_{j=1}^nm_j}{n+1}\not\in\Z,$ then automatically $K_{\Vec{m},n}=0.$
\end{proof}

\begin{cor}\label{cor:residue_1}
    Let $n\in\Z_{>0}$ and $l\ge1.$ Let $m_1,\ldots,m_n\in\Z$ such that $\sum_{j=1}^nm_j\ge (n+1)(n+l-1)-n.$ Then,
    \[
    \nu_{\Rsf}\bigg(\sum_{x\in C_{\Wsf}}\mathrm{Res}_x(g_{\Vec{m}}(\Vec{z}))\bigg)\ge 0.
    \]
\end{cor}
\begin{proof}
By assumption, $\sum_{j=1}^nm_j\ge (n+1)(n+l-1)-n>0,$ so not all of $m_j$ are negative. Hence, $C_{\Wsf}=C_{g_{\Vec{m}}},$ and it follows that
     \[
     \sum_{x\in C_{\Wsf}}\mathrm{Res}_x(g_{\Vec{m}}(\Vec{z}))=-\mathrm{Res}_\infty(g_{\Vec{m}}(\Vec{z})).
     \]
     Therefore,
     \begin{align*}
     \nu_{\Rsf}\bigg( \sum_{x\in C_{\Wsf}}\mathrm{Res}_x(g_{\Vec{m}}(\Vec{z})) \bigg)&=\nu_{\Rsf}\bigg(\mathrm{Res}_\infty (g_{\Vec{m}}(\Vec{z}))\bigg)\\
    &\overset{\ref{lm:residue_1}}=-(n+l-1)+\frac{n+\sum_{j=1}^nm_i}{n+1}\\
    &\ge 0.
\end{align*}
\end{proof}

\begin{lm}\label{lm:residue3}
    Let $n\in\Z_{>0}$ and $l\ge1.$ Let $m_1,\ldots,m_n\in\Z$ such that $\sum_{i=1}^nm_i< (n+1)(n+l-1)-n.$ 
    Then,
\[
\mathrm{Res}_\infty(g_{\Vec{m}}(\Vec{z}))= 0.
\]
\end{lm}
\begin{proof}
We have
\[  |\mathrm{Res}_\infty(g_{\Vec{m}}(\Vec{z}))|=\bigg|\frac{1}{(2\pi i)^n} \oint_{\substack{|z_j|=R_j\\ 1\le j\le n}}\frac{\prod_{j=1}^n z_j^{m_j}}{\prod_{j=1}^n(z_j\prod_{k=1}^nz_k- \TT)^{l_j}}dz_1\cdots dz_n \bigg|.
\]
    Taking the limit $R_j\to \infty$ for every $1\le j\le n$ yields

\begin{align*}
    |\mathrm{Res}_\infty(g_{\Vec{m}}(\Vec{z}))|&\le \frac{1}{(2\pi)^n}\lim_{\substack{R_j\to\infty\\ 1\le j\le n}} \oint_{\substack{|z_j|=R_j\\ 1\le j\le n}}\bigg|\frac{\prod_{j=1}^n z_j^{m_j}}{\prod_{j=1}^n(z_j\prod_{k=1}^nz_k- \TT)^{l_j}}\bigg|dz_1\cdots dz_n\\
    &\le \frac{1}{(2\pi)^n}\lim_{\substack{R_j\to\infty\\ 1\le j\le n}}\oint_{\substack{|z_j|=R_j\\ 1\le j\le n}}\frac{\prod_{j=1}^n R_j^{m_j}}{\prod_{j=1}^n|R_j\prod_{k=1}^nR_k- |T||^{l_j}}dz_1\cdots dz_n\\
    &= \lim_{\substack{R_j\to\infty\\ 1\le j\le n}}\frac{\prod_{j=1}^n R_j^{m_j+1}}{\prod_{j=1}^n|R_j\prod_{k=1}^nR_k- |T||^{l_j}}\\
    &=0,
\end{align*}
where the last equality follows from the assumption
$\sum_{j=1}^n m_j<(n+1)(n+l-1)-n.$

\end{proof}

\begin{lm}\label{lm:residue4}
     Let $n\in\Z_{>0},$ $l\ge1$ and $m_1,\ldots,m_n\in\Z_{<0}.$ Then, 
     \[
     \mathrm{Res}_0(g_{\Vec{m}}(\Vec{z}))=K_{\Vec{m},n}\cdot T^{-(n+l-1)+\frac{n+\sum_{j=1}^nm_j}{n+1}},\quad K_{\Vec{m},n}\in\C.
     \]
\end{lm}
\begin{proof} 
We need to compute the coefficient of $\frac{1}{z_1\cdots z_n}$ in the Laurent series of $g_{\Vec{m}}(\Vec{z})$ at $0.$
We have
\begin{align*}
g_{\Vec{m}}(\Vec{z})&=\frac{\prod_{j=1}^n z_j^{m_j}}{\prod_{j=1}^n(z_j\prod_{k=1}^nz_k -\TT)^{l_j}}d\Vec{z}\\
&=\frac{(-1)^{\sum_{j=1}^nl_j}\cdot \prod_{j=1}^nz_j^{m_j}}{\prod_{j=1}^n \TT^{l_j}\cdot\prod_{j=1}^n(1 -\frac{z_j\prod_{k=1}^nz_k}{\TT})^{l_j}}d\Vec{z}\\
&=\frac{(-1)^{n+l-1}}{\TT^{\sum_{j=1}^nl_j}}\cdot\prod_{j=1}^n z_j^{m_j}\cdot \prod_{j=1}^n\big( \sum_{a=0}^\infty (\frac{z_j\prod_{k=1}^nz_k}{\TT})^a
\big)^{l_j}d\Vec{z}\\
&=\frac{(-1)^{n+l-1}}{\TT^{n+l-1}}\cdot\prod_{j=1}^n z_j^{m_j}\cdot \prod_{j=1}^n\bigg( \prod_{t=1}^{l_j} \big(\sum_{a_{j_t}=0}^\infty (\frac{z_j\prod_{k=1}^nz_k}{\TT})^{a_{j_t}}
\big)\bigg)d\Vec{z}.
\end{align*}
Hence, the coefficient of $\frac{1}{z_1\cdots z_n}$ is $K_{\Vec{m}}\cdot \TT^{-(n+l-1)-\sum_{t=1}^{l_1}a_{1_t}-\ldots-\sum_{t=1}^{l_n}a_{n_t} },$ where $K_{\Vec{m}}\in\C$ and
\[
\sum_{t=1}^{l_1}a_{1_t}+\ldots+ 2\sum_{t=1}^{l_j}a_{j_t}+\ldots+\sum_{t=1}^{l_n}a_{n_t}=-m_j-1, \quad \forall j\in\{1,\ldots, n\}.
\]
Thus,     \[\sum_{t=1}^{l_1}a_{1_t}+\ldots+\sum_{t=1}^{l_n}a_{n_t}=\frac{-n-\sum_{j=1}^nm_j}{n+1},
\]
which implies 
$\mathrm{Res}_0(g_{\Vec{m}}(\Vec{z}))=K_{\Vec{m},n}\cdot T^{-(n+l-1)+\frac{n+\sum_{j=1}^nm_j}{n+1}}$ for some $K_{\Vec{m},n}\in\C$ that depends on $\Vec{m}$ and $n.$ Note that since $a_{j_t}$ are non-negative integers, it follows that if $\frac{-n-\sum_{j=1}^nm_j}{n+1}\not\in\Z,$ then automatically $K_{\Vec{m},n}=0.$
\end{proof}

Recall the Hom-complexes of the normed matrix factorization category are equipped with families of valuations defined in~\eqref{eq:valuation_lambda_hom}. These families induce valuations on the Hom-complexes as defined in~\eqref{eq:val_hom_01}.

Recall $w\in \Rsf$ is element satisfying $\nu_{\Rsf}(w)\ge\sup_{\zeta\in\interior{\simp}}\nu_\zeta^{\Ssf}(\Wsf).$
Let $(M^{j}, D_{M^{j}}, \{\nu^{M^{j}}_\zeta\}_{\zeta\in\interior{\simp}})$ for $j=1,\ldots, l$ be a collection of $l$ objects of $\MF(\Wsf,w).$
Since $\Wsf=(-1)^{\frac{n(n+1)}{2}}z_0+\ldots+ z_n$ and $D_{M^j}^2=\Wsf-w,$ it follows that 
\[
\min\{\zeta_0,\ldots, \zeta_n\}=\nu_\zeta^{\Ssf}(\Wsf)= \min\{\nu_\zeta^{\Ssf}(\Wsf), \nu_{\Rsf}(w) \}={\nu_{M^{j}}}_\zeta({D_{M^{j}}}^2)=2{\nu_{M^{j}}}_\zeta(D_{M^{j}}). 
\]
Hence, ${\nu_{M^j}}_\zeta(D_{M^{j}})=\min\{\frac{1}{2}\zeta_0,\ldots,\frac{1}{2}\zeta_n\}$ for all $j\in\{1,\ldots ,l\}.$
Abbreviate
\[
{\nu_j}_\zeta={\nu_{M^{j},M^{j+1}}}_\zeta, \quad  \nu_j=\nu_{M^{j},M^{j+1}}, \quad j\in\{1,\ldots, l-1\},
\]
\[
{\nu_l}_\zeta={\nu_{M^{l},M^{1}}}_\zeta, \quad  \nu_l=\nu_{M^{l},M^{1}}.
\]

\begin{proof}[Proof of Theorem~\ref{thm:norm}] 
    By Lemma~\ref{Theta} $\Theta$ is an $\infty$-trace. Hence, by Lemmas~\ref{lm:vanish_theta},~\ref{lm:degree_theta}, we need to show the set $\mD$ defined in~\ref{eq:preserves_val_D} is a discrete subset of $\R_{\ge0}.$ By Lemma~\ref{lm:ima_m_eq_ima_r}, we have $\ima\nu_{M^1, M^2}=\ima\nu_{\Rsf}$ for all objects $M^1, M^2$ of $\MF(\Wsf,w),$ where $\ima\nu_{\Rsf}=\frac{1}{n+1}\Z\cup\{\infty\}.$ Hence, $\mD$ is discrete. It remains to show $\mD\subset \R_{\ge0}.$ 

Let    
$\Phi_j\in\Hom(M^j, M^{j+1})$ where $j\in\{1,\ldots, l-1\},$ and $\Phi_l\in\Hom(M^l, M^1)$ such that $\Theta_l(\Phi_l\otimes\ldots\otimes\Phi_1)\ne0.$  
    We need to prove $\nu_{\Rsf}\big(\Theta_l(\Phi_l\otimes\ldots\otimes\Phi_1)\big)\ge\sum_{j=1}^l\nu_j(\Phi_j)$ for all $l\ge1.$
    Hence, it suffices to show  \begin{equation}\label{eq:F}
    \nu_{\Rsf}\bigg(\sum_{x\in C_{\Wsf}}\mathrm{Res}_x\big(F_{l,\Vec{j},\Vec{i},\Vec{\Phi}}(\Vec{z})\big)\bigg)\ge\sum_{j=1}^l\nu_j(\Phi_j).
    \end{equation}
First, since ${\nu_{M^{k}}}_\zeta(D_{M^{k}})=\min\{\frac{1}{2}\zeta_0,\ldots,\frac{1}{2}\zeta_n\},$ it follows that ${\nu_{M^{k}}}_\zeta(\partial_jD_{M^{k}})=\min\{\frac{1}{2}\zeta_0-\zeta_j,\ldots,\frac{1}{2}\zeta_n-\zeta_j\}$ for all $k\in\{1,\ldots, l\}$ and $j\in\{1,\ldots, n\}.$ Hence,
    \begin{align*}
        \nu^{\Ssf}_\zeta(f_{l,\Vec{j},\Vec{i},\Vec{\Phi}}(\Vec{z}))&\ge\sum_{j=1}^n l_j\min\{\frac{1}{2}\zeta_0-\zeta_j,\ldots,\frac{1}{2}\zeta_n-\zeta_j\}+\sum_{j=1}^l{\nu_j}_\zeta(\Phi_j)\\
        &\ge \sum_{j=1}^n l_j\min\{\frac{1}{2}\zeta_0-\zeta_j,\ldots,\frac{1}{2}\zeta_n-\zeta_j\}+\sum_{j=1}^l\nu_j(\Phi_j).
    \end{align*}
    The function $f_{l,\Vec{j},\Vec{i},\Vec{\Phi}}(\Vec{z})$ is a linear combination of functions of the form $\frac{T^a}{\prod_{j=1}^nz_j^{m_j}},$ so in order to prove~\eqref{eq:F}, we may assume 
    $
    f_{l,\Vec{j},\Vec{i},\Vec{\Phi}}(\Vec{z})=\frac{T^a}{\prod_{j=1}^nz_j^{m_j}}
    $ where $a\in\frac{1}{n+1}\Z,$ $m_j\in\Z.$
    Thus,
    \[
\nu^{\Ssf}_\zeta(f_{l,\Vec{j},\Vec{i},\Vec{\Phi}}(\Vec{z}))=a-\sum_{j=1}^n\zeta_jm_j
    .
    \]
We obtain,
\begin{equation}\label{eq:zeta}
a-\sum_{j=1}^n\zeta_jm_j\ge \sum_{j=1}^n l_j\min\{\frac{1}{2}\zeta_0-\zeta_j,\ldots,\frac{1}{2}\zeta_n-\zeta_j\}+\sum_{j=1}^l\nu_j(\Phi_j).
\end{equation}
Next, consider the limit in~\eqref{eq:zeta} where $\zeta_0\to 1$ and $\zeta_j\to0$ for $j\in\{1,\ldots, n\}.$ Thus, we obtain $a\ge\sum_{j=1}^l\nu_j(\Phi_j).$
In addition, for each $k\in\{1,\ldots,n\},$ consider the limit in~\eqref{eq:zeta} where $\zeta_k\to 1$ and $\zeta_j\to 0$ for all $j\ne k.$ Thus, we obtain $a+l_k\ge m_k+\sum_{j=1}^l\nu_j(\Phi_j).$ Hence, 
\begin{equation}\label{eq:n_l_1}
n+l-1=\sum_{j=1}^n l_j\ge\sum_{j=1}^nm_j+ n\sum_{j=1}^l\nu_j(\Phi_j)-na.
\end{equation}
    We have
    \begin{align*}
    F_{l,\Vec{j},\Vec{i},\Vec{\Phi}}(\Vec{z})&=\frac{f_{l,\Vec{j},\Vec{i},\Vec{\Phi}}(\Vec{z})\wedge \Omega}{(\partial_1\Wsf)^{l_1}\cdots (\partial_n\Wsf)^{l_n}}\\
    &=\frac{T^a}{\prod_{j=1}^nz_j^{m_j}}\cdot \frac{1}{\prod_{j=1}^nz_j\cdot\prod_{j=1}^n(1-\frac{\TT}{z_j\prod_{k=1}^nz_k})^{l_j}}d\Vec{z}\\
    &=\frac{T^a\cdot \prod_{j=1}^nz_j^{-m_j+l+n-2+l_j}}{\prod_{j=1}^n(z_j\prod_{k=1}^nz_k-\TT)^{l_j}}d\Vec{z}.
     \end{align*}
Consider the following three cases:
\begin{enumerate}[label=(\arabic*)]
\item Suppose $\sum_{j=1}^n(-m_j+l+n-2+l_j)\ge (n+1)(n+l-1)-n.$ Thus, by Corollary~\ref{cor:residue_1} and since $a\ge\sum_{j=1}^l\nu_j(\Phi_j)$ equation~\eqref{eq:F} holds.

\item Suppose $\sum_{j=1}^n(-m_j+l+n-2+l_j)< (n+1)(n+l-1)-n,$ and suppose there exists $k\in\{1,\ldots,n\}$ such that  $-m_k+l+n-2+l_k\ge0.$ Thus, $C_{\Wsf}=C_{F_{l,\Vec{j},\Vec{i},\Vec{\Phi}}}.$ So, 
$-\mathrm{Res}_\infty\big(F_{l,\Vec{j},\Vec{i},\Vec{\Phi}}(\Vec{z})\big)= \sum_{x\in C_{\Wsf}}\mathrm{Res}_x\big(F_{l,\Vec{j},\Vec{i},\Vec{\Phi}}(\Vec{z})\big),$ and it follows by Lemma~\ref{lm:residue3} that
\[
\nu_{\Rsf}\bigg(\sum_{x\in C_{\Wsf}}\mathrm{Res}_x\big(F_{l,\Vec{j},\Vec{i},\Vec{\Phi}}(\Vec{z})\big)\bigg)=\infty.
\]

\item  Suppose $\sum_{j=1}^n(-m_j+l+n-2+l_j)< (n+1)(n+l-1)-n,$ and suppose $-m_k+l+n-2+l_k<0$ for all $k\in\{1,\ldots,n\}.$  Hence,
\[
0>\sum_{k=1}^n\zeta_k(n+l-2)+\sum_{k=1}^n\zeta_k(l_k-m_k),
\]
where $\zeta_1,\ldots, \zeta_n\in \R$ such that there exists $\zeta_0\in \R$ such that $(\zeta_0,\ldots, \zeta_n)\in\interior{\simp}.$
Since, $a+l_k\ge m_k+\sum_{j=1}^l\nu_j(\Phi_j)$ for all $k\in\{1,\ldots, n\},$ it follows that 
\begin{align*}
0&>\sum_{k=1}^n\zeta_k(n+l-2)+\sum_{k=1}^n\zeta_k\big(-a+ \sum_{j=1}^l\nu_j(\Phi_j)\big). 
\end{align*}
Taking the limit $(\zeta_0,\zeta_1,\ldots, \zeta_n)\to (0,\frac{1}{n},\ldots, \frac{1}{n}),$ we obtain $a>n+l-2+\sum_{j=1}^l\nu_j(\Phi_j).$ So, by~\eqref{eq:n_l_1} we get 
\[
    (n+1)a>-1+\sum_{j=1}^n m_j+(n+1)\sum_{j=1}^l\nu_j(\Phi_j).
    \]
By Lemma~\ref{lm:ima_m_eq_ima_r} we have $\ima\nu_j=\ima\nu_{\Rsf}=\frac{1}{n+1}\Z\cup\{\infty\}$ for all $j\in\{1,\ldots,l\},$ so $-1+\sum_{j=1}^n m_j+(n+1)\sum_{j=1}^l\nu_j(\Phi_j)\in\Z.$ Hence,   
\begin{equation}\label{eq:a_a_a}
    (n+1)a\ge\sum_{j=1}^n m_j+(n+1)\sum_{j=1}^l\nu_j(\Phi_j).
\end{equation}
Since \[
-\mathrm{Res}_\infty\big(F_{l,\Vec{j},\Vec{i},\Vec{\Phi}}(\Vec{z})\big)=\sum_{x\in C_{\Wsf}}\mathrm{Res}_x\big(F_{l,\Vec{j},\Vec{i},\Vec{\Phi}}(\Vec{z})\big)+\mathrm{Res}_0\big(F_{l,\Vec{j},\Vec{i},\Vec{\Phi}}(\Vec{z})\big),
\]
it follows that
\begin{align*}
\nu_{\Rsf}\bigg(\sum_{x\in C_{\Wsf}}\mathrm{Res}_x\big(F_{l,\Vec{j},\Vec{i},\Vec{\Phi}}(\Vec{z})\big)\bigg)&\overset{\ref{lm:residue3}}=\nu_{\Rsf}\bigg(\mathrm{Res}_0\big(F_{l,\Vec{j},\Vec{i},\Vec{\Phi}}(\Vec{z})\big)\bigg)\\
&\overset{\ref{lm:residue4}}\ge a-(n+l-1)+\frac{n+\sum_{j=1}^n(-m_j+l+n-2+l_j)}{n+1}\\
&=\frac{(n+1)a-\sum_{j=1}^nm_j}{n+1}\\
&\overset{\eqref{eq:a_a_a}}{\ge}\sum_{j=1}^l\nu_j(\Phi_j).
\end{align*}
\end{enumerate}
This completes the proof.
\end{proof}
\subsection{Normed Calabi-Yau object}\label{section:normed_cy_object}
Our goal is to prove Theorem~\ref{thm:cy}.
The proof of this theorem relies on the following lemma.
\begin{lm}\label{lm:theta_of_h_0}
    Let $h_0\in\End(\Msf)$ be the endomorphism given by left multiplication by $-\frac{z_1\cdots z_n}{T^{1/2}}e_0.$ Then,
     $\Theta_1(h_0)=1.$
\end{lm}
We use the notations introduced in Section~\ref{section:spherical_object}.
Let $n\ge1$ be an odd number.  
Write $Y_0(n)=\bigcup_{i=0}^nY_0(i,n).$
Let $I=(j_1,\ldots,j_i)\in Y(i,n).$
We denote by $I^c$ the $n-i+1$-tuple $(a_1,\ldots,a_{n-i+1})\in Y(i,n)$ such that 
\[
\{j_1,\ldots,j_i,a_1,\ldots,a_{n-i+1}\}=\{0,\ldots,n\}.
\]
We write $z_I$ for $z_{j_1}\cdots z_{j_i},$ where $z_\emptyset=1.$ Let $a\in\{0,\ldots, n\}$ such that there exits $1\le k\le i$ such that $j_k<a<j_{k+1}.$ Let $I\cup\{a\}$ denote the $i+1$-tuple $(j_1,\ldots,j_k,a,j_{k+1},\ldots ,j_i).$ 
For $I\in Y_0(i,n),$ let $\rho_I\in\{0,1\}$ denote the number such that  
\[
 e_I\cdot e_0\cdots e_n=(-1)^{\rho_I}z_Ie_{I^c}.
\]
Recall $\overline{T}=(-1)^{\frac{n(n+1)}{2}}T.$
We begin by proving Lemma~\ref{lm:theta_of_h_0}. The formula of $\Theta_1$ involves supertraces of certain operators. To simplify these computations, we consider the basis $\{e_I\}_{I\in Y_0(n)}$ of $\Msf,$ in which the supertraces become much easier to compute. 
The elements of the basis $\{e_I\}_{I\in Y_0(n)}$ satisfy
\[e_i\cdot e_I=
         \left\{\begin{array}{ll}
(-1)^{\mu_{I,i}} e_{I\cup \{i\}}, &   i\not\in I, i\ne0,\\
       (-1)^{\mu_{I,i}}z_i\cdot e_{I\backslash \{i\}}, & i\in I, \\
     (-1)^{\rho_I+\frac{n(n+1)}{2}} \frac{z_{I\cup\{0\}}}{T^{1/2}}\cdot e_{I^c\backslash\{0\}}, & i=0. 
        \end{array} \right.\]
Hence, for $j\ne0,$ we obtain
\[\partial_je_i\cdot e_I=
         \left\{\begin{array}{ll}
0, &  i\not\in I, i\ne0,\\
       (-1)^{\mu_{I,i}}\delta_{ij}e_{I\backslash \{i\}}, & i\in I, \\
       0, & j\in I, i=0,\\
       (-1)^{1+\rho_I+\frac{n(n+1)}{2}}\frac{T^{1/2}}{z_j\cdot z_{I^c\backslash\{0\}}}e_{I^c\backslash\{0\}}, & j\not\in I, i=0.\\
        \end{array} \right.\]
Let $h_0=-\frac{z_1\cdots z_n}{T^{1/2}}e_0.$    We get
    \begin{align*}
   \sum_{i=1}^n&\sum_{(j_1,\ldots, j_{n-1})\in S^i_n} \mathrm{str}\big( h_0\partial_i\Dsf\partial_{j_1}\Dsf\cdots \partial_{j_{n-1}}\Dsf\big) \\&=\sum_{i=1}^n\sum_{(j_1,\ldots, j_{n-1})\in S^i_n}\mathrm{str}\big( h_0\partial_i(e_0+\ldots +e_n)\cdots \partial_{j_{n-1}}(e_0+\ldots +e_n)\big)\\
    &=\sum_{(b_1,\ldots ,b_n)\in\{0,\ldots, n\}^n}\sum_{(a_1,\ldots, a_n)\in S_n} \mathrm{str}\big( h_0\partial_{a_1}e_{b_1}\cdots \partial_{a_n}e_{b_n}\big)\\ 
     \end{align*}
Observe:
\begin{enumerate}[label=(\arabic*)]
    \item $ \mathrm{str}\big( h_0\partial_{a_1}e_{b_1}\cdots \partial_{a_n}e_{b_n}\big)\ne0$ only if either $b_i=0$ for all $i\in\{1,\ldots,n\}$ or $a_i=b_i$ for all $i\in\{1,\ldots,n\}.$
    \item\label{case_2:supertrace} Assume $a_i=b_i$ for all $i\in\{1,\ldots, n\},$ and denote $\sigma=(a_1,\ldots, a_n).$ So, $\sigma(i)=a_i.$
    Then, $h_0\partial_{\sigma(1)}e_{\sigma(1)}\cdots \partial_{\sigma(n)}e_{\sigma(n)}\cdot e_I\ne0$ if and only if $I=(0,\ldots, n).$ We get,
    \[
    h_0\partial_{\sigma(1)}e_{\sigma(1)}\cdots \partial_{\sigma(n)}e_{\sigma(n)}\cdot e_{(1,\ldots, n)}=(-1)^{\frac{n(n-1)}{2}+\mathrm{sgn}(\sigma)}h_0\cdot e_\emptyset=(-1)^{\mathrm{sgn}(\sigma)}\cdot e_{(1,\ldots, n)}.
    \]
    Since $e_{(1,\ldots, n)}$ is an element of odd degree, we get
    \[
    \mathrm{str}\big( h_0\partial_{\sigma(1)}e_{\sigma(1)}\cdots \partial_{\sigma(n)}e_{\sigma(n)}\big)=(-1)^{1+\mathrm{sgn}(\sigma)}. 
    \]
    \item\label{case_3:supertrace} Assume $b_i=0$ for all $i\in\{1,\ldots,n\}.$ Then, $h_0\partial_{a_1}e_0\cdots \partial_{a_n}e_0\cdot e_I\ne0$ if and only if $I=(j_1,\ldots, j_{\frac{n+1}{2}})$ where
    $\{j_1,\ldots, j_{\frac{n+1}{2}}\}=\{a_1,a_3,\ldots,a_n\}.$ In this case, we get
    \begin{align*}
     h_0\partial_{a_1}e_0\cdots \partial_{a_n}e_0\cdot e_I&=(-1)^{1+\frac{n+1}{2}\rho_I+\frac{n-1}{2}\rho_{I^c\cup\{0\}}+\frac{n(n+1)}{2}}h_0\cdot\frac{T^{n/2}}{\prod_{j=1}^nz_j^{\frac{n+1}{2}}\cdot z_{I^c\backslash\{0\}}} e_{I^c\backslash\{0\}}\\
     &=(-1)^{\frac{n+1}{2}(\rho_I+\rho_{I^c\cup\{0\}})}\frac{T^{n/2}}{ \prod_{j=1}^nz_j^{\frac{n+1}{2}}} e_I.
     \end{align*}
     Note that $e_I$ is of even degree if $n\equiv3 \pmod{4},$ and of odd degree if $n\equiv1 \pmod{4}.$
     Thus,
    \[
    \mathrm{str}\big( h_0\partial_{a_1}e_0\cdots \partial_{a_n}e_0\big)=(-1)^{\frac{n+1}{2}(\rho_I+\rho_{I^c\cup\{0\}})+\frac{n(n+1)}{2}}\frac{T^{n/2}}{ \prod_{j=1}^nz_j^{\frac{n+1}{2}}}.
    \]
    Let $\sigma=(a_1,\ldots, a_n),$ and write $\xi_{\sigma,n}=\frac{n+1}{2}(\rho_I+\rho_{I^c\cup\{0\}})+\frac{n(n+1)}{2}.$
\end{enumerate}
Let $\sigma\in S_n.$ Define
\[ G_\sigma(\Vec{z}):=\frac{\mathrm{str}\big( h_0\partial_{\sigma(1)}e_{\sigma(1)}\cdots \partial_{\sigma(n)}e_{\sigma(n)}\big) \wedge\Omega}{\partial_1\Wsf\cdots \partial_n \Wsf}.
\]
\begin{lm}\label{lm:G_residues_x}
    $\sum_{x\in C_{\Wsf}}\mathrm{Res}_xG_\sigma(\Vec{z})=(-1)^{1+\mathrm{sgn}(\sigma)}.$
\end{lm}
\begin{proof}
    We have
    \begin{align*}
       G_\sigma(\Vec{z}) &\overset{\ref{case_2:supertrace}}{=}\frac{(-1)^{1+\mathrm{sgn}(\sigma)}}{\prod_{j=1}^nz_j \cdot \prod_{j=1}^n\big( 1-\frac{\overline{T}}{z_j\prod_{i=1}^nz_i}  \big)}dz_1\wedge\ldots\wedge dz_n\\
   &=\frac{(-1)^{1+\mathrm{sgn}(\sigma)}\prod_{j=1}^nz_j^n}{\prod_{j=1}^n\big(  z_j\prod_{i=1}^nz_i-\overline{T} \big)}dz_1\wedge\ldots\wedge dz_n.
    \end{align*}

So, $C_{G_\sigma}=C_{\Wsf}.$ We have
\begin{align*}
    \mathrm{Res}_{\infty}G_\sigma(\Vec{z})&=- \mathrm{Res}_0 (\frac{1}{z_1^2\cdots z_n^2}G_\sigma(\frac{1}{\Vec{z}}))\\
    &=\mathrm{Res}_0 \frac{(-1)^{\mathrm{sgn}(\sigma)}}{\prod_{j=1}^nz_j\cdot\prod_{j=1}^n\big(1-\overline{T}z_j\prod_{i=1}^nz_i   \big)}dz_1\wedge\ldots\wedge dz_n\\
    &=\mathrm{Res}_{z_n=0}\cdots\mathrm{Res}_{z_1=0}\frac{(-1)^{\mathrm{sgn}(\sigma)}}{\prod_{j=1}^nz_j\cdot\prod_{j=1}^n\big(1-\overline{T}z_j\prod_{i=1}^nz_i   \big)}dz_1\wedge\ldots\wedge dz_n\\
    &=\mathrm{Res}_{z_n=0}\cdots\mathrm{Res}_{z_2=0}\frac{(-1)^{\mathrm{sgn}(\sigma)}}{\prod_{j=2}^nz_j}dz_2\wedge\ldots\wedge dz_n\\
    &=(-1)^{\mathrm{sgn}(\sigma)}.
\end{align*}
In the fourth equality we used the following statement. If $f(z)=\frac{1}{z-a}g(z),$ where $g$ is analytic in a neighborhood of $a,$ then $\mathrm{Res}_af(z)=g(a).$
Hence, $-\mathrm{Res}_{\infty}G_\sigma(\Vec{z})=\sum_{x\in C_{\Wsf}}\mathrm{Res}_xG_\sigma(\Vec{z})=(-1)^{1+\mathrm{sgn}(\sigma)}.$
\end{proof}
Let $\sigma\in S_n.$ Define
\[ F_\sigma(\Vec{z}):=\frac{\mathrm{str}\big( h_0\partial_{\sigma(1)}e_0\cdots \partial_{\sigma(n)}e_0\big) \wedge\Omega}{\partial_1\Wsf\cdots \partial_n \Wsf}.\]
\begin{lm}\label{lm:F_residues_x}
    $\sum_{x\in C_W}\mathrm{Res}_xF_\sigma(\Vec{z})=0.$
\end{lm}
\begin{proof}
     We have
    \begin{align*}
       F_\sigma(\Vec{z}) &\overset{\ref{case_3:supertrace}}{=}\frac{(-1)^{\xi_{\sigma,n}}T^{n/2}}{\prod_{j=1}^nz_j^{\frac{n+1}{2}} \cdot \prod_{j=1}^n\big( 1-\frac{\overline{T}}{z_j\prod_{i=1}^nz_i}  \big)}dz_1\wedge\ldots\wedge dz_n\\
&=\frac{(-1)^{\xi_{\sigma,n}}T^{n/2}\prod_{j=1}^nz_j^{\frac{n+1}{2}}}{\prod_{j=1}^n\big(  z_j\prod_{i=1}^nz_i-\overline{T} \big)}dz_1\wedge\ldots\wedge dz_n.
    \end{align*}
So, $C_{F_\sigma}=C_{\Wsf}.$ We have
\begin{align*}
\mathrm{Res}_{\infty}F_\sigma(\Vec{z})&=- \mathrm{Res}_0 (\frac{1}{z_1^2\cdots z_n^2}F_\sigma(\frac{1}{\Vec{z}}))\\
    &=-\mathrm{Res}_0 \frac{(-1)^{\xi_\sigma}T^{n/2}\prod_{j=1}^nz_j^{\frac{n-1}{2}}}{\prod_{j=1}^n\big(1-\overline{T}z_j\prod_{i=1}^nz_i   \big)}\\
    &=0.
\end{align*}
Hence,  $\sum_{x\in C_{\Wsf}}\mathrm{Res}_xF_\sigma(\Vec{z})=-\mathrm{Res}_{\infty}F_\sigma(\Vec{z})=0.$

\end{proof}

\begin{proof}[Proof of Lemma~\ref{lm:theta_of_h_0}]
\begin{align*}
    \Theta_1(h_0)&=\frac{1}{n!}\sum_{i=1}^n(-1)^i\sum_{(j_1,\ldots,j_{n-1})\in S^i_n}\mathrm{sgn}(j_1,\ldots,j_{n-1})\sum_{x\in C_{\Wsf}}\mathrm{Res}_x\frac{\mathrm{str}\big(h_0\partial_i\Dsf\partial_{j_1}\Dsf\cdots \partial_{j_{n-1}} \Dsf\big)\wedge\Omega}{\partial_1W\cdots \partial_nW}\\
    &=\frac{1}{n!}\sum_{\sigma\in S_n}\sum_{x\in C_{\Wsf}}(-1)^{\sigma(1)+\mathrm{sgn}(\sigma(2),\ldots,\sigma(n))}\mathrm{Res}_x\frac{\mathrm{str}\big( h_0\partial_{\sigma(1)}e_{\sigma(1)}\cdots \partial_{\sigma(n)}e_{\sigma(n)}\big) \wedge \Omega}{\partial_1W\cdots \partial_n W}\\
    &+\frac{1}{n!}\sum_{\sigma\in S_n}\sum_{x\in C_{\Wsf}}(-1)^{\sigma(1)+\mathrm{sgn}(\sigma(2),\ldots,\sigma(n))}\mathrm{Res}_x\frac{\mathrm{str}\big( h_0\partial_{\sigma(1)}e_0\cdots \partial_{\sigma(n)}e_0\big) \wedge \Omega}{\partial_1W\cdots \partial_n W}\\
    &\overset{\ref{lm:G_residues_x}+~\ref{lm:F_residues_x}}{=}\frac{1}{n!}\sum_{\sigma\in S_n}(-1)^{\sigma(1)+\mathrm{sgn}(\sigma(2),\ldots,\sigma(n))}\cdot (-1)^{1+\mathrm{sgn}(\sigma)}\\
    &=\frac{1}{n!}\sum_{\sigma\in S_n}1\\
    &=1.
\end{align*}
\end{proof}
\begin{proof}[Proof of Theorem~\ref{thm:cy}]
By Theorem~\ref{thm:norm} the $\infty$-trace $\Theta$ on $\MF(\Wsf,w)$ is normed, 
and by Theorem~\ref{lm:cohomology} we have $H^*(\End(\Msf)_0)\cong\R\oplus\R[n].$ So, we need to show that there exists $[\Phi]\in H^n(\End(\Msf)_0)$ such that $\Theta_1(\Phi)\equiv 1 \pmod{F^0\End(\Msf)}.$
Let $h_0=-\frac{z_1\cdots z_n}{T^{1/2}}e_0,$ and note that  $|h_0|=n,$ $\nu_{\Msf}(h_0)=0.$ In addition,
\[
\delta(h_0)= -2T^{1/2}\equiv0 \pmod{F^0\End(\Msf)}.
\]
Thus, $[h_0]\in H^n(\End(\Msf)_0).$ Therefore, the claim follows from Lemma~\ref{lm:theta_of_h_0}.
\end{proof}

\subsection{The case \texorpdfstring{$n=1$}{of one dimension}}\label{section:example}
Consider the object $\Msf$ of $\MF(\Wsf,0)$ where $n=1.$
In this case, we have
\[
\Wsf=z_1-\frac{T}{z_1}, \quad \Omega=\frac{1}{z_1}. 
\]
Our goal is to prove Proposition~\ref{prop:bounding cochain} and Corollary~\ref{cor:n_d_k}, which together imply Theorem~\ref{thm:gw_inv_numerical_invariants}. 
In order to prove Proposition~\ref{prop:bounding cochain}, we will need the following lemmas.
\begin{lm}\label{lm:res1} 
Let $l\ge 1,$ and let $m\in\Z$ such that $-1<m<2l-1.$ Then, 
    \[
   \sum_{x\in \{\pm iT^{1/2}\}} \mathrm{Res}_x\frac{z^m}{(z^2+T)^l}=0.
    \]
\end{lm}
\begin{proof}
Since $m>-1,$ the set of the critical points of $\frac{z^m}{(z^2+T)^l}$ consists the critical points of $\Wsf.$ 
Since $C_{\Wsf}=\{\pm iT^{1/2}\},$ it follows that
\[
-\mathrm{Res}_\infty\frac{z^m}{(z^2+T)^l}=\sum_{x\in \{\pm iT^{1/2}\}} \mathrm{Res}_x\frac{z^m}{(z^2+T)^l}.
\]
Hence, the claim follows from Lemma~\ref{lm:residue3} where $n=1.$
\end{proof}
\begin{lm}\label{lm:res2}
      Let $l\ge1.$ Then,
         \[
    \sum_{x\in \{\pm iT^{1/2}\}}\mathrm{Res}_x\frac{z^{2l-1}}{(z^2+T)^l}=1.
    \]
\end{lm}
\begin{proof}
    We prove it by induction on $l.$ For $l=1$ we have 
    \[
    \frac{z}{(z^2+T)}=\frac{1}{2}\bigg( \frac{1}{z-iT^{1/2}}+ \frac{1}{z+iT^{1/2}}\bigg).
    \]
    So,
    \[
    \sum_{x\in \{\pm iT^{1/2}\}}\mathrm{Res}_x\frac{z}{(z^2+T)}=1.
    \]
Let $l\ge2.$ Since
\[
\frac{z^{2l-3}}{(z^2+T)^{l-1}}=\frac{z^{2l-1}}{(z^2+T)^l}+\frac{Tz^{2l-3}}{(z^2+T)^l},
\]
it follows by Lemma~\ref{lm:res1} that
\[
 \sum_{x\in \{\pm iT^{1/2}\}}\mathrm{Res}_x\frac{z^{2l-1}}{
 (z^2+T)^{l-1}}= 
 \sum_{x\in \{\pm iT^{1/2}\}}\mathrm{Res}_x\frac{z^{2l-3}}{(z^2+T)^{l-1}}.
\]
    Hence, by the induction hypothesis we get
    \[
    \sum_{x\in \{\pm iT^{1/2}\}}\mathrm{Res}_x\frac{z^{2l-1}}{(z^2+T)^l}=1.
    \]
\end{proof}  
Let $b\in\Mmm(\End(\Msf)^s).$ Recall that $\Theta(\expp(b))$ does not depend on the choice of $y_b.$ Hence, we can work with the canonical choice of $y_b$ described in Section~\ref{section:constructing_y_b}.
Write
\[
   y_b= -\sum_{n=0}^\infty \sum_{m=0}^n \sum_{J\in\zzs}\sum_{I\in\Z_{>0}^{m+1}}C_J
\bigotimes_{k=0}^m\big( (c\otimes 1)^{\otimes j_k}\otimes b^{i_k}\big) ,
\]
where 
\[
b=(e^{-s}-e^s)\frac{z_1}{2T^{1/2}}\cdot e_0- (e^{-s}+e^s-2)\cdot\frac{e_1}{2}, \quad c=(e^s-e^{-s})T^{1/2},
\]
and the constants $C_J$ are computed inductively as described in Remark~\ref{rem:constant_c}. 
In the case of $n=1$, the formula of $\Theta_{l,x}$ given in Theorem~\ref{Theta}
becomes quite simple
\[
\Theta_{l,x}(\Phi_l\otimes\ldots \otimes\Phi_1)= -\mathrm{Res}_x\bigg( \frac{\mathrm{str}\big(\Phi_l\partial_1 \Dsf\cdots \Phi_1 \partial_1 \Dsf\big)\wedge \Omega}{(\partial_1\Wsf)^l} \bigg).
\]
In the following we identify $\Msf$ with ${\Ssf}^{\oplus2}$ via 
\[
1\mapsto \begin{pmatrix}
    1\\
    0
\end{pmatrix}, \qquad
e_1\mapsto \begin{pmatrix}
    0\\
    1
\end{pmatrix}.
\]
Hence, we can identify $\Ems$ with $\Rms\otimes_{\Rsf}M_{2\times 2}(\Ssf).$
Note that under this identification we have
\[
\Dsf= \begin{pmatrix}
    0 & z_1+T^{1/2}\\
    1 -\frac{T^{1/2}}{z_1} & 0
\end{pmatrix}.\]
\begin{lm}\label{lm:exp_y_b_compute}
    $\Theta(y_b)=0.$
\end{lm}

\begin{proof}
  Let $l\ge3.$  Since $\Theta$ is $\Rms$-multilinear, it suffices to prove that 
    \[
    \Theta_l( \bigotimes_{k=0}^m (1\otimes 1)^{\otimes j_k}\otimes b^{\otimes i_k})=0
    \]
for $m\ge 0$ and $2(j_0+\ldots+j_m)+i_0+\ldots +i_m+=l.$ 
We have
\[
b=\begin{pmatrix}
0&  (1-e^s)z_1\\
1-e^{-s}& 0
\end{pmatrix},\qquad \partial_1\Dsf=\begin{pmatrix}
0& 1 \\
\frac{T^{1/2}}{z_1^2}& 0
\end{pmatrix}.
\]
Since
\[
\prod_{k=0}^m\big( (\partial_1\Dsf)^{2j_k}(b\partial_1\Dsf)^{i_k} \big)=(\partial_1\Dsf)^{2(j_0+\ldots+j_m)}(b\partial_1\Dsf)^{i_0+\ldots+i_m},
\]
it follows that $\Theta_l( \bigotimes_{k=0}^m (1\otimes 1)^{\otimes j_k}\otimes b^{\otimes i_k})=\Theta_l(1^{\otimes2(j_0+\ldots+j_m)}\otimes b^{\otimes(i_0+\ldots +i_m)}).$
Write $j=j_0+\ldots+j_m,$ and $a=i_0+\ldots+i_m.$
We have,
\[
\partial_1\Dsf^{2j}(b\partial_1\Dsf)^a=
\begin{pmatrix}
    \frac{T^{j/2}}{z_1^{2j}} & 0\\
    0 &  \frac{T^{j/2}}{z_1^{2j}}
\end{pmatrix}
\begin{pmatrix}
        (1-e^s)^a\frac{T^{a/2}}{z_1^a}& 0\\
0& (1-e^{-s})^a
    \end{pmatrix}
\]
\[
= \begin{pmatrix}
         (1-e^s)^a \frac{T^{\frac{j+a}{2}}}{z_1^{2j+a}} &0 \\
          0 & (1-e^{-s})^a\frac{T^{j/2}}{z_1^{2j}}
    \end{pmatrix}.
\]
Thus,
\begin{align*}
\Theta_l( \bigotimes_{k=0}^m (1\otimes 1)^{\otimes j_k}\otimes b^{\otimes i_k})&=\Theta_l(1^{\otimes2j}\otimes b^{\otimes a})\\
&=-\sum_{x\in\{\pm iT^{1/2}\}}\mathrm{Res}_x \frac{\mathrm{str}((\partial_1\Dsf)^{2j}(b\partial_1\Dsf)^{a})\wedge \Omega}{(\partial_1\Wsf)^{2j+a}}\\
&=-\sum_{x\in\{\pm iT^{1/2}\}}\mathrm{Res}_x \bigg( \frac{(1-e^s)^aT^{\frac{j+a}{2}}}{z_1^{2j+a}}\\
&-(1-e^{-s})^a\frac{T^{j/2}}{z_1^{2j}}\bigg)\cdot \frac{z_1^{2(2j+a)-1}}{(z_1^2+T)^{2j+a}}dz_1\\
&\overset{\ref{lm:res1}}=0.
\end{align*}

\end{proof}

\begin{proof}[Proof of Proposition~\ref{prop:bounding cochain}]
    First, we claim that $mc(b)=c\cdot 1.$ Indeed,
    \begin{align*}
        mc(b)&=(e_0+e_1)\cdot b+b\cdot (e_0+e_1)- b^2\\
        &=-(e^{-s}-e^s)\frac{z_0z_1}{T^{1/2}}- (e^{-s}+e^s-2)z_1+ (e^{-s}-e^s)^2\frac{z_1}{4}- (e^{-s}+e^s-2)^2\frac{z_1}{4}\\
        &=(e^s-e^{-s})T^{1/2}\\
        &=c.
    \end{align*}
In addition,
\[
d_{\Rms}(c)=0, \quad \nu_{\Rms}(c)>0, \quad \nu_{\Mms}(b)>0.
\]    
Next, we claim that $\Theta(\expp(b))=s.$ By Lemma~\ref{lm:exp_y_b_compute} we need to prove $\Theta(\exp(b))=s.$
Since
\[
b=\begin{pmatrix}
0&  (1-e^s)z_1\\
1-e^{-s}& 0
\end{pmatrix},\qquad
\partial_1\Dsf=\begin{pmatrix}
0& 1 \\
\frac{T^{1/2}}{z_1^2}& 0
\end{pmatrix},
\]
it follows that
 \[ \big(b\partial_1\Dsf \big)^l= \begin{pmatrix}
        (1-e^s)^l\frac{T^{l/2}}{z_1^l}& 0\\
0& (1-e^{-s})^l
    \end{pmatrix}.\]
Let $l\ge1.$ 
We obtain,
\begin{align*}
    \Theta_l(b^{\otimes l})&=- \sum_{x\in\{\pm iT^{1/2}\}}\mathrm{Res}_x \frac{\mathrm{str}(b\partial_1\Dsf\cdots b\partial_1 \Dsf)\wedge \Omega}{(\partial_1\Wsf)^l}\\
    &=-\sum_{x\in\{\pm iT^{1/2}\}}\mathrm{Res}_x\bigg( (1-e^s)^l\frac{T^{l/2}}{z_1^l}
    - (1-e^{-s})^l\bigg)\cdot \frac{z_1^{2l-1}}{(z_1^2+T)^l}dz_1\\ 
    &\overset{\ref{lm:res1}}= \sum_{x\in\{\pm iT^{1/2}\}}\mathrm{Res}_x \bigg((1-e^{-s})^l\frac{z_1^{2l-1}}{(z_1^2+T)^l}\bigg)dz_1\\
    &\overset{\ref{lm:res2}}= (1-e^{-s})^l.
\end{align*}
Therefore,
\begin{align*}
\Theta(\exp(b))&=\sum_{l=1}^\infty\frac{1}{l}\Theta_l(b^{\otimes l})\\
&= -\sum_{l=1}^\infty \frac{(-1)^{l+1}(e^{-s}-1)^l}{l}\\
&= -\ln(1+ (e^{-s}-1))\\
&=s. 
\end{align*}

\end{proof}
\begin{proof}[Proof of Corollary~\ref{cor:n_d_k}]
By Proposition~\ref{prop:bounding cochain} there exists a bounding cochain $b\in \Mmm(\Ems,c)$ such that $\Theta(\expp(b))=s$ and $c=(e^{-s}-e^s)T^{1/2}.$
Thus, the invariants $N_{d,k}$  are computed directly from Definition~\ref{def:numerical_invariants}.
\end{proof}
Our goal is to prove Theorem~\ref{thm:gw_inv_numerical_invariants}. We have $(X_{\triangle^1}, X_{\triangle^1}^\R)=(\cp^1,\rp^1).$
Recall the open Gromov-Witten invariants with only boundary constraints are defined by 
\[
\oogw^{\rp^1}_{\beta,k}=[T^\beta]\frac{\partial^k\Omega(b)}{\partial s^k}|_{s=0},
\] where $\Omega$ is the superpotential and $b$ is a bounding cochain of $\rp^1$ that satisfies $\int_{\rp^1}b=s.$ 
We refer the reader to~\cite{Point_like_bc} for a full account of setting where open Gromov-Witten invariants can be defined.

\begin{proof}[Proof of Theorem~\ref{thm:gw_inv_numerical_invariants}]
By the top degree axiom, the differential form $b=\frac{s}{2\pi}d\theta$ of $\rp^1$ is a bounding cochain satisfying
$\int_{\rp^1}b=s.$
Consider the superpotential $\Omega$ associated to $b$
\[
\Omega(b)=-\sum_{k\ge1}\frac{1}{k+1}\langle  \mathfrak{m}_k(b^{\otimes k}), b\rangle,
\]
where $\langle\cdot, \cdot\rangle$ is the Poincar\'e pairing defined by \[\langle\alpha_1,\alpha_2 \rangle=(-1)^{|\alpha_2|}\int_{\rp^1}\alpha_1\wedge\alpha_2.\]
Recall $H_2(\cp^1,\rp^1;\Z)\cong\Z\oplus\Z.$ Let $\beta_N, \beta_S$ denote the generators of $H_2(\cp^1,\rp^1;\Z)$ representing the north and south hemispheres respectively.
By the degree axiom, $\oogw_{\beta,k}\ne0$ only if $\beta=\beta_N,\beta_S.$
Hence,
\begin{multline*}
\mathfrak{m}_k(b^{\otimes k})=T^{\beta_N}\frac{s^k}{(2\pi)^k}\int_0^{2\pi}d\theta_k\int_0^{\theta_k}d\theta_{k-1}\cdots \int_0^{\theta_2}d\theta_1\\
+(-1)^kT^{\beta_S}\frac{s^k}{(2\pi)^k}\int_0^{2\pi}d\theta_k\int_0^{\theta_k}d\theta_{k-1}\cdots \int_0^{\theta_2}d\theta_1.
\end{multline*}
Thus, 
\[
\mathfrak{m}_k(b^{\otimes k})= (T^{\beta_N}+(-1)^kT^{\beta_S})\frac{s^k}{k!},
\]
and so, 
\[
\Omega(b)=\sum_{k\ge1}(T^{\beta_N}+(-1)^kT^{\beta_S})\frac{s^{k+1}}{(k+1)!}.
\]
Let $\varphi_{\triangle^n}:H_2(\cp^1,\rp^1;\Z)\to \frac{1}{2}\Z$ defined by $\varphi_{\triangle^n}(\beta_N)=\varphi_{\triangle^n}(\beta_S)=\frac{1}{2}.$ Hence,
\[
\sum_{\tilde{\beta}\in\varphi^{-1}(\frac{1}{2})}\oogw^{\rp^1}_{\tilde{\beta},k}=\oogw^{\rp^1}_{\beta_N,k}+\oogw^{\rp^1}_{\beta_S,k}=\left\{\begin{array}{ll}
        2, &   k\in 2\Z_{\ge0} ,\\
        0, & \mathrm{otherwise},\\
        \end{array} \right.
\]
\[
\sum_{\tilde{\beta}\in\varphi^{-1}(x)}\oogw^{\rp^1}_{\tilde{\beta},k}=0, \quad x\in \frac{1}{2}\Z\backslash\{\frac{1}{2}\}.
\]
Thus, the claim follows from Corollary~\ref{cor:n_d_k}.
\end{proof}

\bibliography{references.bib}

\providecommand{\bysame}{\leavevmode\hbox to3em{\hrulefill}\thinspace}
\providecommand{\MR}{\relax\ifhmode\unskip\space\fi MR }
% \MRhref is called by the amsart/book/proc definition of \MR.
\providecommand{\MRhref}[2]{%
  \href{http://www.ams.org/mathscinet-getitem?mr=#1}{#2}
}
\providecommand{\href}[2]{#2}
\begin{thebibliography}{10}

\bibitem{auroux2007mirror}
D.~Auroux, \emph{Mirror symmetry and {$T$}-duality in the complement of an
  anticanonical divisor}, J. G\"{o}kova Geom. Topol. GGT \textbf{1} (2007),
  51--91.

\bibitem{da_Silva_toric}
A.~Cannas~da Silva, \emph{Symplectic toric manifolds}, Symplectic geometry of
  integrable {H}amiltonian systems ({B}arcelona, 2001), Adv. Courses Math. CRM
  Barcelona, Birkh\"{a}user, Basel, 2003, pp.~85--173.

\bibitem{Cho}
C.-H. Cho, \emph{Counting real {$J$}-holomorphic discs and spheres in dimension
  four and six}, J. Korean Math. Soc. \textbf{45} (2008), no.~5, 1427--1442,
  \href {http://dx.doi.org/10.4134/JKMS.2008.45.5.1427}
  {\path{doi:10.4134/JKMS.2008.45.5.1427}}.

\bibitem{Delzant}
T.~Delzant, \emph{Hamiltoniens p\'eriodiques et images convexes de
  l'application moment}, Bull. Soc. Math. France \textbf{116} (1988), no.~3,
  315--339.

\bibitem{eisenbud}
D.~Eisenbud, \emph{Homological algebra on a complete intersection, with an
  application to group representations}, Trans. Amer. Math. Soc. \textbf{260}
  (1980), no.~1, 35--64, \href {http://dx.doi.org/10.2307/1999875}
  {\path{doi:10.2307/1999875}}.

\bibitem{FOOO1}
K.~Fukaya, Y.-G. Oh, H.~Ohta, and K.~Ono, \emph{Lagrangian intersection {F}loer
  theory: anomaly and obstruction. {P}arts {I} and {II}}, AMS/IP Studies in
  Advanced Mathematics, vol.~46, American Mathematical Society, Providence, RI;
  International Press, Somerville, MA, 2009, \href
  {http://dx.doi.org/10.1090/crmp/049/07} {\path{doi:10.1090/crmp/049/07}}.

\bibitem{Georgieva}
P.~Georgieva, \emph{Open {G}romov-{W}itten disk invariants in the presence of
  an anti-symplectic involution}, Adv. Math. \textbf{301} (2016), 116--160,
  \href {http://dx.doi.org/10.1016/j.aim.2016.06.009}
  {\path{doi:10.1016/j.aim.2016.06.009}}.

\bibitem{pavel}
P.~Giterman, \emph{Generalized superpotential}, to appear.

\bibitem{Hori_Vafa}
K.~Hori and C.~Vafa, \emph{{Mirror symmetry}},  (2000), \href
  {http://arxiv.org/abs/hep-th/0002222} {\path{arXiv:hep-th/0002222}}.

\bibitem{Kapustin_2003}
A.~Kapustin and Y.~Li, \emph{D-branes in {L}andau-{G}inzburg models and
  algebraic geometry}, Journal of High Energy Physics \textbf{2003} (2003),
  no.~12, 005–005, \href {http://dx.doi.org/10.1088/1126-6708/2003/12/005}
  {\path{doi:10.1088/1126-6708/2003/12/005}}.

\bibitem{kapustin2003topological}
A.~Kapustin and Y.~Li, \emph{{Topological correlators in {L}andau-{G}inzburg
  models with boundaries}}, Adv. Theor. Math. Phys. \textbf{7} (2003), no.~4,
  727--749, \href {http://arxiv.org/abs/hep-th/0305136}
  {\path{arXiv:hep-th/0305136}}, \href
  {http://dx.doi.org/10.4310/ATMP.2003.v7.n4.a5}
  {\path{doi:10.4310/ATMP.2003.v7.n4.a5}}.

\bibitem{KS}
M.~Kontsevich and Y.~Soibelman, \emph{Notes on {$A_\infty$}-algebras,
  {$A_\infty$}-categories and non-commutative geometry}, Homological mirror
  symmetry, Lecture Notes in Phys., vol. 757, Springer, Berlin, 2009,
  pp.~153--219.

\bibitem{kontsevich1994homological}
M.~Kontsevich, \emph{Homological algebra of mirror symmetry}, Proceedings of
  the {I}nternational {C}ongress of {M}athematicians, {V}ol. 1, 2
  ({Z}\"{u}rich, 1994), Birkh\"{a}user, Basel, 1995, pp.~120--139.

\bibitem{quadratic_forms}
T.~Y. Lam, \emph{Introduction to quadratic forms over fields}, Graduate Studies
  in Mathematics, vol.~67, American Mathematical Society, Providence, RI, 2005,
  \href {http://dx.doi.org/10.1090/gsm/067} {\path{doi:10.1090/gsm/067}}.

\bibitem{Loday}
J.-L. Loday, \emph{Cyclic homology}, second ed., Grundlehren der mathematischen
  Wissenschaften [Fundamental Principles of Mathematical Sciences], vol. 301,
  Springer-Verlag, Berlin, 1998, Appendix E by Mar\'ia O. Ronco, Chapter 13 by
  the author in collaboration with Teimuraz Pirashvili, \href
  {http://dx.doi.org/10.1007/978-3-662-11389-9}
  {\path{doi:10.1007/978-3-662-11389-9}}.

\bibitem{Morgan}
J.~W. Morgan, \emph{The {S}eiberg-{W}itten equations and applications to the
  topology of smooth four-manifolds}, Mathematical Notes, vol.~44, Princeton
  University Press, Princeton, NJ, 1996.

\bibitem{Morrison_Walcher}
D.~R. Morrison and J.~Walcher, \emph{D-branes and normal functions}, arXiv:
  High Energy Physics - Theory (2007).

\bibitem{orlov2004triangulated}
D.~O. Orlov, \emph{Triangulated categories of singularities and {D}-branes in
  {L}andau-{G}inzburg models}, Tr. Mat. Inst. Steklova \textbf{246} (2004),
  240--262.

\bibitem{PSW}
R.~Pandharipande, J.~Solomon, and J.~Walcher, \emph{Disk enumeration on the
  quintic 3-fold}, J. Amer. Math. Soc. \textbf{21} (2008), no.~4, 1169--1209,
  \href {http://dx.doi.org/10.1090/S0894-0347-08-00597-3}
  {\path{doi:10.1090/S0894-0347-08-00597-3}}.

\bibitem{Shklyarov}
D.~Shklyarov, \emph{Calabi-{Y}au structures on categories of matrix
  factorizations}, J. Geom. Phys. \textbf{119} (2017), 193--207, \href
  {http://dx.doi.org/10.1016/j.geomphys.2017.05.006}
  {\path{doi:10.1016/j.geomphys.2017.05.006}}.

\bibitem{Solomon}
J.~P. Solomon, \emph{Intersection theory on the moduli space of holomorphic
  curves with lagrangian boundary conditions}, 2006, \href
  {http://arxiv.org/abs/math/0606429} {\path{arXiv:math/0606429}}.

\bibitem{Point_like_bc}
J.~P. Solomon and S.~B. Tukachinsky, \emph{Point-like bounding chains in open
  {G}romov-{W}itten theory}, Geom. Funct. Anal. \textbf{31} (2021), no.~5,
  1245--1320, \href {http://dx.doi.org/10.1007/s00039-021-00583-3}
  {\path{doi:10.1007/s00039-021-00583-3}}.

\bibitem{relative_quantum_co}
J.~P. Solomon and S.~B. Tukachinsky, \emph{Relative quantum cohomology},
  Journal of the European Mathematical Society (2023), \href
  {http://dx.doi.org/10.4171/jems/1337} {\path{doi:10.4171/jems/1337}}.

\bibitem{Walcher}
J.~Walcher, \emph{Opening mirror symmetry on the quintic}, Comm. Math. Phys.
  \textbf{276} (2007), no.~3, 671--689, \href
  {http://dx.doi.org/10.1007/s00220-007-0354-8}
  {\path{doi:10.1007/s00220-007-0354-8}}.

\bibitem{Welschinger2}
J.-Y. Welschinger, \emph{Invariants of real symplectic 4-manifolds and lower
  bounds in real enumerative geometry}, Invent. Math. \textbf{162} (2005),
  no.~1, 195--234, \href {http://dx.doi.org/10.1007/s00222-005-0445-0}
  {\path{doi:10.1007/s00222-005-0445-0}}.

\bibitem{Welschinger}
J.-Y. Welschinger, \emph{Spinor states of real rational curves in real
  algebraic convex 3-manifolds and enumerative invariants}, Duke Math. J.
  \textbf{127} (2005), no.~1, 89--121, \href
  {http://dx.doi.org/10.1215/S0012-7094-04-12713-7}
  {\path{doi:10.1215/S0012-7094-04-12713-7}}.

\end{thebibliography}
\bibliographystyle{amsabbrvcnobysame}

\end{document}